\documentclass[10pt,a4paper]{article}
\usepackage[utf8]{inputenc}
\usepackage[colorlinks=true]{hyperref}
\usepackage{amsmath, amsthm}
\usepackage{amsfonts}
\usepackage{times}
\usepackage{amssymb}
\usepackage{enumitem}
\usepackage{fixltx2e}
\usepackage{mathtools}
\usepackage[hyperref, dvipsnames]{xcolor}
\hypersetup{
  colorlinks=true,
}

\usepackage{enumitem}

\usepackage[left=2cm,right=2cm,top=2cm,bottom=2cm]{geometry}
\usepackage{comment}
\newcommand{\grad}{\nabla}
\newcommand{\cl}[1]{\overline{#1}}

\newcommand{\norm}[2]{\left\lVert #1\right\rVert_{#2}}
\newcommand{\weaklyto}{\rightharpoonup}

\newcommand{\Id}{\mathrm{I}}

\newcommand{\cts}{\hookrightarrow}
\newcommand{\ctsCompact}{\xhookrightarrow{c}}
\newcommand{\ctsDense}{\xhookrightarrow{d}}

\newtheorem{theorem}{Theorem}[section]
\newtheorem{prop}[theorem]{Proposition}
\newtheorem{lem}[theorem]{Lemma}

\newtheorem{example}[theorem]{Example}
\newtheorem{remark}[theorem]{Remark}
\newtheorem{ass}[theorem]{Assumption}
\theoremstyle{definition}

\begin{document}
\hypersetup{
  urlcolor     = blue, 
  linkcolor    = Bittersweet, 
  citecolor   = Cerulean
}
\title{Optimal control and directional differentiability for elliptic quasi-variational inequalities\footnotetext{The authors extend their gratitude to the two referees for their careful reading and excellent comments which helped to greatly improve  some of the results and presentation. AA and MH were partially supported by the DFG through the DFG SPP 1962 Priority Programme \emph{Non-smooth and Complementarity-based
Distributed Parameter Systems: Simulation and Hierarchical Optimization} within project 10. 
MH and CNR acknowledge the support of Germany's Excellence Strategy - The Berlin Mathematics Research Center MATH+ (EXC-2046/1, project ID: 390685689) within project AA4-3.
In addition, MH acknowledges the support of SFB-TRR154 within subproject B02, and CNR was supported by NSF grant DMS-2012391.}}
\author{Amal Alphonse\thanks{Weierstrass Institute, Mohrenstrasse 39, 10117 Berlin, Germany ({\tt alphonse@wias-berlin.de})} \and Michael Hinterm\"{u}ller\thanks{Weierstrass Institute, Mohrenstrasse 39, 10117 Berlin, Germany ({\tt hintermueller@wias-berlin.de})}  \and Carlos N. Rautenberg\thanks{Department of Mathematical Sciences and the Center for Mathematics and Artificial Intelligence (CMAI), George Mason University, Fairfax, VA 22030, USA ({\tt crautenb@gmu.edu})}}

\maketitle
\begin{abstract}
We focus on elliptic quasi-variational inequalities (QVIs) of obstacle type and prove a number of results on the existence of solutions, directional differentiability and optimal control of such QVIs. We give three existence theorems based on an order approach, an iteration scheme and a sequential regularisation through partial differential equations. We show that the solution map taking the source term into the set of solutions of the QVI is directionally differentiable for general data and locally Hadamard differentiable obstacle mappings, thereby extending in particular the results of our previous work which provided the first differentiability result for QVIs in infinite dimensions. Optimal control problems with QVI constraints are also considered and we derive various forms of stationarity conditions for control problems, thus supplying among the first such results in this area.
\end{abstract}

\tableofcontents

\section{Introduction}
Quasi-variational inequalities (QVIs) are generalisations of variational inequalities (VIs) where the constraint set in which the solution is sought depends on the unknown solution itself. The very nature of the dependency of the constraint set on the solution intrinsically leads to a complicated and challenging mathematical structure since it significantly amplifies the nonlinear and nonsmooth nature of VIs. Another attribute that fundamentally distinguishes QVIs from VIs is the lack of uniqueness of solutions (in general) which then necessitates the consideration of multi-valued or set-valued solution mappings.  QVIs arise in a multitude of models describing phenomena in fields such as biology, physics, economics and social sciences amongst others. First introduced by Bensoussan and Lions \cite{BensoussanLionsArticle, Lions1973} in the study of stochastic impulse controls, specific applications involving QVIs are thermoforming processes \cite{AHR, ARJF}, the formation and growth of lakes, rivers and sandpiles \cite{PrigozhinSandpile,BarrettPrigozhinSandpile,Prigozhin1996,Prigozhin1994, MR3231973}, games in the context of generalised Nash equilibrium problems \cite{HARKER199181, Facchinei2007, Pang2005}, and magnetisation of superconductors \cite{KunzeRodrigues,BarrettPrigozhinSuperconductivity,Prigozhin,MR1765540}. See \cite{AHRTrends, Baiocchi} for additional details and references.

In this paper, we focus on elliptic QVIs of obstacle type or \emph{compliant obstacle problems}. These have the form 
\begin{equation}\label{eq:QVI}
\text{find } y \in \mathbf{K}(y) : \langle Ay-f, y-v \rangle \leq 0 \quad \forall v \in \mathbf{K}(y)\text{ where $\mathbf{K}(y):=\{v \in V : v \leq \Phi(y)\}.$}%\tag{P\textsubscript{QVI}}
\end{equation}
Here $f \in V^*$ is data, $\Phi\colon V \to V$ is a given \emph{obstacle map}, and $V$ is a reflexive Banach space possessing an ordering $\leq$ which is used in the definition of the constraint set (we shall be more precise below). Let us define $\mathbf{Q}$ to be the solution map associated to the QVI in \eqref{eq:QVI} so that it reads $y \in \mathbf{Q}(f)$. We develop in this paper theory addressing the matters of existence for \eqref{eq:QVI}, directional differentiability of $\mathbf{Q}$ and stationarity conditions for optimal control problems with QVI constraints of the form
\begin{equation}\label{eq:ocProblem}
\min_{\substack{u \in U_{ad}\\ y \in \mathbf Q(u)}}\frac 12\norm{y-y_d}{H}^2 + \frac{\nu}{2}\norm{u}{U}^2.%\tag{{P\textsubscript{OC}}}
\end{equation}
Different methodologies exist for the mathematical treatment of existence for QVIs. There is an approach based on order that was pioneered by Tartar \cite{tartar1974inequations} which relies on the existence of subsolutions and supersolutions to guarantee existence of solutions (typically, one takes $0$ as a subsolution which would hold under sign conditions on the source term). In certain cases, the QVI can be expressed as a generalized equation and it therefore belongs to a more general problem class \cite{MR2571488,MR2997552, MR3238849,MR3307334,MR3299010}. In problems involving constraints on derivatives (which is not the case under consideration in this paper), special forms of regularisation of the constraint that modify the partial differential operator may be suitable, see  \cite{MR1765540, MR2861829,MR2979609,Azevedo:2014ij}. For more details, we refer the reader to the survey paper \cite{AHRTrends}. We discuss in \S \ref{sec:existence} appropriate conditions on the function spaces and the obstacle map $\Phi$ for $\mathbf{Q}(f)$ to be non-empty. One approach relies on an iteration argument where a contraction-type property of $\Phi$ is used. Another existence result is given for source terms bounded from below by using the aforementioned Birkhoff--Tartar theory, and we also study a sequential regularisation approach of the QVI by PDEs where the QVI constraint is handled by a penalty term.

Literature on the differentiability and sensitivity analysis for solution maps associated to QVIs in infinite dimensions is almost non-existent: our contributions \cite{AHR, AHRStability} appear to be the first ones that address these issues. In \cite{AHR}, we give a first directional differentiability result for the solution map taking the source term into the set of solutions for non-negative sources and directions whilst in \cite{AHRStability} we studied continuity properties related to minimal and maximal solution mappings of QVIs. In \S \ref{sec:Diff}, we derive directional differentiablity results for $\mathbf{Q}$.  We extend and improve here  our previous work \cite{AHR} which provided differentiability results for source and direction terms that are non-negative; in this paper we shall remove this restriction in our Theorem \ref{thm:dirDiff1}, which requires minimal (and locally formulated) assumptions to apply. We give a characterisations of the QVI that is satisfied by the directional derivative of $\mathbf{Q}$ as a complementarity system and in \S \ref{sec:continuityOfDirDer} we also prove a continuity result that shows that the derivative depends continuously on the direction under some assumptions.  This gives a comprehensive answer to the question of sensitivity analysis of QVIs under rather general conditions.

%This section becomes of fundamental use later when we look at strong stationarity conditions.

The scarcity of work done on the optimal control of QVIs in infinite dimensions is unsurprisingly even more pronounced; see \cite{MR2657939, AHRStability, MR1770033, MR1816158, MR3012911} for some of the very few contributions. In our work \cite{AHRStability}, in addition to stability properties we also provided results on the optimal control of minimal and maximal solutions of QVIs. While this article was under preparation, we note that \cite{MR4083198} has appeared wherein the author considers elliptic QVIs and their differential sensitivity and strong stationarity conditions for the optimal control problem but for Fr\`echet differentiable obstacle maps $\Phi$; we assume only Hadamard differentiability of $\Phi$ for the differentiability result and we furthermore provide other forms of stationarity as well as existence/approximation results. For QVIs in the finite dimensional setting, see \cite{MR2338444} and the references therein. In sharp contrast, control problems with VI constraints have attracted wide attention: see for example \cite{MR742624, MR739836, MR1477352, MR1455432, MR2869508, MR3232626, MR2822818, MR3056408, MR2954637, MR3796767, Wachsmuth} and the references therein.  We shall consider in \S \ref{sec:existenceOC} the optimal control problem \eqref{eq:ocProblem} where existence of the optimal control will be shown using a standard calculus of variations argument. Then we turn our attention to the derivation of stationarity conditions for the optimal control and state. There are a number of concepts of stationarity for these types of control problems, see \cite{MR3232626} for a discussion. We first work on obtaining  Bouligand stationarity in \S \ref{sec:BS}, then a form of weak C-stationarity in \S \ref{sec:weakStationarity}, moving on to $\mathcal{E}$-almost C-stationarity conditions %\cite{MR2822818, MR2515801}
 in \S \ref{sec:penalisationOfQVI} by approximating the QVI control-to-state map through PDEs (as done in \S \ref{sec:penalisationForExistence}) and then passing to the limit. We discuss in \S \ref{sec:epsAlmostToC} how to upgrade to C-stationarity from $\mathcal{E}$-almost C-stationarity and finally, in \S \ref{sec:ss}, we provide a strong stationarity result.

\subsection{Contributions of the paper}
We summarise the main results of this work.
\begin{itemize}
\item \textbf{QVI}: 
\begin{itemize}
%\item Theorem \ref{thm:existenceUnsigned}: iteration by solutions of VIs using complete continuity of $\Phi$,
%\item Theorem \ref{lem:existenceTransformedQVI}: Birkhoff--Tartar order approach under lower bounds on $\Phi$ and the source term,
\item Theorems \ref{thm:penalisedConvergence} and \ref{thm:newOne}: existence for \eqref{eq:QVI} via a penalty approach, %under complete continuity of $\Phi$.

\item Theorem \ref{thm:dirDiff1}: directional differentiability for QVIs for locally Hadamard maps $\Phi$ under local Lipschitz conditions,
% \item Theorem \ref{thm:dirDiff2}: for locally Hadamard differentiable maps $\Phi$ for source/direction terms bounded from below.

\item Proposition \ref{lem:complementarityForQVI}: complementarity characterisations of the QVI in \eqref{eq:QVI},
\item Proposition \ref{lem:alphaUniquenessAndCty}: continuity properties %uniqueness 
of the QVI satisfied by directional derivative,
\item Proposition \ref{lem:complementarityCharacterisationDerivativeA}: complementarity characterisation of the QVI satisfied by the directional derivative of the solution map.
\end{itemize}
\item \textbf{Optimal control}:
\begin{itemize}
\item Theorem \ref{thm:existenceOC}: existence of optimal controls for \eqref{eq:ocProblem}.
\end{itemize}
\item\textbf{Stationarity conditions for \eqref{eq:ocProblem}}:
\begin{itemize}
\item Proposition \ref{lem:characterisationOfOC}: Bouligand stationarity,
\item Theorem \ref{thm:weakStationarity}: weak C-stationarity,
\item Theorem \ref{thm:ocPenalisation}: %and \ref{thm:epsAlmostCStationaryStrongerAss}
 $\mathcal{E}$-almost C-stationarity,
 \item Proposition \ref{prop:CStationarity}: C-stationarity,
\item Theorem \ref{thm:strongStationarity}: strong stationarity.
\end{itemize}
\end{itemize}

\subsection{Basic assumptions and notations}\label{sec:basicass}
We make some standing assumptions that are necessary throughout the paper, except where mentioned otherwise. 

%Let $V \subset H$ be a continuous embedding of real separable Hilbert spaces. We denote the inner product on $H$ by $(\cdot,\cdot)$ and 
%Suppose that there exists an ordering to elements of $H$ via a closed convex cone $H_+$ satisfying $H_+ = \{h \in H : (h,g) \geq 0 \quad \forall g \in H_+\}$; the ordering then is $h_1 \leq h_2$ if and only if $h_2-h_1 \in H_+$. This also induces an ordering  for $V$ in the obvious way (and we write $V_+ := \{ v \in V : v \geq 0\}$) and in addition for $V^*$ via
%\[V^*_+ := \{ f \in V^* : \langle f, v\rangle \geq 0 \quad \forall v \in V_+\}.\]
%We write $h^+$ for the orthogonal projection of $h \in H$ onto the space $H_+$ and we have the decomposition $h = h^+ - h^-$. We suppose that $v \in V$ implies that $v^+ \in V$ and that there exists a constant $C >0$ such that for all $v \in V$,
%\[\norm{v^+}{V} \leq C\norm{v}{V}.\]
%An example of such a space $V$ is the Sobolev space $V=H^1(\Omega)$ or $H^1_0(\Omega)$ over a domain $\Omega$ with $H=L^2(\Omega)$ (see \cite{Adams} for a definition). The ordering relation $u \leq v$ in this case is equivalent to `$u \leq v$ a.e. in $\Omega$' as expected.

We always work with real Banach or Hilbert spaces. Let $V$ be a Banach space and denote the standard duality pairing on $V^*\times V$ by $\langle \cdot, \cdot \rangle = \langle \cdot, \cdot \rangle_{V^*, V}$.  %We take the obstacle map $\Phi\colon V \to V$ %to satisfy
%\[\Phi\colon V \to V \text{ is increasing,}\]
Take $A\colon V \to V^*$ to be a linear operator that satisfies the following properties for all $u, v \in V$: %$A$ satisfies
\begin{align*}
\langle Au, v \rangle &\leq C_b\norm{u}{V}\norm{v}{V}\tag{boundedness},\\
\langle Au, u \rangle &\geq C_a\norm{u}{V}^2\tag{coercivity},\\
\langle Au^+, u^- \rangle &\leq 0\tag{T-monotonicity},
%|u| &\in V.
\end{align*}
where $C_a, C_b > 0$ are constants. We will frequently suppose that the Banach space $V$ is a \emph{vector lattice} for a partial ordering $\leq$.  This means that for all $u, v \in V$, the following holds: \vspace{-.2cm}
\begin{enumerate}
\item[(i)] $u\leq u$ (reflexivity), \vspace{-.2cm}
\item[(ii)] $u \leq v$ and $v \leq u$ implies $u=v$ (anti-symmetry), \vspace{-.2cm}
\item[(iii)] $u \leq v$ and $v \leq w$ implies $u \leq w$ (transitivity),\vspace{-.2cm}
\item[(iv)]  $u \leq v$ implies that $u+w \leq v+ w$ and $\lambda u \leq \lambda v$ for $\lambda \geq 0$, \vspace{-.2cm}
\item[(v)]there exists a greatest lower bound $\inf(u,v)$ and a least upper bound $\sup(u,v)$ belonging to $V$.
\end{enumerate}
See for example \cite{MR809372, MR1128093} or
\cite[\S 4:5]{Rodrigues} for more details.  It should be emphasised that in the context of function spaces over a bounded  Lipschitz domain $\Omega$, with $\leq$ chosen as the usual a.e. ordering, (v) allows for $V=L^p(\Omega)$ and $V=W^{1,p}(\Omega)$ for $1 \leq p < \infty$ but not $V=W^{2,p}(\Omega)$ in general.  
We write the positive cone of $V$ as
\[V_+ := \{ v \in V : v \geq 0\}\]
(this is convex but not necessarily closed). If $V$ is a Banach lattice, the projection onto $V_+$ (assuming this is well defined) of an element $v \in V$ agrees with $\sup(0,v)$, but this is not necessarily the case for a general vector lattice. Note that the dual space $V^*$ inherits an ordering: we say $f\leq g$ in $V^*$ if and only if $\langle g-f,v\rangle\geq 0$ for all $v\in V_+$.

Regarding the obstacle map, we take $\Phi\colon V \to V$ to be given.

%Since later $V$ will be assumed to be part of a Gelfand triple with a pivot space $H$, we will only rarely need to use the inner product on $V$ and when we do so this will always be denoted by $(\cdot,\cdot)_V$ (with the subscript). 
The identity operator will be denoted by $\Id$. We denote continuous, dense, and compact embeddings of spaces by $\cts$, $\ctsDense$, and $\ctsCompact$ respectively. The notation $B_R(u)$ will be used to mean the closed ball in $V$ of radius $R$ centred at $u$.
\section{Existence for QVIs}\label{sec:existence}
We begin by discussing three existence results for the QVI in  \eqref{eq:QVI}, reproduced here:
\[ y \leq \Phi( y) : \langle A y -  f,  y - v \rangle \leq 0 \quad \forall v \in V : v \leq \Phi( y),\] involving different approaches. We start by obtaining existence through iteration by solutions of VIs. Then we consider a translation of the theory by Birkhoff--Tartar for source terms that are bounded from below and we finish by considering a sequential regularisation approach through PDEs. These existence results entail different assumptions. The third approach is useful for purposes of numerical realisation. The second approach requires $\Phi$ to be increasing and bounded below in a certain sense.

Before we proceed, let us give the following characterisation involving \eqref{eq:QVI}.
\begin{prop}\label{lem:complementarityForQVI}
The QVI in \eqref{eq:QVI} is equivalent to the complementarity system
\begin{subequations}
\begin{align}
 \xi &:=  f - A y,\\
 \xi & \geq 0, \\
\langle \xi&, \Phi( y) -  y\rangle = 0,\label{eq:subO}\\
0 &\leq \Phi( y) -  y.
\end{align}
\end{subequations}
\end{prop}
\begin{proof}
The proof is standard. By definition, $ \xi$ satisfies $\langle  \xi,  y - v \rangle \geq 0$ for all feasible $v$. Setting $v=\Phi( y)$ and then $v= 2 y - \Phi( y)$, we obtain the orthogonality condition \eqref{eq:subO} for $ \xi$.  Testing with $v=y-\varphi$ for $\varphi \in V$ with $\varphi \geq 0$ gives the stated non-negativity. The reverse direction follows from writing $\langle Ay-f, y-v \rangle = \langle Ay-f, y-\Phi(y)\rangle + \langle Ay-f, \Phi(y)-v\rangle$ (where $v$ is a feasible test function) and using the second and third lines in the system.
\end{proof}
\subsection{Iteration scheme}\label{sec:existenceContraction}
We need the following assumption for this section (as an example, $V=L^2(\Omega)$ or $H^1(\Omega)$ on a bounded Lipschitz domain are valid).
%\begin{ass}[Vector lattice]\label{ass:ordering}
%Let $(V,H,V^*)$ be a Gelfand triple of Hilbert spaces and let $H_+\subset H$ be a closed convex cone satisfying
%\begin{equation*}
%H_+=\{h \in H :  (h,g)\geq 0 \text{ for all } g\in H_+\}.
%\end{equation*}
%This defines a cone of non-negative elements and induces an ordering 
%%\begin{equation*}
%\[h_1\leq h_2 \text{ if and only if } h_2-h_1 \in H_+.\]
%%\end{equation*}
%We write $h^+$ to denote the orthogonal projection of $h \in H$ onto $H_+$ and define $h^-:=h^+-h$. The infimum and supremum of two elements $h_1,h_2\in H$ are defined as $\inf (h_1,h_2):=h_1-(h_1-h_2)^+$ and $\sup (h_1,h_2):=h_1+(h_2-h_1)^+$.
%
%Finally, we assume that $y\in V$ implies $y^+\in V$ and that there exists a $C>0$ with $\|y^+\|_V\leq C\|y\|_V$ for all $y\in V$. 
%\end{ass}
%\begin{ass}[Vector lattice]\label{ass:ordering}
%Let $V_+\subset V$ be a closed convex cone satisfying
%\begin{equation*}
%V_+=\{v \in V :  (v,w)\geq 0 \text{ for all } w\in V_+\}.
%\end{equation*}
%This %defines a cone of non-negative elements and 
%induces an ordering 
%%\begin{equation*}
%\[v_1\leq v_2 \text{ if and only if } v_2-v_1 \in V_+.\]
%%\end{equation*}
%We write $v^+$ to denote the orthogonal projection of $v \in V$ onto $V_+$ and define $v^-:=v^+-v$. The infimum and supremum of two elements $v_1,v_2\in V$ are defined as $\inf (v_1,v_2):=v_1-(v_1-v_2)^+$ and $\sup (v_1,v_2):=v_1+(v_2-v_1)^+$.
%%Finally, we assume that $y\in V$ implies $y^+\in V$ and that there exists a $C>0$ with $\|y^+\|_V\leq C\|y\|_V$ for all $y\in V$. 
%\end{ass}
\begin{ass}\label{ass:vectorLattice}
Let $V$ be a Hilbert space and a vector lattice with $V_+$ closed and suppose that $\Phi\colon V \to V$ is increasing.
\end{ass}
The lattice and increasing properties are necessary to apply the comparison principle for VIs \cite[\S 4:5]{Rodrigues}. This assumption also implies the following useful property (whose proof is in Appendix \ref{app:technical}), which can be thought of as a weak monotone convergence theorem (in fact, it suffices for $V$ to be a reflexive Banach space rather than Hilbert for the result).
\begin{lem}\label{lem:mct}
If $\{v_n\} \subset V$ is a bounded sequence which is either increasing or decreasing (i.e., either $v_n \leq v_{n+1}$ for all $n$, or $v_n \geq v_{n+1}$ for all $n$), then there exists a $v \in V$ such that $v_n \weaklyto v$ in $V$ (for the full sequence).
\end{lem}

%\noindent The assumption that $V$ is a vector lattice means that for all $u, v \in V$, the following holds: \vspace{-.2cm}
%\begin{enumerate}
%\item[(i)] $u\leq u$, \vspace{-.2cm}
%\item[(ii)] $u \leq v$ and $v \leq u$ implies $u=v$, \vspace{-.2cm}
%\item[(iii)] $u \leq v$ and $v \leq w$ implies $u \leq w$,\vspace{-.2cm}
%\item[(iv)]  $u \leq v$ implies that $u+w \leq v+ w$ and $\lambda u \leq \lambda v$ for $\lambda \geq 0$, \vspace{-.2cm}
%\item[(v)]there exists a greatest lower bound $\inf(u,v)$ and a least upper bound $\sup(u,v)$ belonging to $V$.
%\end{enumerate}
%See \cite{ConesAndDuality} for more details.  It should be emphasised that the assumption on $V$ allows for $V=H^1(\Omega)$ but not $V=H^2(\Omega)$. 

%The assumption on $V$ allows for $V=H^1(\Omega)$ but not $V=H^2(\Omega)$. Note that the ordering induces an ordering for $V$ in the obvious way (and we write $V_+ := \{ v \in V : v \geq 0 \}$) and also 
%We write the positive cone as
%\[V_+ := \{ v \in V : v \geq 0\}\]
%(this is convex but not necessarily closed). Note that the dual space $V^*$ inherits an ordering: we say $f\leq g$ in $V^*$ if and only if $\langle g-f,v\rangle\geq 0$ for all $v\in V_+$.}

%\textcolor{blue}{In this section, we suppose that 
%\[\Phi\colon V \to V \text{ is increasing.}\]}
Let $S\colon V^* \times V \to V$ be the solution mapping of the VI associated to the class of QVIs under consideration, i.e. $y=S(f,\psi)$ solves
\[y \leq \Phi(\psi) : \langle Ay-f, y- v\rangle \leq 0 \quad \forall v \in V : v \leq \Phi(\psi).\]
Take a source term $f \in V^*$ and set 
\[y_0 := A^{-1}f\] to be the solution of the unconstrained problem. The function $y_1 := S(f,y_0)$ satisfies\footnote{Heuristically, $y_0$ is considered as a solution of the VI with source $f$ and obstacle equal to $\infty$. %we have $y_1=S(f,y_0) \leq \bar S(f,\Phi^{-1}(\infty)) = A^{-1}f = y_0$.
} $y_1 \leq y_0$ by the comparison principle \cite[\S 4:5, Theorem 5.1]{Rodrigues}, and defining 
\[y_n := S(f, y_{n-1}),\]
we see that $y_n \leq y_{n-1}$ by repeated applications of the comparison principle. Hence $\{y_n\}$ is monotonically decreasing and each $y_n$ satisfies 
\begin{equation}\label{eq:VIforun}
y_n \in V, y_n \leq \Phi(y_{n-1}) : \langle Ay_n -f , y_n -v \rangle \leq 0 \quad \forall v \in V : v \leq \Phi(y_{n-1}).
\end{equation}
We look for a uniform bound on $\{y_n\}$. When the obstacle map is such that it always dominates some given function $v_0 \in V$, this is easy since we may test with $v=v_0$. Otherwise, we need the following.
\begin{lem}If 
\begin{equation}\label{ass:normPhi}
\norm{\Phi(v)}{V} \leq C_X\norm{v}{V} \quad \forall v \in V \text{ where $C_X < \frac{C_a}{C_b}$,}
\end{equation}
then $\{y_n\}$ is bounded in $V$.
\end{lem}
\begin{proof}
Since $y_n \leq y_{n-1}$ and $\Phi$ is increasing, $\Phi(y_n) \leq \Phi(y_{n-1})$ and so $\Phi(y_n)$ is a valid test function in \eqref{eq:VIforun} and we obtain
\begin{align*}
C_a\norm{y_n}{V}^2 &\leq \langle Ay_n, \Phi(y_n) \rangle + \langle f, y_n - \Phi(y_n)\rangle\\
&\leq C_b\norm{y_n}{V}\norm{\Phi(y_{n})}{V} + \norm{f}{V^*}\norm{y_n - \Phi(y_{n})}{V}\\
%&\leq C_b\norm{y_n}{V}\norm{\Phi(y_{n-1})}{V} + \norm{f}{V^*}\norm{y_n}{V} + \norm{f}{V^*}\norm{\Phi(y_{n-1})}{V}\\
&\leq C_bC_X\norm{y_n}{V}^2 + (1+C_X)\norm{f}{V^*}\norm{y_n}{V}.
\end{align*}
From this, we deduce that $y_n$ is bounded in $V$ under the condition on $C_X$ in \eqref{ass:normPhi}. %See the proof of Lemma 6 in \cite{AHR} for more details.
\end{proof}
%The assumption \eqref{ass:normPhi} places a limitation on the variation on the bound of the constraint map $\Phi$ which implies uniqueness of solutions for \eqref{eq:QVI}. %This is unfortunately the price to pay for a general theory without conditions on $f$, unless for example (as mentioned) a sign condition on $\Phi$ holds. 

Now we pass to the limit and show that $\mathbf{Q}\colon V^* \rightrightarrows V$ is such that $\mathbf{Q}(f) \neq \emptyset$ under certain circumstances. 
\begin{theorem}\label{thm:existenceUnsigned}
Let Assumption \ref{ass:vectorLattice} hold and %let $\Phi\colon V \to V$ be increasing and 
suppose that
\begin{align}
&\text{either there exists $v_0 \in V$ such that  $v_0 \leq \Phi(v)$ for all $v \in V$, or \eqref{ass:normPhi}},\label{eq:assH1}\\%\tag{H1a}\\	
&\text{if $\{v_n\} \subset V$ is decreasing with $v_n \weaklyto v$ in $V$ and $v \leq \Phi(v_n)$, then $v \leq \Phi(v).$}\label{ass:PhipresOrder}
%&\text{$\Phi\colon V \to V$ is weakly sequentially continuous}.\label{ass:PhiWeakCts}
\end{align}
For any $f \in V^*$, there exists a solution $y \in \mathbf{Q}(f) \cap (-\infty, A^{-1}f]$ which is the weak limit of the sequence $\{y_n\}$ defined above.
\end{theorem}
\begin{proof}
We obtain, thanks to monotonicity and Lemma \ref{lem:mct} that $y_n \weaklyto y$ in $V$ (for the full sequence) for some $y$.  Since $\{y_n\}$ is decreasing, %$y_{n-1}-y_n \in V_+$ for all $n$ and hence 
$y_m - y_n \in V_+$ where $n \geq m$ for $m$ fixed. As $V_+$ is closed and convex, it is weakly sequentially closed, giving $y_m \geq y$. This implies that for arbitrary $v^* \in V$ with $v^* \leq \Phi(y)$, we have $v^* \leq \Phi(y_m)$. We take such a $v^*$ as the test function in the VI for $y_m$ and then pass to the limit to obtain that $y$ satisfies the inequality in \eqref{eq:QVI} and it remains to be seen that $y \leq \Phi(y)$. This follows from passing to the limit in $y \leq y_m \leq \Phi(y_{m-1})$ by making use of \eqref{ass:PhipresOrder}.
%
%
%and taking as test function $v_n = v^*-\Phi(y) + \Phi(y_{n-1})$, which is feasible for the VI for $y_n$ and strongly converges to $v^*$, we can easily pass to the limit in \eqref{eq:VIforun} and we find $y \in \mathbf{Q}(f)$ in the stated interval.
\end{proof}
The assumption \eqref{ass:PhipresOrder} is rather weak and it is satisfied if, for example, $\Phi\colon V \to V$ is weakly sequentially continuous.

\begin{remark}
For QVIs with more general or different types of constraints one might need to assume Mosco convergence (see \cite[\S 4:4]{Rodrigues}) properties of the underlying constraint sets. %Given the structure of our unilaterally constrained compliant obstacle problem, this Mosco convergence is obtained in part due to the complete continuity assumption on $\Phi$.
\end{remark}

\begin{example}
The prototypical example for $\Phi$ to have in mind is a map given by the inverse of a partial differential operator such as
\[\Phi(w) := L^{-1}w + f_0,\]
for example with $L\colon V \to V^*$ a second-order linear elliptic operator on a bounded Lipschitz domain $\Omega$ and $f_0 \in V$. The validity of elliptic regularity and continuous dependence estimates for $L$ would give compactness properties for $\Phi$ (and weak maximum principles would also yield the increasing property for $\Phi$). See \cite[\S 1.2]{AHR} for more details on this and on an application to fluid flow.
\end{example}
%
%\begin{remark}\label{rem:AHR}
%If the source term $f$ is non-negative (i.e. if $f \in V^*_+$), $\Phi\colon H \to V$ \textcolor{blue}{is increasing}, and $\Phi(0) \geq 0$, then the function $0$ acts a subsolution for the map $S(f,\cdot)$ which, in combination with the supersolution $y_0$ defined as above, allows us to directly apply the theory of fixed points in vector lattices of Tartar--Birkhoff and obtain existence of solutions for \eqref{eq:QVI} in the interval $[0, y_0]$. In this case one does not need the assumption \eqref{eq:assH1}. %This was the approach taken in \cite{AHR}.
%\end{remark}

\subsection{Birkhoff--Tartar order approach}\label{sec:existenceTranslation}
%Let $G \in H^1_0(\Omega)$ with $G > 0$ q.e. on $\Omega$. 
%We return to the issue of existence for \eqref{eq:QVI}. 
In this section, we extend Birkhoff--Tartar-type existence results typically used for QVIs with non-negative source terms to QVIs with source terms that are allowed to be negative. This leads to different assumptions than those made in \S \ref{sec:existenceContraction}. The bedrock of this technique, as detailed in the introduction, is the result of Tartar \cite{tartar1974inequations} that gives existence of fixed points for increasing maps that possess subsolutions and supersolutions, see also \cite[Chapter 15, \S 15.2]{MR556865}. We need the following functional setup in this section.
\begin{ass}\label{ass:forTartar}
Let $V \ctsDense H$ be a continuous and dense embedding of Hilbert spaces and let $C\subset H$ be a closed convex cone satisfying
\begin{equation}
C=\{h \in H :  (h,g)_H \geq 0 \text{ for all } g\in C\}.\label{eq:cone}
\end{equation}
This %defines a cone of non-negative elements and 
induces an ordering defined by
%\begin{equation*}
\[h_1\leq h_2 \text{ if and only if } h_2-h_1 \in H_+.\]
Note that $H_+ \equiv C$. We write $h^+=P_{H_+}h$ to denote the orthogonal projection of $h \in H$ onto $H_+$ and define $h^-:=h^+-h$. %The infimum and supremum of two elements $h_1,h_2\in H$ are defined as $\inf (h_1,h_2):=h_1-(h_1-h_2)^+$ and $\sup (h_1,h_2):=h_1+(h_2-h_1)^+$.
We assume that $v \in V$ implies $v^+\in V$ and that there exists a $C>0$ with $\|v^+\|_V\leq C\|v\|_V$ for all $v\in V$. Finally, suppose that
\[\Phi\colon H \to V \text{ is increasing.}\]
\end{ass}
%\end{equation*}
Note that $H$ is a vector lattice (induced by $C$) and that the ordering induces an ordering for $V$ in the obvious way and also an ordering for $V^*$ as elucidated in \S \ref{sec:basicass}. Also, $-h^- \in P_{-H_+}h$ because $C$ satisfies \eqref{eq:cone}.

Let us recall the Birkhoff--Tartar result (see \cite[\S 15.2, Proposition 2]{MR556865}) for increasing maps under the assumptions on the function spaces in Assumption \ref{ass:forTartar}.
\begin{theorem}[Birkhoff--Tartar]\label{thm:BirkhoffTartar} 
Suppose that $T\colon H\rightarrow H$ is an increasing map and let $\underline{h}$ be a subsolution and $\overline{h}$ be a supersolution of the map $T$, i.e.,% That is:
\begin{equation*}
\underline{h}\leq T(\underline{h}) \quad \text{ and } \quad T(\overline{h})\leq \overline{h}.
\end{equation*}
If $\underline{h}\leq  \overline{h}$, then the set of fixed points of $T$ in the interval $[\underline{h}, \overline{h}]$ is non-empty and has a minimal and a maximal element.
\end{theorem}
With this at hand, we can study existence for \eqref{eq:QVI}.
\begin{theorem}\label{lem:existenceTransformedQVI}
Let % $\Phi\colon H \to V$ be increasing and let 
 Assumption \ref{ass:forTartar} hold and suppose that
 \begin{equation}\label{ass:phiTranslationAtR}
\text{there exists $v_0 \in V$ such that $v_0 \leq \Phi(v_0).$}
\end{equation} Given $f \in V^*$ with $Av_0 \leq f \leq F$ for some $F \in V^*$, there exist solutions $y \in \mathbf{Q}(f)\cap[v_0, A^{-1}F]$. Furthermore, there exists a minimal and a maximal solution on this interval.
\end{theorem}
\begin{proof}
By the comparison principle, $S(f,v_0) \geq S(Av_0,v_0) = v_0$, hence $v_0$ is a subsolution for $S(f, \cdot)$. Since $\Phi$ is increasing, $A^{-1}F =  S(F, \Phi^{-1}(\infty))\footnote{By $S(F, \Phi^{-1}(\infty))$ we simply mean the solution of the unconstrained problem with source $F$. No invertibility of $\Phi$ is necessary.} \geq S(F, A^{-1}F) \geq S(f,A^{-1}F)$ so that $A^{-1}F$ is a supersolution. We also have $v_0 = S(Av_0,v_0) \leq S(F, v_0) \leq S(F,\Phi^{-1}(\infty)) = A^{-1}F$, i.e., the subsolution lies below the supersolution.  Finally, $S(f,\cdot)$ is increasing due to $\Phi$ being increasing.  The result follows from the Birkhoff--Tartar theorem.
\end{proof}
A typical situation in examples is when $\Phi(0) \geq 0$ and $f \in V^*_+$. While the assumption \eqref{ass:phiTranslationAtR} of the existence of such a $v_0$ may appear to be restrictive, note that choosing $v_0 \equiv 0$ recovers the setting of \cite{AHR} which has been successfully applied to an application in thermoforming. The next example illustrates the existence of such a function $v_0$ to a map $\Phi$ related to solution maps of elliptic PDEs.
\begin{example}
Let $\Omega \subset \mathbb{R}^n$ be a bounded Lipschitz domain  and set $H:=L^2(\Omega)$. Suppose $(V,H,V^*)$ is a Gelfand triple\footnote{Recall that $V \subset H \equiv H^* \subset V^*$ is called a \emph{Gelfand triple} if $V$ is a reflexive Banach space continuously and densely embedded into the Hilbert space $H$ and $H$ has been identified with its dual through the Riesz map.} with $V$ a reflexive Banach space. Given a linear, bounded, coercive and T-monotone operator $B\colon V \to V^*$ and a source term $g \in V^*$, let $\Phi(u) = \varphi$ be defined\footnote{The interest in such obstacle mappings is not merely academic, see \cite{AHR} for some applications.} as the solution of 
\[B\varphi = g + u.\]
Take any $v_0 \in V$. We claim that if $g$ is such that
\[g \geq Bv_0-v_0,\]
in $V^*$, then \eqref{ass:phiTranslationAtR} is satisfied. To see this, set $v:=\Phi(v_0)$ so that $B                                                                                                                                                                                                                                                                                                                                                                                                                                                                                                                                                                                                                                                                                                                                                                                                                                                                                                                                                                                                                                                                                                                                                                                                                                                                                                                                                                                                                                                                                                                                                                                                                                                                                                                                                                                                                                                                                                                                                       v = g +v_0$. Adding the same term to both sides, we obtain $B(v-v_0) = g +v_0-Bv_0$. Test this with the function $(v-v_0)^-$ to obtain
\[\langle B(v-v_0)^-, (v-v_0)^-\rangle \leq -\langle g +v_0-Bv_0, (v-v_0)^-\rangle \leq 0.\]
%and the right-hand side of this is non-positive. % (recall $f \geq 0$). 
\end{example}

\subsection{Sequential regularisation by PDEs}\label{sec:penalisationForExistence}
In this section, we obtain existence results for \eqref{eq:QVI} by regularising the QVI by PDEs by a penalty approach similar to \cite[\S 3.5.2,  p.~370]{Lions1969}. There has been considerable effort on various aspects and methods of regularisation of VIs by PDEs; see for example \cite[\S 3.2]{Glowinski} for an approach similar to what we consider here and \cite{MR0247551} and \cite[\S IV]{Kinderlehrer} for a penalisation involving approximations to the Heaviside graph (see also \cite[\S 5:3]{Rodrigues} on this).

\begin{ass}\label{ass:seqReg}
Let $V$ be a %strictly convex, 
reflexive Banach space and a vector lattice such that $V_+$ is closed. %Furthermore, let also $V^*$ be strictly convex.
\end{ass}

Recall that an operator $T\colon X \to X^*$ is \emph{hemicontinuous} \cite[Definition 2.3]{Roubicek} if $s \mapsto \langle T(x+sy), z \rangle_{X^*,X}$ is continuous for all $x, y, z \in X$. For each $\rho > 0,$ let $m_\rho\colon V \to V^*$ be a hemicontinuous map  such that
\begin{align}
&m_\rho(v) = 0 \text{ if $v \leq 0$}\label{ass:mrPenal}\\
%&m_\rho \colon V \to V^* \text{ is bounded}\label{ass:mrBd}\\
&\langle m_\rho(u)-m_\rho(v), u-v\rangle \geq 0\label{ass:mrMonotonicity}\\
%&m_\rho \geq 0\label{ass:mrNonNeg}\\
&z_\rho \weaklyto z \text{ in $V$ and } m_\rho(z_\rho) \to 0 \text{ in $V^*$ (as $\rho \to 0$)} \implies z \leq 0\label{ass:mrForFeas}
\end{align}
\begin{remark}
The last condition precludes the possibility of having `bad' choices of $m_\rho$ such as $\rho(\cdot)^+$. It is also worth pointing out that if $m_\rho \equiv m$ for some map $m$, then \eqref{ass:mrForFeas} implies that 
\[m(z) = 0 \implies z \leq 0,\] which is the converse of \eqref{ass:mrPenal}, so \eqref{ass:mrForFeas} can be thought of a strengthening of the classical kernel or penalty condition that one finds in penalty approaches for VIs.
\end{remark}
It is always possible to find such a sequence of maps $\{m_\rho\}$, see the next example as well as Example \ref{eg:gelfandTripleEg} for the Gelfand triple case. For this reason, we will not usually explicitly refer to \eqref{ass:mrPenal}--\eqref{ass:mrForFeas} in statements of theorems.
\begin{example}[Existence of $m_\rho$]\label{eg:generalCase}
Let $V$ and $V^*$ be strictly convex\footnote{All Hilbert spaces (and thus their duals) are strictly convex. In fact, the strict convexity requirement in the assumption is no issue in the setting of reflexive Banach spaces: by Asplund's theorem (see e.g., \cite[\S 2.2.2, Theorem 2.5]{Lions1969}), $V$ can be renormed via an equivalent norm making $V$ and $V^*$ strictly convex.}.  Indeed, with $\mathcal J\colon V \to V^*$ denoting the duality mapping\footnote{The assumption of strict convexity gives appropriate properties of $\mathcal J$ (such as single-valuedness), see \cite[\S 2.2.2, p.~174]{Lions1969} and \cite[\S 32.3d]{Zeidler2B} for more details.} the choice 
\[m_\rho(u) := \mathcal J(u-P_{V_-}(u))\] 
furnishes such an example where $P_{V_-}\colon V \to V_-$ is the metric projection\footnote{This is well defined since we assumed $V_+$ (and hence $V_-$) is closed and because $V$ is a reflexive and strictly convex space.} onto the set of non-positive elements $V_-$. %(note that $P_{V_-}$ can be characterised in terms of $\mathcal J$).  
Properties \eqref{ass:mrPenal} and \eqref{ass:mrMonotonicity} as well as hemicontinuity follow as in \cite[\S 3.5.2, Theorem 5.1, p.~370]{Lions1969}. In fact, note that $m_\rho(u) =0$ implies that $u \leq 0$ (because $\mathcal J$ is bijective and passes through the origin \cite[Proposition 32.22 (a), (b)]{Zeidler2B}).  For \eqref{ass:mrForFeas}, denoting $m_\rho \equiv m$, by monotonicity, we have for every $\lambda >0$ and $v \in V$ that
\begin{align*}
\langle m(z_\rho) - m(z+\lambda v), z_\rho - z - \lambda v \rangle \geq 0
\end{align*}
whence passing to the limit $\rho \to 0$, using $m(z_\rho) \to 0$ in $V^*$ (by hypothesis),
%\begin{align*}
$\langle m(z+\lambda v),  \lambda v \rangle \geq 0$
%\end{align*}
and then dividing through by $\lambda$ and sending $\lambda \to 0$, by hemicontinuity, we obtain that  $m(z) = 0$ in $V^*$ and thus $z \leq 0$.
%\end{example}
%\begin{example}
%If $V$ happens to be a Hilbert space, setting $v^+ = P_{V_+}v$ to be the orthogonal projection, %of $v$ onto $V_+$, 
%we can take $m_\rho(v) := (v^+, \cdot)_V$. In this case, we also have that $m_\rho(v) = 0$ if and only if $v \leq 0$.
\end{example}

%Another example can be found in \S \ref{sec:penalisationOfQVI}.

%Now, taking $\Phi \colon V \to V$, 
We consider the penalisation\footnote{For the results of this section, it would be sufficient to simply consider the case where each $m_\rho \equiv m$, but in anticipation of the optimal control problem that we shall later study (in particular when we derive optimality conditions), it becomes useful to consider this generality now.}
\begin{equation}\label{eq:penalised}
Ay_\rho + \frac{1}{\rho}m_\rho(y_\rho-\Phi(y_\rho)) = f
\end{equation}
of  \eqref{eq:QVI} and study the convergence properties of its solution as $\rho \to 0$. First, we discuss existence. We recall that a map $T\colon X \to X^*$ is said to be \emph{radially continuous} \cite[Definition 2.3]{Roubicek} if $s \mapsto \langle T(x+sy), y \rangle_{X^*,X}$ is continuous for all $x, y \in X$, and a map $R\colon X \to Y$ between Banach spaces is said  to be \emph{completely continuous} \cite[\S 2]{Showalter}  if $x_n \weaklyto x$ in $X$ implies that $R(x_n) \to R(x)$ in $Y$.
%\begin{remark}
%\textcolor{red}{REMOVE? In fact, we do not need $\Phi$ to be increasing for the results in this section.}
%\end{remark}
\begin{prop}[Existence for the penalised equation]\label{lem:existencePenalisedPDE}Under Assumption \ref{ass:seqReg}, assume
\begin{align}
&\text{there exists $v_0 \in V$ such that $v_0 \leq \Phi(v)$ for all $v \in V$}\label{ass:feasiblePoint}
%&m_\rho\colon V \to V^* \text{ is increasing or }\lim_{\norm{y}{V} \to \infty} \norm{\Phi(y)}{V}\slash \norm{y}{V} < \infty\label{ass:mrIncreasing}.
\end{align}
and one of the following:
\begin{subequations}
\begin{align}
&m_\rho(\Id-\Phi)\colon V \to V^* \text{ is completely continuous,}\label{ass:mrCC}\\
&m_\rho(\Id-\Phi)\colon V \to V^* \text{ is monotone, radially continuous and bounded}\label{ass:mrMonotoneBdded}.
\end{align}
\end{subequations}
Given $f \in V^*$, there exists a solution $y_\rho \in V$ of \eqref{eq:penalised}. Furthermore, every solution satisfies
\begin{equation*}
\norm{y_\rho}{V} \leq C\left(\norm{f}{V^*} + \norm{v_0}{V}\right),
\end{equation*}
where $C$ is independent of $\rho$. 
%
%Furthermore, if
%\begin{equation}\label{ass:PhiGrowth}
%\end{equation}
%then \eqref{ass:mrIncreasing} is not necessary.
\end{prop}
\begin{proof}
%Since $\Phi$ is completely continuous, it is compact and therefore bounded.
%We see that, using $H \cts V^*$ and \eqref{ass:mrBd}, % the Lipschitz continuity of $m_\rho$, % \cite[Lemma 2.1 (v)]{MR2822818},
%\begin{align*}
%\norm{Ay + \frac{1}{\rho}m_\rho(y-\Phi(y))}{V^*} \leq C_b\norm{y}{V} + \frac{C}{\rho}\norm{y-\Phi(y)}{H},
%\end{align*}
%so that 
We have that $A+(1\slash \rho)m_\rho(\mathrm{I}-\Phi)$ is a bounded operator (under \eqref{ass:mrCC}, recall that completely continuous maps are bounded). Let us show that it is also coercive.  First, by adding and subtracting the same term, observe the formula
\begin{align*}
\nonumber \langle m_\rho(y_\rho-\Phi(y_\rho)), y_\rho-  v_0\rangle &= \langle m_\rho(y_\rho-\Phi(y_\rho)) - m_\rho(v_0-\Phi(y_\rho)), y_\rho-v_0\rangle\\
&\geq 0
\end{align*}
(by monotonicity \eqref{ass:mrMonotonicity} and because $m_\rho \equiv 0$ on $(-\infty,0]$ from \eqref{ass:mrPenal}). Now, using this, we have
\begin{align*}
\langle Ay_\rho, y_\rho -v_0 \rangle + \langle m_\rho(y_\rho - \Phi(y_\rho)), y_\rho - v_0 \rangle \geq C_a\norm{y_\rho}{V}^2 - C_b\norm{y_\rho}{V}\norm{v_0}{V},
\end{align*}
which yields coercivity of the full elliptic operator. 

Suppose that \eqref{ass:mrCC} is available. By \cite[\S 2, Lemma 2.1]{Showalter}, $A$ is a type M operator. Since the sum of a type M operator and a completely continuous operator is type M \cite[\S 2, Example 2.B]{Showalter},  %$\Phi$ and $V \ctsCompact H$, 
%the term $\rho^{-1}m_\rho(y-\Phi(y))$ is completely continuous and 
we get that the full elliptic operator is of type M. Then \cite[\S 2, Corollary 2.2]{Showalter} yields existence.  Under \eqref{ass:mrMonotoneBdded}, the full elliptic operator is pseudomonotone by \cite[Lemma 2.9 and Lemma 2.11]{Roubicek} giving existence via \cite[Theorem 2.6]{Roubicek}.

Regarding the estimate on the solution, we test the equation with $y_\rho-v_0$ and use the above coercivity estimate to find %the formula \eqref{eq:mFormula} to get
\begin{align*}
C_a\norm{y_\rho}{V}^2 %&\leq C_b\norm{y_\rho}{V}\norm{v_0}{V} + \norm{f}{H}\norm{y-v_0}{H}\\
&\leq C_b\norm{y_\rho}{V}\norm{v_0}{V} + \norm{f}{V^*}\norm{y_\rho}{V} + \norm{f}{V^*}\norm{v_0}{V}\\
&\leq \frac{C_a}{3}\norm{y_\rho}{V}^2 + \frac{3C_b^2}{4C_a}\norm{v_0}{V}^2 + \frac{3}{4C_a}\norm{f}{V^*}^2 + \frac{C_a}{3}\norm{y_\rho}{V}^2 + \frac 12 \norm{f}{V^*}^2  + \frac 12 \norm{v_0}{V}^2.
\end{align*}
This gives the uniform bound
\begin{align*}
\frac{C_a}{3}\norm{y_\rho}{V}^2 &\leq \left(\frac{3C_b^2}{4C_a} + \frac 12\right)\norm{v_0}{V}^2 + \left(\frac{3}{4C_a}+ \frac 12\right)\norm{f}{V^*}^2.\qedhere
\end{align*}
\end{proof}
\begin{remark}
The assumptions of the previous lemma are by no means necessary. One could, for example, ask for
$(\Id-\Phi)\colon V \to V \text{ to be invertible and } A(\Id-\Phi)^{-1}\colon V \to V^* \text{ to be pseudomonotone and coercive}$ instead of \eqref{ass:mrCC} or \eqref{ass:mrMonotoneBdded} and then apply \cite[Theorem 2.6]{Roubicek} to obtain the same result.
\end{remark}
Let us point out a very common setting.
\begin{example}[Gelfand triple case]\label{eg:gelfandTripleEg}
%A common example is the following.

Suppose that
\begin{align}
	 &\text{$V \subset H \equiv H^* \subset V^*$ is a Gelfand triple with $V \ctsCompact H$ and $H$ is a vector lattice defined via \eqref{eq:cone} with $H_+$ closed},\label{ass:gf1}\\
	 &\Phi\colon V \to H \text{ is completely continuous}.\label{ass:PhiVtoHCC}
\end{align}
Set $h^+=P_{H_+}h$ to be the orthogonal projection in $H$. %of $h$ onto $H_+$. 
We assume that $(\cdot)^+ \colon V \to V$. %Let also $\Phi\colon V \to H$ be completely continuous.  

%Let $V \subset H \equiv H^* \subset V^*$ be a Gelfand triple with $V \ctsCompact H$ and suppose that $H$ is a vector lattice defined via \eqref{eq:cone} with $H_+$ closed. Set $h^+=P_{H_+}h$ to be the orthogonal projection in $H$. %of $h$ onto $H_+$. 
%We assume that $(\cdot)^+ \colon V \to V$. Let also $\Phi\colon V \to H$ be completely continuous.  %=\{ h \in H : h \geq 0\}$, which we suppose is closed, 

We can take $m_\rho\colon V \to H^* \equiv H$ defined by
\[m_\rho(v) := (v^+, \cdot)_H\]
and this satisfies \eqref{ass:mrPenal}, \eqref{ass:mrMonotonicity}, \eqref{ass:mrForFeas} and \eqref{ass:mrCC}. %If further $V \ctsCompact H$ is a compact embedding and $\Phi\colon V \to H$ is completely continuous (which is implied by \eqref{ass:PhiCC}), 
Indeed, \eqref{ass:mrForFeas} follows because $P_{H_+}\colon H \to H_+$ is Lipschitz continuous and the compact embedding and complete continuity imply \eqref{ass:mrCC} (using the fact that the projection operator is continuous in $H$).
\end{example}

We write the possibly multivalued solution mapping associated to the equation under study as $\mathbf P_\rho\colon V^* \rightrightarrows V$, so \eqref{eq:penalised} reads $y_\rho \in \mathbf P_\rho(f)$.
Now, thanks to the lemma, for every source term $f_\rho \in V^*$, the following equation has a solution $y_\rho$:
\begin{equation}\label{eq:penalisedGeneral}
Ay_\rho + \frac 1\rho m_\rho(y_\rho - \Phi(y_\rho)) = f_\rho.
\end{equation}
 The next two theorems show that solutions of the regularised problem \eqref{eq:penalised} converge to  solutions of the QVI under varying assumptions. %if we choose the parameter $\epsilon$ such that $\{\epsilon(\rho)\slash \rho\}_\rho$ is bounded; note that is a requirement special to our QVI case and was not necessary in the setting of \cite{MR2822818}.
\begin{theorem}[Existence and approximation of solutions to the QVI]\label{thm:penalisedConvergence}Let Assumption \ref{ass:seqReg}, \eqref{ass:feasiblePoint}, either \eqref{ass:mrCC} or \eqref{ass:mrMonotoneBdded} and
\begin{equation}
\text{$\Phi\colon V \to V$ is completely continuous}\label{ass:PhiCC}
\end{equation}
hold. Take a sequence $f_\rho \to f$ in $V^*$. Then there exists a subsequence $\{\rho_n\}_n$ and elements $y_{\rho_n} \in \mathbf P_{\rho_n}(f_{\rho_n})$ such that $y_{\rho_n} \to y$ in $V$ where $y \in \mathbf{Q}(f)$. %\textcolor{red}{Unlike Theorem 2.3 of \cite{MR2822818}, I am unable to show that this convergence is strong due to the error term in the monotonicity formula below.}
\end{theorem}
\begin{proof}
The proof is in four steps and is similar to the proof of Theorem 2.3 of \cite{MR2822818}.
\\

\noindent\textit{1. Uniform estimates and feasibility of limit.}  For each $\rho$, let $y_\rho$ be a solution of \eqref{eq:penalisedGeneral} (such a selection is possible due to the axiom of choice). By Proposition \ref{lem:existencePenalisedPDE}, it satisfies the estimate
\[\norm{y_\rho}{V} \leq  C\left(\norm{f_\rho}{V^*} + \norm{v_0}{V}\right),\]
and this is bounded, hence for a subsequence (which we do not attempt to differentiate for ease of reading), $y_\rho \weaklyto y$ in $V$ to some $y$.
%Test \eqref{eq:penalisedGeneral} with $y_\rho -v $  where we take $v$ such that $v \leq \Phi(y_\rho)$, and with this choice the final term in \eqref{eq:monotonicityIneq} equals zero. So then we find that \eqref{eq:penalisedGeneral} implies
%\[\langle Ay_\rho, y_\rho \rangle \leq \langle f_\rho, y_\rho -v \rangle + \langle Ay_\rho, v \rangle\]
%which leads to (pick $v=0 \leq \Phi(y_\rho)$) 
%\[C_a\norm{y_\rho}{V} \leq  \norm{f_\rho}{V^*}\] and hence $y_\rho \weaklyto y$ in $V$ to some $y$	.
 Rearranging the equality \eqref{eq:penalisedGeneral}, 
\[\norm{m_\rho(y_\rho - \Phi(y_\rho))}{V^*} = \rho\norm{f_\rho-Ay_\rho}{V^*} \leq C\rho\]
and therefore $m_\rho(y_\rho - \Phi(y_\rho)) \to 0$ in $V^*$ as $\rho \to 0$. Then \eqref{ass:mrForFeas} implies that $y \leq \Phi(y)$.
% Since $\{\epsilon(\rho)\}$ is bounded, we have (for a subsequence that we relabelled) $\epsilon(\rho) \to \bar \epsilon$ for some $\bar \epsilon\geq 0$ and we get
% \begin{align*}
%\norm{\max_{\bar\epsilon}(0, y- \Phi(y)) - \max_{\epsilon(\rho)}(0, y_\rho - \Phi(y_\rho))}{V^*} &\leq \norm{\max_{\bar\epsilon}(0, y- \Phi(y)) - \max_{\bar\epsilon}(0, y_\rho - \Phi(y_\rho))}{H}\\
%&\quad + \norm{\max_{\bar\epsilon}(0, y_\rho- \Phi(y_\rho)) - \max_{\epsilon(\rho)}(0, y_\rho - \Phi(y_\rho))}{H}\\
%&\leq \norm{y - y_\rho}{H} + \norm{\Phi(y) - \Phi(y_\rho)}{H} + \frac 32|\bar \epsilon - \epsilon(\rho)|\\
%&\to 0
%\end{align*}
%with the convergence due to \cite[Lemma 2.1 (iv) and (v)]{MR2822818} (in fact, up to here, complete continuity of $\Phi\colon V \to H$ would suffice rather than \eqref{ass:PhiCC}).
% Since $\epsilon(\rho)\slash \rho$ is bounded, we have $\epsilon(\rho) \to 0$; we use this in the following calculation:
%\begin{align*}
%\norm{\max(0, y- \Phi(y)) - \max_{\epsilon(\rho)}(0, y_\rho - \Phi(y_\rho))}{V^*} &\leq \norm{\max(0, y- \Phi(y)) - \max(0, y_\rho - \Phi(y_\rho))}{H}\\
%&\quad + \norm{\max_{}(0, y_\rho- \Phi(y_\rho)) - \max_{\epsilon(\rho)}(0, y_\rho - \Phi(y_\rho))}{H}\\
%&\leq \norm{y - y_\rho}{H} + \norm{\Phi(y) - \Phi(y_\rho)}{H} + \frac{\epsilon(\rho)}{2}|\Omega|^{1\slash 2}\\
%&\to 0
%\end{align*}
%using \cite[Lemma 2.1 (iv)]{MR2822818} (in fact, up to here, complete continuity of $\Phi\colon V \to H$ would suffice rather than \eqref{ass:PhiCC}). Hence we find $\max(0, y- \Phi(y))=0$, which tells us that $y \leq \Phi(y)$. 

\medskip

\noindent\textit{2. Monotonicity formula}. 
For $v \in V$, we get by adding and subtracting the same term and using the monotonicity of $m_\rho$,
\begin{align}
\nonumber \langle m_\rho(y_\rho - \Phi(y_\rho)), y_\rho-v\rangle %&= m_\rho(y_\rho - \Phi(y_\rho))(y_\rho-\Phi(y_\rho) + \Phi(y_\rho) - v)\\
\nonumber &= \langle m_\rho(y_\rho - \Phi(y_\rho))-m_\rho(v-\Phi(y_\rho)), y_\rho-\Phi(y_\rho) + \Phi(y_\rho) - v\rangle\\
\nonumber &\quad+ \langle m_\rho(v-\Phi(y_\rho)), y_\rho-v\rangle\\
&\geq \langle m_\rho(v-\Phi(y_\rho)), y_\rho-v\rangle\label{eq:monotonicityIneq}.
\end{align}

\noindent \textit{3. Passage to the limit.} Test the equation \eqref{eq:penalisedGeneral} with $y_\rho - v$ for $v \in V$ and use \eqref{eq:monotonicityIneq} to find
%\[\langle Ay_\rho, y_\rho \rangle + \frac 1\rho \langle m_{\rho}(y_\rho - \Phi(y_\rho)), y_\rho - v \rangle = \langle f_\rho, y_\rho - v \rangle + \langle Ay_\rho, v \rangle \]
%To deal with the second term, 
\begin{align}
\langle Ay_\rho, y_\rho \rangle +\frac 1\rho \langle m_\rho(v-\Phi(y_\rho)), y_\rho-v\rangle \leq \langle f_\rho, y_\rho - v \rangle + \langle Ay_\rho, v \rangle. \label{eq:prelim}
\end{align}
Now, choose an arbitrary $v^* \in V$ with $v^* \leq \Phi(y)$ and select the test function to be 
\[v_\rho=v^* - \Phi(y) + \Phi(y_\rho).\]
With this choice, the second term on the left-hand side of the above inequality \eqref{eq:prelim} is equal to zero by \eqref{ass:mrPenal} and we find
\[\langle Ay_\rho, y_\rho \rangle \leq  \langle f_\rho, y_\rho - v_\rho \rangle + \langle Ay_\rho, v_\rho \rangle.\]
Noting that $v_\rho \to v^*$ in $V$ (thanks to the complete continuity \eqref{ass:PhiCC}) and $v_\rho \leq \Phi(y_\rho)$, take the limit inferior as $\rho \to 0$ above and use weak lower semicontinuity to get $y \in \mathbf{Q}(f)$. \\%We also need complete continuity of $\Phi\colon V \to V$ to pass to the limit in the last term of the inequality as we need $v_\rho \to v^*$ strongly in $V$.\\

\noindent \textit{4. Strong convergence.} Define $v_\rho := y+\Phi(y_\rho)-\Phi(y)$ which has the properties
\begin{align*}
&v_\rho \to y \text{ in $V$},\\
&v_\rho \leq \Phi(y_\rho),\\
&y_\rho - v_\rho = (y_\rho - y) + (\Phi(y)-\Phi(y_\rho)) \weaklyto 0 \text{ in $V$},
\end{align*}
the first holding since we already have $y_\rho \weaklyto y$ in $V$. %By coercivity we obtain the estimate
%\begin{align*}
%\langle A(y_\rho -v_\rho), y_\rho - v_\rho \rangle %&\geq C_a\norm{y_\rho -v_\rho}{V}^2\\
%&\geq C_a\norm{(y_\rho - y) + (\Phi(y)-\Phi(y_\rho))}{V}^2\\
%&= C_a\norm{y_\rho - y}{V}^2 + C_a\norm{\Phi(y)-\Phi(y_\rho)}{V}^2 + 2C_a(y_\rho - y, \Phi(y)-\Phi(y_\rho))_V.
%\end{align*}
Testing \eqref{eq:penalisedGeneral} appropriately, we have
\[\langle A(y_\rho - v_\rho), y_\rho-v_\rho \rangle = \langle f_\rho, y_\rho -v_\rho \rangle - \frac 1\rho \langle m_\rho(y_\rho-\Phi(y_\rho)), y_\rho - v_\rho\rangle  - \langle Av_\rho, y_\rho - v_\rho \rangle\]
and to this we apply the monotonicity formula and coercivity of $A$ to find
\begin{align*}
C_a\norm{y_\rho - v_\rho}{V}^2 &\leq \langle f_\rho, y_\rho -v_\rho \rangle - \frac 1\rho \langle  m_\rho(v_\rho-\Phi(y_\rho)), y_\rho - v_\rho\rangle - \langle Av_\rho, y_\rho - v_\rho \rangle\\
&= \langle f_\rho, y_\rho -v_\rho \rangle- \langle Av_\rho, y_\rho - v_\rho \rangle\tag{since $v_\rho \leq \Phi(y_\rho)$}.
\end{align*}
The right-hand side converges to zero, hence $y_\rho-v_\rho \to 0$ strongly in $V$, implying $y_\rho \to y$.
\end{proof}

Theorem \ref{thm:penalisedConvergence} requires the complete continuity condition \eqref{ass:PhiCC} on $\Phi$. Let us consider how this assumption can be weakened or substituted.
\begin{theorem}\label{thm:newOne}
%Theorem \ref{thm:penalisedConvergence} is still true if assumption \eqref{ass:PhiCC}  is replaced with
Assume the conditions of Theorem \ref{thm:penalisedConvergence}, except replace the assumption \eqref{ass:PhiCC} with
\begin{align}
&\text{$\langle A(\cdot), (\Id-\Phi)(\cdot) \rangle \colon V \to \mathbb{R}$ is weakly lower semicontinuous}\label{ass:newWLSC}
\end{align}
and assume one of the following:
\begin{align}
&\text{$\Phi\colon V \to V$ is weakly sequentially continuous},\label{ass:PhiWeakCts}\\
&\text{\eqref{ass:gf1}, \eqref{ass:PhiVtoHCC} and $f_\rho \weaklyto f$ in $H$.}\nonumber
\end{align}
%and suppose either that
%\begin{align}
%&\text{$\Phi\colon V \to V$ is weakly sequentially continuous}\label{ass:PhiWeakCts},
%\end{align}
%or that \eqref{ass:gf1}, \eqref{ass:PhiVtoHCC} and $f_\rho \weaklyto f$ in $H$.
Then there exists a subsequence $\{\rho_n\}_n$ and elements $y_{\rho_n} \in \mathbf P_{\rho_n}(f_{\rho_n})$ such that $y_{\rho_n} \weaklyto y \in \mathbf{Q}(f)$ in $V$. 
\end{theorem}
\begin{proof}
We modify the third step of the proof of Theorem \ref{thm:penalisedConvergence} (and, like before, we do not distinguish subsequences of $\{\rho\}$). We write the final inequality of step 3 as $\langle Ay_\rho, y_\rho-v_\rho \rangle \leq  \langle f_\rho, y_\rho - v_\rho \rangle$, which, recalling $v_\rho=v^* - \Phi(y) + \Phi(y_\rho),$ is
\[\langle Ay_\rho, y_\rho-\Phi(y_\rho)\rangle + \langle Ay_\rho, \Phi(y)-v^* \rangle \leq  \langle f_\rho, y_\rho - v_\rho\rangle.\] %we note that $v_\rho \weaklyto v^*$ in $V$ 
By \eqref{ass:newWLSC}, we can take the limit inferior on the left-hand side. Regarding the right-hand side, let us consider the two cases separately.

\medskip

\noindent (1) Under \eqref{ass:PhiWeakCts}, $y_\rho-v_\rho \weaklyto y-v^*$ in $V$,  and since $f_\rho \to f$ in $V^*$, we can pass to the limit on the right-hand side and we obtain $y \in \mathbf{Q}(f)$, hence $y_\rho \weaklyto y$ in $V$. 

\medskip

\noindent (2) In the Gelfand triple case, we write the final term in the inequality above as the inner product $(f_\rho, y_\rho-v_\rho)_H$ and pass to the limit easily.
\end{proof}
%\begin{remark}
%In the Gelfand triple $V \ctsCompact H \cts V^*$ setting, \eqref{ass:PhiWeakCts} can be replaced by 
%\begin{align}
%	&\Phi\colon V \to H \text{ is completely continuous},\\
%	&f_\rho \weaklyto f \text{ in $H$}
%\end{align}
% for the weak convergence result $y_\rho \weaklyto y$. %we could assume $\Phi\colon V \to H$ to be complete continuous instead of \eqref{ass:PhiWeakCts} if we had the additional condition that $f_\rho \weaklyto f$ in the pivot space $H$. To get strong convergence, \eqref{ass:PhiWeakCts} seems necessary. 
%\end{remark}
\begin{remark}	
It is not difficult to see that \eqref{ass:newWLSC} and \eqref{ass:PhiWeakCts} are weaker assumptions than  \eqref{ass:PhiCC}. %\eqref{ass:PhiCC} implies both \eqref{ass:PhiWeakCts} and \eqref{ass:newWLSC}. The converse is not true since $\Phi=\Id$ satisfies the latter two assumption but is not completely continuous. Hence the two assumptions are weaker. %THIS IS AN IMPORTANT EXAMPLE CASE!
\end{remark}
The theorem provides only weak convergence but strong convergence can be attained under additional assumptions as the next remark shows.
\begin{remark}
If, in addition to the conditions of Theorem \ref{thm:newOne} under the weak sequential continuity condition \eqref{ass:PhiWeakCts}, we also have
\begin{equation} 
\langle A\Phi(\cdot), (\Id-\Phi)(\cdot)\rangle\colon V \to \mathbb{R} \text{ is weakly lower semicontinuous\footnotemark}\label{ass:newWLSCForStrong},
\end{equation}
\footnotetext{Note that \eqref{ass:PhiWeakCts} and \eqref{ass:newWLSCForStrong} imply \eqref{ass:newWLSC}. Indeed, taking the limit inferior of $\langle Au_n, (\Id-\Phi)(u_n) \rangle = \langle A(\Id-\Phi)u_n, (\Id-\Phi)u_n \rangle + \langle A\Phi(u_n), (\Id-\Phi)(u_n) \rangle$, using superadditivity and weak sequential continuity on the first term and \eqref{ass:newWLSCForStrong} on the second term allows us to deduce the claim.}then $y_{\rho_n}-\Phi(y_{\rho_n}) \to y-\Phi(y)$ in $V$. To see this, returning to step 4 of the proof of Theorem \ref{thm:penalisedConvergence} where we recall $v_\rho = y + \Phi(y_\rho) - \Phi(y)$, we start with the calculation
\begin{align*}
&\liminf_{\rho \to 0}\langle Av_\rho, y_\rho - v_\rho \rangle\\ %&= \liminf_{\rho \to 0}\langle A(v^*-\Phi(y) + \Phi(y_\rho)), (\Id-\Phi)(y_\rho) + \Phi(y)-y \rangle\\
&\quad= \liminf_{\rho \to 0}\left(\langle A(y-\Phi(y)), (\Id-\Phi)(y_\rho) + \Phi(y)-y \rangle + \langle A\Phi(y_\rho), (\Id-\Phi)(y_\rho) \rangle + \langle A\Phi(y_\rho), \Phi(y)-y \rangle\right)\\
&\quad\geq \langle A(y-\Phi(y)), (\Id-\Phi)(y) + \Phi(y)-y \rangle + \langle A\Phi(y), (\Id-\Phi)(y) \rangle + \langle A\Phi(y), \Phi(y)-y \rangle\\%\tag{by superadditivity of the limit inferior}\\
&\quad= 0,
\end{align*}
where for the inequality we used weak continuity for the first and last terms and \eqref{ass:newWLSCForStrong} for the middle term. Now, taking the limit superior in the final inequality of the proof of Theorem \ref{thm:penalisedConvergence}, using the identity $\limsup(a_n) + \liminf(b_n) \leq \limsup(a_n + b_n)$ and the above calculation, we get
\begin{align*}
\limsup_{\rho \to 0} \langle f_\rho, y_\rho -v_\rho \rangle &\geq \limsup_{\rho \to 0} \left(C_a\norm{y_\rho - v_\rho}{V}^2 + \langle Av_\rho, y_\rho - v_\rho \rangle\right)\\
%&\geq \limsup_{\rho \to 0} C_a\norm{y_\rho - v_\rho}{V}^2 + \liminf_{\rho \to 0}\langle Av_\rho, y_\rho - v_\rho \rangle\\
& \geq \limsup_{\rho \to 0} C_a\norm{y_\rho - v_\rho}{V}^2.
\end{align*}
%We thus obtain $\limsup_{\rho \to 0}\norm{y_\rho - v_\rho}{V}^2 =0$.  Then, since $\limsup \geq \liminf$ and the norm is non-negative, we find that the limit exists and equals $0$, so that $y_\rho \to y$ in $V$.
Since the left-hand side is zero (by \eqref{ass:PhiWeakCts}), we deduce that $y_\rho - v_\rho \to 0$ and hence $y_\rho -\Phi(y_\rho) \to y-\Phi(y)$ in $V$. 

We see then that if for example 
\[(\Id-\Phi)^{-1}\colon V \to V \text{ exists and is continuous,}\] we would also get the strong convergence $y_\rho \to y$.
%The proof reveals that we would get $y_\rho \weaklyto y$ (with the emphasis on the weak convergence) if we assumed \eqref{ass:newWLSC} instead of \eqref{ass:newWLSCForStrong} in the theorem.
\end{remark}

\begin{remark}If $\mathbf{Q}(f)$ is a singleton, then the convergence results of the previous theorems hold for the entire sequence and not just a subsequence because the limit $y = \mathbf{Q}(f)$ is unique.
\end{remark}
%
%\begin{remark}\label{rem:penalisedConv}
%In the situation of Remark \ref{rem:differentBoundForPenalisation}, the requirement that $\{\epsilon(\rho)\slash \rho\}$ is bounded for Theorem \ref{thm:penalisedConvergence} is unnecessary. Since $\{\epsilon(\rho)\}$ is bounded, we have (for a subsequence that we relabelled) $\epsilon(\rho) \to \bar \epsilon$ for some $\bar \epsilon\geq 0$ and we can replace the calculation in the first step of the proof of the previous theorem by
%\begin{align*}
%\norm{\max_{\bar\epsilon}(0, y- \Phi(y)) - \max_{\epsilon(\rho)}(0, y_\rho - \Phi(y_\rho))}{V^*} &\leq \norm{\max_{\bar\epsilon}(0, y- \Phi(y)) - \max_{\bar\epsilon}(0, y_\rho - \Phi(y_\rho))}{H}\\
%&\quad + \norm{\max_{\bar\epsilon}(0, y_\rho- \Phi(y_\rho)) - \max_{\epsilon(\rho)}(0, y_\rho - \Phi(y_\rho))}{H}\\
%&\leq \norm{y - y_\rho}{H} + \norm{\Phi(y) - \Phi(y_\rho)}{H} + \frac 32|\bar \epsilon - \epsilon(\rho)|\\
%&\to 0
%\end{align*}
%with the convergence due to \cite[Lemma 2.1 (iv) and (v)]{MR2822818}.
%\end{remark}

%We shall need more assumptions on the obstacle map in the forthcoming sections  and these will mainly be differentiability requirements on $\Phi$. In case of the example above with a linear $L$, these can be checked without great difficulty. 
%In \cite[\S 6]{AHR} we studied in substantial detail an application in thermoforming and the mathematical model given there of the thermoforming process involves a QVI with a nonlinear obstacle mapping $\Phi$ related to a solution of a PDE and we showed that all desired assumptions (including those on differentiability) were satisfied.
\section{Directional differentiability}\label{sec:Diff}
In this section, we extend the results of our previous work \cite{AHR} which dealt with directional differentiability of the source-to-solution map $\mathbf{Q}$ associated to  \eqref{eq:QVI}  for non-negative source terms and directions.
% Here, we shall see that similar results hold for 
%\begin{itemize}
%\item[(a)]unsigned source and direction terms and 
%\item[(b)] for source and direction terms bounded from below
%\end{itemize}
%by using two different approaches. 
Formally, the goal is to show that the following limit exists: %$\mathbf{Q}$ is directionally differentiabile, i.e., that 
%there exists a $\mathbf{Q}'(f)(d) \in V$ such that
\[\lim_{s \to 0^+}\frac{\mathbf{Q}(f+sd)-\mathbf{Q}(f)}{s}.\]% = \mathbf{Q}'(f)(d).\]
This is merely a formal limit since $\mathbf{Q}\colon V \rightrightarrows V$ is  set valued in general, however in case $\mathbf{Q}\colon V \to V$ is single valued, it is precise. It is important to obtain such a sensitivity result not only for applications but also for the procurement of certain types of stationarity conditions for optimal control problems with QVI constraints, a topic that we will address in \S \ref{sec:stationarity}.

We will follow closely the approach of our earlier work \cite{AHR} where we combined an iteration (by VIs) argument with the directional differentiability result for VIs in Dirichlet space case provided by Mignot \cite{MR0423155} but here, we make two refinements: instead of the order approach for the iterations employed in \cite{AHR}, we shall use a contraction technique similar to that in \S \ref{sec:existenceContraction}, and secondly, we shall use the VI differentiability result in \cite{WGuidedTour} given under a general vector lattice setting, which generalises the result in \cite{MR0423155}. For this, we begin with the following assumption on the ordering.

\begin{ass}\label{ass:forDirDiff}
Let $V$ be a reflexive Banach space which is a vector lattice induced by a closed convex cone $C$ satisfying $C \cap -C = \{0\}$ and suppose that $v_n \to v$ in $V$ implies $\sup(0,v_n) \weaklyto \sup(0,v)$ in $V$.
\end{ass}
As before, we will identify $C$ with $V_+$ and note  that the strong-weak convergence part of the above assumption is satisfied if there exists a constant $M>0$ such that $\norm{\sup(0,v)}{V} \leq C\norm{v}{V}$ for all $v \in V$.
%As we explained before in \S \ref{sec:existenceTranslation}, this assumption defines an order on $V$.
%The pair $(V,\xi)$ is known as a \emph{regular Dirichlet form} and $V$ is the so-called \emph{Dirichlet space}. 
%Taken together with the bilinear form $a\colon V \times V \to \mathbb{R}$ generated by the operator $A$ defined in \S \ref{sec:basicass}, $(V,a)$ is known as a regular positivity preserving coercive form. 
%This framework allows us to define the notions of capacity, quasi-continuity and related concepts, see \cite[Chapter 2]{FukushimaBook}, \cite[\S 3]{Haraux1977} and \cite[\S 6.4.3]{MR1756264} for more details.  We will typically choose $X$ to be $\Omega$ or its closure $\overline{\Omega}$ (where $\Omega \subset \mathbb{R}^n$ is a sufficiently regular domain) depending on the choice of $V$, see the examples below in Remark \ref{rem:Examples}.
%\textcolor{blue}{In the following, we use the notion of a compact nonlinear operator. A map $T\colon X \to Y$ between Banach spaces is said to be \emph{compact} if it is continuous and maps bounded sets into relatively compact sets (i.e., $B \subset X$ bounded implies that $\overline{T(B)}$ is compact).} 
To state the main result, we need to introduce some notation. Recall from \eqref{eq:QVI} the constraint set mapping $\mathbf{K}\colon V \rightrightarrows V$ defined by 
\[\mathbf K(w) := \{ v \in V : v \leq \Phi(w)\}.\]
This is convex and closed (since $V_+$ is closed), and associated to this, we define the \emph{radial cone} of $\mathbf K(w)$ at a point $u\in \mathbf K(w)$ by
\[\mathcal{R}_{\mathbf K(w)}(u) := \{ h \in V : \exists  s^* > 0 \text{ such that } u+sh \in \mathbf{K}(w) \;\forall s \in [0,s^*]\}\]
and the corresponding \emph{tangent cone} $\mathcal{T}_{\mathbf K(w)}(u) := \overline{\mathcal{R}_{\mathbf K(w)}(u)}$. Finally, recall the notation $B_R(y)$ to stand for the closed ball in $V$ of radius $R$ centred at $u$.

\begin{theorem}\label{thm:dirDiff1}
Let Assumption \ref{ass:forDirDiff} hold and given $f \in V^*$ and $d \in V^*$, take $y \in \mathbf{Q}(f)$ satisfying %assumption \eqref{ass:PhiCC}, 
the local assumptions 
\begin{align}
&\text{there exists $\epsilon > 0$ such that  } \Phi\colon B_\epsilon(y) \to V \text{ is Lipschitz with Lipschitz constant $C_\Phi < C_a\slash(C_a+C_b)$},\label{ass:PhiLipBound}\\
&\text{$\Phi\colon V \to V$ is Hadamard directionally differentiable at $y$.}\label{ass:PhiHadamardAtY}\end{align}
Then, for $s > 0$ sufficiently small, there exists $y^s \in \mathbf{Q}(f+sd)\cap B_R(y)$  (where $0 < R \leq \epsilon$) and $\alpha=\alpha(d)\in V$ such that
\[y^s = y + s\alpha + o(s)\]
%Under some assumptions on $\Phi$ (see Assumption \ref{ass:onPhi} on page \pageref{ass:onPhi}) including that $\Lip{\Phi} < 1\slash 2$, the mapping $Q$ is directionally differentiable at every non-negative $f \in L^\infty$ in every non-negative direction $d \in L^\infty$:
%\[\mathbf{Q}(f+td) = \mathbf{Q}(f) + tQ'(f)(d)+ o^*(t)\]
where $s^{-1}o(s) \to 0$ in $V$ as $s \to 0^+$  and $\alpha$ satisfies the QVI
\begin{equation}\label{eq:QVIforAlpha}
\begin{aligned}
\alpha \in \mathcal{K}^y(\alpha) &: \langle A\alpha - d, \alpha - v \rangle \leq 0 \quad \forall v \in \mathcal{K}^y(\alpha),\\
\mathcal{K}^y(\alpha) &:= \Phi'(y)(\alpha)  + \mathcal{T}_{\mathbf K(y)}(y) \cap [f-Ay]^\perp.   %\{ \varphi \in V : \varphi \leq \Phi'(y)(w) \text{ q.e. on }\mathcal{A}(y) \text{ and } \langle Ay-f, \varphi -\Phi'(y)(w)\rangle = 0\}.
\end{aligned}
\end{equation}
The directional derivative $\alpha=\alpha(d)$ is positively homogeneous in $d$. 

%Furthermore, 
%\begin{enumerate}
%\item[(i)] if $d \in V^*_+$ or $-d \in V^*_+$, then the result remains valid when 
%\eqref{ass:PhiBoundedlyDiff} and \eqref{ass:compactnessOfAPhiPrime}\footnote{Obviously, \eqref{ass:compactnessOfAPhiPrime} is implied by \eqref{itm:compContOfDerivOfPhi}.} are replaced with 
%\begin{align}
%&\text{$\Phi\colon V \to V$ is increasing,}\\
%&\Phi'(y)\colon V \to V \text{ is completely continuous.} \label{itm:compContOfDerivOfPhi}
%\end{align}
%\item[(ii)] if \eqref{ass:PhiDerivativeLipschitzConstant} holds, then \eqref{ass:PhiBoundedlyDiff} and  \eqref{ass:compactnessOfAPhiPrime} are not needed.
%\end{enumerate}
\end{theorem}
The proof of this theorem will be given in the next subsections. For now, let us make some observations.
\begin{remark}\label{rem:remarks}
\begin{enumerate}[label=(\roman*)]
\item The stated assumptions do \textbf{not} force solutions of the QVI to be unique. We will construct examples demonstrating this fact in \S \ref{sec:examplesMulti}.
\item If there exists an $\epsilon$ such that $\Phi$ is Hadamard differentiable on $B_\epsilon(y)$ and 
\begin{equation}
%\exists \epsilon > 0, 
\forall z \in B_\epsilon(y), \forall v \in V, \quad \norm{\Phi'(z)(v)}{V} \leq C_\Phi\norm{v}{V}\quad\text{ where } C_\Phi < C_a\slash(C_a+C_b),\label{itm:smallnessOfDerivOfPhi}
\end{equation}
then \eqref{ass:PhiLipBound} holds. This is immediate: take $u, v \in B_\epsilon(y)$ and use the mean value theorem to find
\begin{align*}
\norm{\Phi(u)-\Phi(v)}{V} &\leq \sup_{\lambda \in (0,1)}\norm{\Phi'(\lambda u + (1-\lambda)v)(u-v)}{V} \leq C_\Phi\norm{u-v}{V},
\end{align*}
where we utilised the fact that $\lambda u + (1-\lambda) v \in B_\epsilon(y)$.  It can sometimes be easier to verify \eqref{itm:smallnessOfDerivOfPhi} than \eqref{ass:PhiLipBound} depending on the problem at hand.

% and \eqref{ass:PhiDerivativeLipschitzConstant} follow. 
\item The derivative $\alpha$ is the unique solution of the QVI \eqref{eq:QVIforAlpha}, see Proposition \ref{prop:strongConvAlphas}.
%\item We see that no form of compactness is necessary if \eqref{ass:PhiDerivativeLipschitzConstant} holds, which holds in particular when $\Phi$ is Fr\'echet differentiable. In such a situation, we are in the setting of \cite{MR4083198}.
\item All of the required assumptions on $\Phi$ are local, i.e., they are based at or around a neighbourhood of the chosen point $y$ and we do not ask for them to hold globally on the whole of $V$. We may introduce more local assumptions in the course of the paper and one should bear in mind that such conditions are stated in terms of a fixed element $y$ which, in later sections, need to be modified appropriately (for example in \S \ref{sec:stationarity} such assumptions should be evaluated at the function that we call $y^*$). This should become apparent from the context.
\item  In the theorem, the existence of a particular $y \in \mathbf{Q}(f)$ is \emph{assumed}; conditions under which $\mathbf{Q}(f)$ is non-empty were given in the existence results of \S \ref{sec:existence}.
\item This theorem generalises and improves the result of Theorem 1.6 in our earlier paper \cite{AHR}. In particular, the case $f, d \in V^*_+$ corresponds to the main result of \cite{AHR} (which also requires additional assumptions).
\item A differentiability result for QVIs also appears in \cite[Theorem 5.5]{MR4083198}. There, in particular, the author requires Fr\'echet differentiability for $\Phi$ at $y$. In contrast, we require only Hadamard differentiability. In \cite{MR4083198}, $A$ can be nonlinear of Fr\'echet type; we have taken $A$ to be linear in this paper for simplicity but this can be generalised: see Remark \ref{rem:generalisationOnA}.
\end{enumerate}
\end{remark}
\begin{remark}\label{rem:generalisationOnA}
We have taken $A$ to be linear for technical simplicity but an examination of the proofs that follow show that it would be possible for us to consider nonlinear $A$ that are Hadamard differentiable in this section (a key point would be to generalise \cite[Proposition 1]{AHR}, as we shall come to see in the proceeding). For the stationarity results of section \S \ref{sec:weakStationarity}, $A$ would need to be continuously Fr\'echet differentiable. The details and the resulting changes are left to the reader.
\end{remark}
Let us give an example of the functional setup which is typical for many applications.
\begin{example}[The case of a Dirichlet space]\label{eg:dirichletCase}
Suppose that $H:=L^2(X;\mu)$ where $X$ is a locally compact, separable metric space and $\mu$ is a positive Radon measure on $X$ with full support\footnote{That is, $\mu$ is a non-negative Borel measure which is finite on compact sets and strictly positive on non-empty open sets.}, and let $V \subset H$ be a dense subspace. The ordering on these spaces is given by the usual $a.e.$ ordering of functions.

Assume that there exists a symmetric, positive semidefinite bilinear form $\xi\colon V \times V \to \mathbb{R}$ such that endowing $V$ with
\[(\cdot,\cdot)_V := (\cdot,\cdot)_H + \xi(\cdot,\cdot)\]
makes it a Hilbert space. Furthermore, we assume the \emph{Markov property}\footnote{This is also known as the  \emph{unit contraction property}.}
\[\text{if $u \in V$ then $\hat u := \min(u^+,1) \in V$ and $\xi(\hat u,\hat u) \leq \xi(u,u)$}\]
and the density
\begin{equation*}%\label{eq:assDensity}
V\cap C_c(X)  \ctsDense C_c(X) \qquad \text{and}\qquad  V\cap C_c(X)  \ctsDense V. %\qquad%\text{are dense embeddings.}
\end{equation*}
The pair $(V,\xi)$ is known as a \emph{regular Dirichlet form} and $V$ is the so-called \emph{Dirichlet space}. 
This framework allows us to define the notions of capacity, quasi-continuity and related concepts, see \cite[\S 2.1]{FukushimaBook} and \cite[\S 3]{Haraux1977} for more details.  

In this setting, Mignot proved\footnote{In fact, Mignot uses a weaker setting of \emph{positivity-preserving forms} rather than the Dirichlet form setting described here with also some other weaker conditions.}the polyhedricity of sets of obstacle type in \cite[Theorem 3.2]{MR0423155} and the differentiability of VI solution maps associated to such constraint sets in \cite[Theorem 3.3]{MR0423155}. We also have an explicit expression for the critical cone appearing in \eqref{eq:QVIforAlpha} via \cite[Lemma 3.2]{MR0423155}:
\[\mathcal{K}^y(w) := \{ \varphi \in V : \varphi \leq \Phi'(y)(w) \text{ q.e. on }\mathcal{A}(y) \text{ and } \langle Ay-f, \varphi -\Phi'(y)(w)\rangle = 0\}.\]
Here, `q.e.' stands for quasi-everywhere and a statement holds quasi-everywhere if it holds everywhere except on a set of capacity zero, and $\mathcal{A}(y)$ refers to the \emph{active} or \emph{coincidence set} of the solution $y$ to the QVI related to an obstacle map $\Phi$, i.e.,
\[\mathcal{A}(y) := \{ x \in X : y(x) = \Phi(y)(x)\} \quad \text{for $y \in V$.}\]
We in fact take the quasi-continuous representatives of the functions appearing in the definition so that $\mathcal{A}(y)$ is quasi-closed and defined up to sets of capacity zero. It is important to note that the set of points defining the active set is taken over $X$; in the context of some Sobolev spaces over a domain $\Omega$, this can sometimes be $X=\overline{\Omega}$ and not merely $\Omega$, see \cite[\S 1.2]{AHR} for more details.
\end{example}
Before we proceed, let us provide some notation. Define the \emph{critical cone}
\begin{equation}\label{eq:criticalCone}
\mathcal{K}^y := \mathcal{T}_{\mathbf K(y)}(y) \cap [f-Ay]^\perp,
\end{equation}
and observe the relation \[\mathcal{K}^y(w) = \Phi'(y)(w) + \mathcal{K}^y.\]
Recall that the \textit{polar cone} of a set $M \subset V$ is defined as
\[M^\circ = \{ g \in V^* : \langle g, v \rangle \leq 0 \quad \forall v \in M\}.\]

\subsection{Iteration scheme and expansion formulae}
To prove Theorem \ref{thm:dirDiff1}, we employ an iteration and passage to the limit approach like in our previous work \cite{AHR}. We  fix an arbitrary $f \in V^*$ and take an arbitrary but fixed $y \in \mathbf{Q}(f)$\footnote{Again, see \S \ref{sec:existence} for existence of such $y$.}.  
%\text{there exists $\epsilon > 0$ such that $\Phi\colon V \to V$ is Hadamard directionally differentiable on $B_\epsilon(y)$.}%\tag{L1a}
%\exists \epsilon > 0 \text{ such that }\Phi\colon V \to V \text{ is Hadamard directionally differentiable at every $z \in B_\epsilon(y)$.}\tag{L1a}
%\end{equation}
Pick a direction $d \in V^*$ and construct, similarly to \S \ref{sec:existenceContraction}, the sequence
\begin{equation}\label{eq:defnysn}
\begin{aligned}
y^s_0 &:=  y,\\
y^s_n &:= S(f+sd, y^s_{n-1}).
\end{aligned}
\end{equation}
The idea here is to expand each $y^s_n$ in terms of $y$, a directional derivative and a remainder term (both of these would depend on $n$) and then to pass to the limit in such an expansion. The natural way to proceed would be to obtain a uniform bound on $\{y_n^s\}$ which would result in the existence of a weakly convergent \emph{subsequence} $\{y_{n_j}^s\}$. This is not enough to identify the limit of $\{y_{n_j}^s\}$ due to the $(n-1)$ index in the definition of $y_n^s$, so one would need convergence of the whole sequence which holds true when, for example, one has monotonicity. However, in contrast to the sequence considered in \S \ref{sec:existenceContraction}, we do not obtain any monotonicity of $\{y_n^s\}$ since we do not assume a sign on $d$ nor do we assume monotonicity of $\Phi$. Therefore, for convergence of the full sequence, we instead look for a contraction of the map associated to $\{y_n^s\}$ on some small ball. 
\begin{lem}Assume the Lipschitz property \eqref{ass:PhiLipBound}. 
Then for any $0 < R \leq \epsilon$, $S(f+sd, \cdot)\colon B_R(y) \to B_R(y)$ is a contraction whenever 
\begin{equation*}
s \leq C_a\norm{d}{V^*}^{-1}R(1-(1+C_bC_a^{-1})C_\Phi ).
\end{equation*}
\end{lem}
\begin{proof}
Let $v \in B_R(y)$; we want to show that $S(f+sd,v) \in B_R(y)$. Observe that, using $y=S(f,y)$ and continuous dependence (e.g. \cite[Equation (21)]{AHR}),% and the mean value theorem \cite[\S2, Proposition 2.29]{Penot},
\begin{align*}
\norm{S(f+sd,v) - y}{V} &\leq (1+C_bC_a^{-1})\norm{\Phi(v)-\Phi(y)}{V} + C_a^{-1}s\norm{d}{V^*}\\
%&\leq (1+C_bC_a^{-1})\sup_{\lambda \in (0,1)}\norm{\Phi'(\lambda v+(1-\lambda)y)(v-y)}{V} + C_a^{-1}s\norm{d}{V^*}\\
&\leq (1+C_bC_a^{-1})C_\Phi\norm{v-y}{V} + C_a^{-1}s\norm{d}{V^*}\tag{since $v, y \in B_R(y) \subset B_\epsilon(y)$}\\%\tag{since $\lambda v+(1-\lambda)y \in B_R(y) \subset B_\epsilon(y)$}\\
&\leq (1+C_bC_a^{-1})C_\Phi R + C_a^{-1}s\norm{d}{V^*},
%&\leq \epsilon R + C_a^{-1}s\norm{d}{V^*}\tag{for $0 < \epsilon < 1$, since $C_\Phi < (1+C_bC_a^{-1})^{-1}$},
\end{align*}
and, using the fact that $(1+C_bC_a^{-1})C_\Phi$ equals a constant strictly less than $1$, the right-hand side is bounded above by $R$. This shows that $S(f+sd, \cdot)$ maps $B_R(y)$ into itself. 
%\end{proof}
%\begin{lem}Let $R > 0$, $s$ satisfy \eqref{ass:conditionOnS}, and for $h\colon (0,1) \to V$ with $\lim_{\lambda \to 0} h(\lambda) \in B_R(y)$,  
%\begin{equation}\label{ass:contOfPhiDer}
%\text{$\norm{\Phi'(h(\lambda))(z_2-z_1)}{V} \leq \hat C_\Phi\norm{z_2-z_1}{V}$ for $z_1, z_2 \in B_R(0)$ with $\hat C_\Phi(1+C_a^{-1}C_b) < 1$.}\tag{H5}
%\end{equation}
%Then $S(f+sd,\cdot)\colon B_R(y) \to B_R(y)$ is a contraction.
%\end{lem}
%\begin{proof}
To see that the map is a contraction, take $v, w \in B_R(y)$  and observe that
\begin{align*}
\norm{S(f+sd, v) - S(f+sd, w)}{V} %&\leq (1+C_a^{-1}C_b)\sup_{\lambda \in (0,1)}\norm{\Phi'(\lambda w+(1-\lambda)v)(w-v)}{V}\\
&\leq (1+C_a^{-1}C_b)\norm{\Phi(v)-\Phi(w)}{V}
%&=(1+C_a^{-1}C_b)\sup_\lambda \norm{\Phi'(y+z_1+\lambda(z_2-z_1))(z_2-z_1)}{V}\\
\leq C_\Phi(1+C_a^{-1}C_b)\norm{v-w}{V}.\qedhere
%&< \norm{v-w}{V}.
\end{align*}
%which gives the result.
\end{proof} 
Hence, under \eqref{ass:PhiLipBound}, we have that each $y_n^s \in B_R(y)$. By applying the Banach fixed point theorem, we obtain the following existence and convergence result.
\begin{prop}
Given $f, d \in V^*$ and $y \in \mathbf{Q}(f)$, under \eqref{ass:PhiLipBound} and sufficiently small $s>0$, there exists $y^s \in \mathbf{Q}(f+sd) \cap B_R(y)$ such that $y^s_n \to y^s$ in $V$ (where $y^s_n$ is defined in \eqref{eq:defnysn}).
\end{prop}
Since we want to study differentiability of QVIs, we need some differentiability for the constraint set mapping. We will henceforth assume the Hadamard differentiability at $y$ condition 
\eqref{ass:PhiHadamardAtY}. Now, making use of \cite[Theorems 4.18 and 5.2]{WGuidedTour} %differentiability result for VIs provided by Mignot \cite[Theorem 3.3]{MR0423155}, 
we can expand $y_1^s=S(f+sd, y)$ as follows:
\begin{equation*}
y_1^s = y + s\alpha_1 + o_1(s),
\end{equation*}
where $s^{-1}o_1(s) \to 0$ as $s \to 0^+$ and $\alpha_1 = \partial S(f,y)(d)$ is the directional derivative of $S(\cdot,y)$ in the direction $d$, and this satisfies the VI (recall $\mathcal{K}^y$ from \eqref{eq:criticalCone})
\begin{equation*}%\label{eq:VIForS}
\begin{aligned}
\alpha_1 \in \mathcal{K}^y  &:  \langle A\alpha_1 -d, \alpha_1- v \rangle \leq 0 \quad \forall v \in \mathcal{K}^y.%\\
%\mathcal{K}^y &:= \mathcal{T}_{\mathbf K(y)}(y) \cap [f-Ay]^\perp %\{ w \in V : w \leq 0 \text{ q.e. on } \mathcal{A}(y) \text{ and } \langle f-Ay, w \rangle = 0\}.
\end{aligned}
\end{equation*}
To acquire an expansion formula for a general $y_n^s$, %since we assumed that $\Phi$ is Hadamard differentiable at $y$, 
define  for $n >1$,
\begin{align*}
%\delta_n &= \partial S(f,y)(d-A\Phi'(y)(\alpha_{n-1})),\\
\alpha_n &:=\Phi'(y)(\alpha_{n-1}) + \partial S(f,y)(d-A\Phi'(y)(\alpha_{n-1})).
%\alpha_1 = \delta_1
\end{align*}
%\begin{align}\label{eq:defnDeltan}
%\nonumber \delta_n &:= \partial S(f,y)\left(d-A\Phi'(y)(\Phi'(y)(...\Phi'(y)(\Phi'(y)(\delta_0)+\delta_1) + \delta_2...) + \delta_{n-2}) + \delta_{n-1})\right)\quad \text{for $n>1$}
%\end{align}
%and 
%\begin{equation*}%\label{eq:alphan}
%\begin{aligned}
%%\alpha_1 &:= -\delta_0\\
%\alpha_n &:= 
%\begin{cases}
%\delta_1&:\text{if $n =1$}\\
%\Phi'(y)[\Phi'(y)[...\Phi'(y)[\Phi'(y)(\delta_1)+\delta_2] + \delta_3...] + \delta_{n-1}] + \delta_{n} &:\text{if $n \geq 2$},
%\end{cases}
%\end{aligned}
%\end{equation*}
In exactly the same way as in \cite[Proposition 2]{AHR}, we obtain the following result. %expansion formula \eqref{eq:expansionuns} and the convergence behaviour%
%\begin{equation}\label{eq:expansionuns}
%\[y_n^s = u + s\alpha_n + o_n(s)\]
%\end{equation}
%where 
%$t^{-1}o_n(t) \to 0$ as $t \to 0^+$ if \eqref{item:assPhiDerivativeBounded} holds and where $\alpha_n$ satisfies the stated VI.
\begin{prop}Under \eqref{ass:PhiLipBound} and \eqref{ass:PhiHadamardAtY}, for $n \geq 1$, 
\begin{equation}\label{eq:expansionuns}
y_n^s = y + s\alpha_n + o_n(s)
\end{equation}
where $s^{-1}o_n(s) \to 0$ as $s \to 0^+$ and $\alpha_n = \alpha_n(d)$ is positively homogeneous in the direction $d$ and satisfies the VI
\begin{equation*}
\begin{aligned}
\alpha_n \in \mathcal{K}^y(\alpha_{n-1}) &: \langle A\alpha_n - d, \alpha_n - \varphi \rangle \leq 0 \quad \forall \varphi \in \mathcal{K}^y(\alpha_{n-1}),\\
\mathcal{K}^y(\alpha_{n-1}) &:= \mathcal{K}^y + \Phi'(y)(\alpha_{n-1}). %\{ \varphi \in V : \varphi \leq \Phi'(y)(\alpha_{n-1})\text{ q.e. on }\mathcal{A}(y) \text{ and } \langle Ay-f, \varphi-\Phi'(y)(\alpha_{n-1}) \rangle = 0\},
\end{aligned}
\end{equation*}
%If we assume that for any $b \in V$, $h\colon(0,1) \to V$ and $\lambda \in [0,1]$, 
%\begin{equation}
%s^{-1}\norm{\Phi'(y+sb+\lambda h(s))(h(s))}{V} \to 0\quad\text{as $s \to 0^+$ whenever $s^{-1}h(s) \to 0$ as $s \to 0^+$}\label{item:assPhiDerivativeBounded},\tag{H6a}
%\end{equation}
%with 
\end{prop}
%\begin{proof}
%\begin{align}
%\nonumber \alpha_n \in \mathcal{K}^u(\alpha_{n-1}) &: \langle A\alpha_n - d, \alpha_n - \varphi \rangle \leq 0 \qquad \forall \varphi \in \mathcal{K}^u(\alpha_{n-1})\\
%\mathcal{K}^u(\alpha_{n-1}) &:= \{ \varphi \in V : \varphi \leq \Phi'(u)(\alpha_{n-1})\text{ q.e. on }\mathcal{A}(u) \text{ and } \langle Au-f, \varphi-\Phi'(u)(\alpha_{n-1}) \rangle = 0\}.\label{eq:alphanTestFunctionSpace}
%\end{align}
%\end{proof}
See \eqref{eq:defnOn} for the precise definition of $o_n$. The proof of this proposition, which we omit here, is by induction and makes use of the expansion formula of \cite[Proposition 1]{AHR}, which tells us that
\[y_{n+1}^s = S(f+sd,y+s\alpha_n + o_n(s)) = y + s(\Phi'(y)(\alpha_n) + \partial S(f,y)(d-A\Phi'(y)(\alpha_n))) + o_{n+1}(s).\]
%where $o_n(s)$ is a remainder term.
It remains then to pass to the limit in \eqref{eq:expansionuns} and to identify the corresponding limits. 
\subsection{Passage to the limit}\label{sec:ptl}
Observe that the conditions \eqref{ass:PhiLipBound} and \eqref{ass:PhiHadamardAtY}  imply that
\begin{equation}
\text{$\Phi'(y)\colon V \to V$ is Lipschitz with Lipschitz constant }C_L < {C_a}\slash{C_b}\label{ass:PhiDerivativeLipschitzConstant},
\end{equation}
which is precisely what is needed for the coming intermediary results. In particular, it allows for the Banach fixed point theorem to be amenable to show the convergence of $\{\alpha_n\}$ as the next proposition demonstrates. But first, let us prove that \eqref{ass:PhiDerivativeLipschitzConstant} is indeed a consequence. From the expansion formula
$\Phi(y+sh)=\Phi(y) + s\Phi'(y)(h) + o(s;h)$ where $o(\cdot,h)$ is a remainder term, we find
\begin{align*}
\norm{\Phi'(y)(h)-\Phi'(y)(d)}{V} &\leq \frac 1s\norm{\Phi(y+sh)-\Phi(y+sd)}{V}  + \frac 1s\norm{o(s;d)-o(s;h)}{V}.
\end{align*}
Without loss of generality, we may assume that at least one of $h$ and $d$ is non-zero. We see that if $s \leq \epsilon\slash(\norm{h}{V} + \norm{d}{V})$, 
%\[\norm{sh}{V} \leq s\norm{h}{V} \leq \epsilon\]
%and hence, for such $s$, 
we have  $y+sh, y+sd \in B_\epsilon(y)$ and therefore, by \eqref{ass:PhiLipBound},
\begin{align*}
\norm{\Phi'(y)(h)-\Phi'(y)(d)}{V} &\leq C_\Phi\norm{h-d}{V}  + \frac 1s\norm{o(s;d)-o(s;h)}{V}.
\end{align*}
Taking $s \to 0^+$ we obtain the statement after noting that $C_\Phi < C_L$.

\begin{prop}\label{prop:strongConvAlphas}Under \eqref{ass:PhiHadamardAtY} and \eqref{ass:PhiDerivativeLipschitzConstant}, $\alpha_n \to \alpha$ in $V$ where $\alpha$ is the unique solution of the QVI \eqref{eq:QVIforAlpha}.
\end{prop}
\begin{proof}
Denote by $T\colon V \to V$ the solution map $\gamma \mapsto \beta$  of the inequality
\[\beta \in \mathcal{K}^y(\gamma) : \langle A\beta - d, \beta - \varphi \rangle \leq 0 \quad \forall \varphi \in \mathcal{K}^y(\gamma).\]
%\[\beta -\Phi'(y)(\gamma) \in \mathcal{K}^y : \langle A\beta - d, \beta - \Phi'(y)(\gamma) - \varphi \rangle \leq 0 \qquad \forall \varphi \in \mathcal{K}^y.\]
%Note that $\beta$ is a solution if and only if $\hat \beta := \beta - \Phi'(y)(\gamma)$ satisfies
%\[\hat \beta \in \mathcal{K}^y : \langle A\hat \beta + A\Phi'(y)(\gamma) - d, \hat \beta - \varphi \rangle \leq 0 \quad \forall \varphi \in \mathcal{K}^y.\]
%This latter inequality 
This has a unique solution by the Lions--Stampacchia theorem \cite[\S 4:3, Theorem 3.1]{Rodrigues},  %therefore as does the former by setting $\beta := \hat \beta + \Phi'(y)(\gamma)$. 
hence $T$ is well defined.

Consider $\gamma_1, \gamma_2 \in V$ with $\beta_1 := T(\gamma_1)$ and $\beta_2 := T(\gamma_2)$. Testing the inequality for $\beta_1$ with the feasible element $\beta_2-\Phi'(y)(\gamma_2) + \Phi'(y)(\gamma_1)$ and vice versa and then combining both of the resulting inequalities, we find
\[\langle A(\beta_1 - \beta_2), \beta_1-\beta_2 + \Phi'(y)(\gamma_2)- \Phi'(y)(\gamma_1) \rangle \leq 0,\]
which implies, using \eqref{ass:PhiDerivativeLipschitzConstant},
\begin{align*}
\norm{\beta_1-\beta_2}{V} &\leq \frac{C_b}{C_a}\norm{\Phi'(y)(\gamma_2)-\Phi'(y)(\gamma_1)}{V} %\leq \frac{C_LC_b}{C_a}\norm{\gamma_2-\gamma_1}{V}.
< \norm{\gamma_2-\gamma_1}{V}.
\end{align*}
%using the assumption. 
This shows that $T\colon V \to V$ is a contraction. 
%i.e., $R$ is single valued. Using the assumption,
%\[\norm{\beta_1-\beta_2}{V} \leq \frac{C_\Phi C_b}{C_a}\norm{z-w}{V}\]
%and hence $R\colon V \to V$ is a contraction. 
Therefore, thanks to the Banach fixed point theorem, the iterative sequence $\beta_n := T(\beta_{n-1})$, $\beta_1 := \alpha_1$, is such that $\beta_n \equiv \alpha_n$ (by uniqueness of solutions) and $\alpha_n \to \alpha$ strongly in $V$ where $\alpha$ is the fixed point of $T$.
%This strong convergence opens the door for the characterisation of the directional derivative. %It allows us to pass to the limit  in the recurrence formula \eqref{eq:alphaNRecursion} which involves the terms $\alpha_n$ and $\alpha_{n-1}$ (for which arguments using convergences of subsequences would not be viable) and we get $\delta = \alpha - \Phi'(y)(\alpha)$. 
%Noting that $\alpha_n - \Phi'(y)(\alpha_{n-1}) \in \mathcal{K}^y$ and rewriting  \eqref{eq:alphanTestFunctionSpace} as
%\begin{align*}
%\nonumber  &\langle A\alpha_n - d, \alpha_n - \Phi'(y)(\alpha_{n-1}) - \varphi \rangle \leq 0 \qquad \forall \varphi \in \mathcal{K}^y,%\\
%%\mathcal{K}^y&:= \{ \varphi \in V : \varphi \leq 0 \text{ q.e. on }\mathcal{A}(y) \text{ and } \langle Ay-f, \varphi\rangle = 0\},
%\end{align*}
%it is easy to pass to the limit, and $\alpha-\Phi'(y)(\alpha) \in \mathcal{K}^y$ follows by closedness of the critical cone. This shows that $\alpha$ solves the QVI stated in the theorem. 
\end{proof}

%To this end, observe that $s\alpha_n + o_n(s) = y_n^s - y  \to y^s-y$ in $V$. Testing the inequality for $\alpha_n$ with $\varphi=\Phi'(y)(\alpha_{n-1})$ and using assumption \eqref{itm:smallnessOfDerivOfPhi}, which  provides the existence of a constant $c >0$ such that
%\begin{equation*}\label{ass:ForAlphaNBound}
%\norm{\Phi'(y)(v)}{V} \leq \frac{C_a-c}{C_b}\norm{v}{V},
%\end{equation*}
%we find that the sequence $\{\alpha_n\}$ is bounded exactly as shown in  the proof of \cite[Theorem 6]{AHR} and we have the existence of a subsequence $\{n_j\}$ with 
%\begin{align*}
%\alpha_{n_j} &\weaklyto \alpha \text{ in $V$}\qquad\text{and}\qquad  o_{n_j}(s) \weaklyto o^*(s) \text{ in $V$.}
%\end{align*}

Thanks to this result, it follows that $o_n(s) \to o^*(s)$ in $V$ for some $o^*(s)$. We can send $n \to \infty$ in \eqref{eq:expansionuns}  to obtain
\begin{equation*}
y^s = y + s\alpha + o^*(s),
\end{equation*}
and it is left for us to show that $o^*$ is a remainder term. The idea in \cite{AHR} was to show that the convergence $s^{-1}o_n(s) \to 0$ as $s \to 0^+$ is \textit{uniform} in $n$, which is sufficient to commute the limits $s \to 0^+$ and $n \to \infty$ for $s^{-1}o_n(s)$, giving the desired behaviour $s^{-1}o^*(s) \to 0$ as $s \to 0^+$.  
This was done in \cite[Lemma 14]{AHR}, the proof of which we will now adapt under the context of our current (more general) setting.  For this, we need some more notation. 
For $v \in V$ and $h_s \in V$, we define the remainder term associated to $\Phi$ 
\begin{equation}\label{eq:defnHatl}
\hat l(s,h,h_s;v) := \Phi(v+sh_s) - \Phi(v) - s\Phi'(v)(h),
\end{equation}
and since $\Phi$ is Hadamard differentiable at $y$, if $h_s \to h$ in $V$ as $s \to 0^+$, then $s^{-1}\hat l(s,h,h_s;y) \to 0$ as $s \to 0^+$. We write $\hat l(s,h,h;v) = l(s,h;v)$  when $h_s \equiv h$. Now let $S_0\colon V^* \to V$ be the map $f \mapsto u$ of the following VI with zero lower obstacle:
\begin{equation*}
u \in V_+ : \langle Au -f, u-v \rangle \leq 0 \quad \forall v \in V_+.
\end{equation*}
In a similar fashion to $\hat l$, we denote the remainder term associated to the expansion formula of $S_0$ by $\hat o$:
%\[o(s,h;f):=\frac{S_0(f+sh)-S_0(f)-sS_0'(f)(h)}{s}.\]
\[\hat o(s,h, h_s;f):=S_0(f+sh_s)-S_0(f)-sS_0'(f)(h).\]
\begin{prop}\label{lem:convergenceOfOstar}
Under \eqref{ass:PhiLipBound} and \eqref{ass:PhiHadamardAtY}, % and \eqref{ass:PhiDerivativeLipschitzConstant}, 
%\item \label{itm:PhiDiff} %We assume that 
%The map $\Phi\colon V \to V$ is Hadamard directionally differentiable. %That is, for all $v$ and all $h$ in $V$, the limit
%\[\lim_{\substack{h' \to h\\t \to 0^+}}\frac{\Phi(v+t h')-\Phi(v)}{t}\]
%exists in $V$, and we write the limit as $\Phi'(v)(h)$. Hence, if $h(t) \to h$, then
%\begin{equation}\label{eq:derivativeOfPhi}
%\Phi(v+th(t)) = \Phi(v) + t\Phi'(v)(h) + \hat l(t,h,h(t),v)
%\end{equation}
%holds where $\hat l$ is a higher order term, i.e.,  $t^{-1}\hat l(t,h,h(t),v) \to 0$ as $t \to 0^+$. We write $\hat l(t,h,h,v) = l(t,h,v)$  when $h(t) \equiv h$.
%\end{enumerate}
$s^{-1}o^*(s) \to 0$ as $s \to 0^+$. 
\end{prop}
\begin{proof}
%First, let us show that \eqref{ass:PhiLipBound} implies
%%and \ref{itm:PhiDiff} hold and let $b\colon (0,T) \to V$ satisfy $t^{-1}b(t) \to 0$ as $t \to 0^+$. Then
%%Then
%\begin{equation}
%\norm{\hat l(s,\alpha_n, \alpha_n + s^{-1}o_n(s);y)}{V}  \leq C_\Phi\norm{o_n(s)}{V} + \norm{l(s,\alpha_n;y)}{V}.
%\end{equation}
%and thus
%\[\frac{\norm{l(t,h,h + t^{-1}b(t,w))}{V}}{t} \to 0\quad \text{as $t \to 0^+$ uniformly in $h$ and $w$ on the bounded subsets of $V$.}\]
%Here $h$ could be $d$ and thus $C$ could be an arbitrary compact subset of $V^*$.
%\end{lem}
%\begin{proof}
Since $y+s\alpha_n + o_n(s) = y_n^s \in B_\epsilon(y)$ and $\{\alpha_n\}$ is bounded, let us say by $M$, if $s \leq M^{-1}\epsilon$, then $y+s\alpha_n \in B_\epsilon(y)$ too. Hence, from \eqref{eq:defnHatl} and the Lipschitz property \eqref{ass:PhiLipBound}, we have
%it follows that, since $y_n^s \to y^s$ in $V$ as $n \to \infty$, % and $\alpha_n$ is uniformly bounded, by a constant $M$, we have
%\begin{align*}
%\norm{s\alpha_n + \lambda o_n(s)}{V} %&= \norm{\lambda (y_n^s-y) + (1-\lambda)(s\alpha_n)}{V}\\
%&\leq \frac{\epsilon}{2} + sM
%&\leq \epsilon
%\end{align*}
%so that for $s$ sufficiently small, we have $y+s\alpha_n + \lambda o_n(s) \in B_\epsilon(y)$. Thus
\begin{align} 
\nonumber \norm{\hat l(s,\alpha_n, \alpha_n + s^{-1}o_n(s);y)}{V}  &\leq \norm{\hat l(s,\alpha_n, \alpha_n + s^{-1}o_n(s);y) - l(s,\alpha_n;y)}{V} + \norm{l(s,\alpha_n;y)}{V}\\
\nonumber  &= \norm{\Phi(y+s(\alpha_n + s^{-1}o_n(s)))-\Phi(y+s\alpha_n)}{V} + \norm{l(s,\alpha_n;y)}{V}\\
 &\leq C_\Phi\norm{o_n(s)}{V} + \norm{l(s,\alpha_n;y)}{V}.\label{eq:lBound}
\end{align}
%and the reverse triangle inequality yields the result.
%\begin{align*}
%\frac{1}{t}\norm{\hat  l(t,h, h + t^{-1}b(t,w))}{V} &\leq \sup_{\lambda \in [0,1]}\norm{\Phi'(u+th + \lambda b(t,w))}{op}\norm{\frac{b(t,w)}{t}}{H^{1}} + \norm{\frac{l(t,h)}{t}}{V}
%\end{align*}
%It follows that $t^{-1} l(t,h, h + t^{-1}b(t,w)) \to 0$ uniformly on bounded subsets due to the properties of $l$ that we assumed (I think (A4)).
%\end{proof}
We see from \cite[Equation (34) and Proposition 1]{AHR} that  $o_n$ has the definition
\begin{align}
\nonumber o_n(s) &:=\hat l(s,\alpha_{n-1},\alpha_{n-1}+s^{-1}o_{n-1}(s); y)\\
&\quad - \hat o(s,A\Phi'(y)(\alpha_{n-1})-d, A\Phi'(y)(\alpha_{n-1})-d + As^{-1}\hat l(s,\alpha_{n-1},\alpha_{n-1}+s^{-1}o_{n-1}(s)); A\Phi(y)-f).\label{eq:defnOn}
\end{align}
For ease of reading, let us omit the base point from the expressions for $\hat l, l, \hat o$ and $o$ from now on. That is, we write $\hat l(\cdot,\cdot,\cdot)$ instead of $\hat l(\cdot,\cdot,\cdot; y)$ and likewise for the other terms. In the above equality, taking norms and, on the right-hand side, using \eqref{eq:lBound} on the first term and the corresponding estimate
\begin{equation*}
\norm{\hat o(s,h, h + s^{-1}h_s)}{V} \leq C_a^{-1}\norm{h_s}{V^*} + \norm{o(s,h)}{V}
\end{equation*} 
for $S_0$ and its remainder term (see \cite[Lemma 1]{AHR}) on the second term, we find
\begin{align*}
\norm{o_n(s)}{V} %&\leq C_\Phi\norm{o_{n-1}(t)}{V} + \norm{l(t,\alpha_{n-1})}{V} +C_a^{-1}\norm{A\hat l(t,\alpha_{n-1},\alpha_{n-1}+t^{-1}o_{n-1}(t))}{V^*} + \norm{o(t,A\Phi'(y)(\alpha_{n-1})-d)}{V}\\
&\leq C_\Phi\norm{o_{n-1}(s)}{V} + \norm{l(s,\alpha_{n-1})}{V} + C_a^{-1}C_b\norm{\hat l(s,\alpha_{n-1},\alpha_{n-1}+s^{-1}o_{n-1}(s))}{V}+ \norm{o(s,A\Phi'(y)(\alpha_{n-1})-d)}{V}\\
&\leq C_\Phi(1+C_a^{-1}C_b)\norm{o_{n-1}(s)}{V} + (1+C_a^{-1}C_b)\norm{l(s,\alpha_{n-1})}{V} + \norm{o(s,A\Phi'(y)(\alpha_{n-1})-d)}{V},%\tag{using again \eqref{eq:lBound}, now on the third term}.
%&< C\norm{o_{n-1}(t)}{V} + (1+C_a^{-1}C_b)\norm{l(t,\alpha_{n-1})}{V} + \norm{o(t,A\Phi'(y)(\alpha_{n-1}))}{V}.
\end{align*}
where we again used \eqref{eq:lBound} on the penultimate term in the first line to obtain the second inequality. Defining
\[a_{n}(s):=\norm{o_{n}(s)}{V} \qquad \text{and}\qquad b_{n}(s):= (1+C_a^{-1}C_b)\norm{l(s,\alpha_{n})}{V} + \norm{o(s,A\Phi'(y)(\alpha_{n})-d)}{V},\]
the above can be recast as
\[a_n(s) \leq Ca_{n-1}(s) + b_{n-1}(s)\]
for some $C<1$ by the assumption on $C_\Phi$ in \eqref{ass:PhiLipBound}.  
%This recursion formula implies %(we omit the argument $t$ here)
%which can be solved for $a_n$ in terms of $a_1$ and $\{b_i\}_{i=1}^{n-1}$:
%\begin{align}
%a_n(t) %&\leq C(Ca_{n-2} + b_{n-2}) + b_{n-1}\\
%\nonumber &= C^2a_{n-2} + Cb_{n-2} + b_{n-1}\\
%\nonumber &\leq C^2(Ca_{n-3} + b_{n-3}) + Cb_{n-2} + b_{n-1}\\
%\nonumber &= C^3a_{n-3} + C^2b_{n-3} + Cb_{n-2} + b_{n-1}\\
%\nonumber &\leq ...\\
%&\leq C^{n-1}a_{1}(t) + C^{n-2}b_1(t) + C^{n-3}b_2(t) + ... + Cb_{n-2}(t) + b_{n-1}(t)\label{eq:prelim2}.
%\end{align}}
%From the proof of \cite[Lemma 14]{AHR}, we see that $a_n$ satisfies the following recurrence inequality:
%\begin{align*}
%\norm{o_{n_j}(s)}{V} &< C\underbrace{\norm{o_{n_j-1}(s)}{V}}_{=:a_{n_j-1}(s)} + \underbrace{(1+C_a^{-1}C_b)\norm{l(s,\alpha_{n_j-1})}{V} + \norm{o(s,A\Phi'(y)(\alpha_{n_j-1}))}{V}}_{=:b_{n_j-1}(s)}
%\end{align*}
%\[a_{n}(s) \leq Ca_{n-1}(s) + b_{n-1}(s),\]
%where the constant $C<1$ by the assumption on $C_\Phi$ in \eqref{itm:smallnessOfDerivOfPhi}. This implies
Solving this recurrence inequality gives
\begin{align}
a_{n}(s) &\leq C^{n-1}a_{1}(s) + C^{n-2}b_1(s) + C^{n-3}b_2(s) + \hdots + Cb_{n-2}(s) + b_{n-1}(s)\label{eq:prelim2}.
\end{align}
Now, consider
\[\frac{b_{n-1}(s)}{s} = \frac{(1+C_a^{-1}C_b)\norm{l(s,\alpha_{n-1})}{V}}{s} + \frac{\norm{o(s,A\Phi'(y)(\alpha_{n-1})-d)}{V}}{s}.\]
By Proposition \ref{prop:strongConvAlphas}, $\alpha_n \to \alpha$ strongly in $V$, thus $\cl{\{\alpha_{n-1}\}}$ and $\cl{\{A\Phi'(y)(\alpha_{n-1})-d\}}$ are compact sets in $V$ and $V^*$ respectively. Since the remainder terms $l$ and $o$ appearing in the displayed equality above arise from the Hadamard (and hence compact) differentiability of $\Phi$ and the solution map $S_0$ associated to VIs, it follows that $l(s,\gamma)\slash s$ and $o(s,h)\slash s$ both converge to zero uniformly for $\gamma$ and $h$ belonging to   $\cl{\{\alpha_{n-1}\}}$ and $\cl{\{A\Phi'(y)(\alpha_{n-1})-d\}}$ respectively. Because $\{\alpha_{n-1}\} \subset \cl{\{\alpha_{n-1}\}}$ and $\{A\Phi'(y)(\alpha_{n-1})-d\} \subset \cl{\{A\Phi'(y)(\alpha_{n-1})-d\}}$, we have that
\begin{align*}
\frac{l(s,\gamma)}{s} &\to 0 \text{ uniformly in $\gamma \in \{\alpha_{n-1}\}$} \qquad \text{and}\qquad \frac{o(s,h)}{s} \to 0 \text{ uniformly in $h \in \{A\Phi'(y)(\alpha_{n-1})-d\},$}
\end{align*}
which then gives 
\[\frac{b_{n-1}(s)}{s} \to 0 \quad\text{uniformly in $n$}.\]
This, along with \eqref{eq:prelim2} and the geometric series estimate
$C^{n-2} + C^{n-3} + \hdots + C + 1 = (1-C^{n-1})\slash (1-C) \leq 1\slash(1-C)$ implies that for every $\epsilon > 0$, there exists an $s_0$ independent of $n$ such that
\[\frac{\norm{o_n(s)}{V}}{s} \leq \epsilon \quad \text{when $s \leq s_0$}\]
which means precisely that $s^{-1}o_n(s) \to 0$ as $s \to 0^+$ uniformly in $n$. Finally, recalling that $o_n(s)$ converges in $V$, taking the limit as $n \to \infty$ in the above inequality, we deduce that $s^{-1}o^*(s) \to 0$ as $s \to 0^+$. %
%The uniform convergence allows us to commute the limits $j \to \infty$ and $s \to 0$ and deduce that $o^*(s)\slash s \to 0$ as $s \to 0$, showing that $\alpha$ is indeed a directional derivative.
\end{proof}
This concludes the proof of Theorem \ref{thm:dirDiff1}.
\begin{remark}
 It is worth noting that the complete continuity assumption \eqref{ass:PhiCC} is not needed for the result (the strong convergence of $\{y_n^s\}$ assured by the application of the Banach fixed point theorem allowed us to circumvent complete continuity). Furthermore, complete continuity of $\Phi'(y)$ is not needed for the characterisation of the directional derivative; continuity suffices (which is guaranteed since Hadamard derivatives are continuous with respect to the direction), unlike in \S 5.1 and \S 5.2 of \cite{AHR}.
\end{remark}

 \subsection{Continuity properties of the directional derivative}\label{sec:continuityOfDirDer}
We now study the conditions under which continuity of the map taking the direction $d$ into the directional derivative $\alpha$ in \eqref{eq:QVIforAlpha} is assured. We recall \eqref{eq:QVIforAlpha} for convenience:
\begin{equation*}
\begin{aligned}
\alpha \in \mathcal{K}^y(\alpha) &: \langle A\alpha - d, \alpha - v \rangle \leq 0 \quad \forall v \in \mathcal{K}^y(\alpha),\\
\mathcal{K}^y(w) &:= \mathcal{K}^y+ \Phi'(y)(w).  %\{ \varphi \in V : \varphi \leq \Phi'(y)(w) \text{ q.e. on }\mathcal{A}(y) \text{ and } \langle Ay-f, \varphi -\Phi'(y)(w)\rangle = 0\}.
\end{aligned}
\end{equation*}
% and \eqref{eq:otherQVIforAlpha} is assured. %In the next lemma, note that we do not require \eqref{ass:H5b} since we only need boundedness of $\Phi'(y)$.
\begin{prop}\label{lem:alphaUniquenessAndCty}
Under \eqref{ass:PhiDerivativeLipschitzConstant},  %solutions of the QVI \eqref{eq:QVIforAlpha} are unique and % and 
%\begin{align}
%\Phi'(y)\colon V \to V \text{ is completely continuous,} \label{itm:compContOfDerivOfPhi}
%\end{align}
%\eqref{itm:compContOfDerivOfPhi} %(or \eqref{ass:H5b} and \eqref{ass:LnewB}) 
%hold. 
$d \mapsto \alpha(d)$ is continuous from $V^*$ to $V$. That is, if $d_j \to d$ in $V^*$, then
\[\alpha_{j} \to \alpha \quad\text{in $V$}\]
where $\alpha_j$ and $\alpha$ are the solutions of \eqref{eq:QVIforAlpha} with source terms $d_j$ and $d$ respectively.% $\alpha$ is a solution of \eqref{eq:QVIforAlpha} %(or \eqref{eq:otherQVIforAlpha}) 
%with source term $d$.
%The map $d \mapsto \alpha$ where $\alpha$ is a solution of either of the QVIs \eqref{eq:QVIforAlpha} or \eqref{eq:otherQVIforAlpha} %, reproduced here,
%\begin{align}%\label{eq:QVIforAlpha}
%\alpha \in \mathcal{K}^y(\alpha) &: \langle A\alpha - h, \alpha -v \rangle \leq 0 \quad \forall v \in \mathcal{K}^y(\alpha),
%\end{align}
%is such that if $d_j \to d$ in $V^*$ then there exists a subsequence $\{n_j\}$ with $\alpha(d_{n_j}) \to \alpha(d)$ in $V$. %where $\alpha(h)$ is a solution of \eqref{eq:QVIforAlpha}.
\end{prop}
\begin{proof}
%We again prove just the case for the QVI \eqref{eq:QVIforAlpha} obtained in \S \ref{sec:Diff}.
%% First observe that \eqref{itm:smallnessOfDerivOfPhi} implies 
%%\begin{equation*}%\label{ass:PhiDerivativeBoundedLocally}
%\norm{\Phi'(y)(v)}{V} \leq C_\Phi\norm{v}{V}\quad\text{where $C_\Phi < (1+C_a^{-1}C_b)^{-1}$}.%\tag{C1}
%\end{equation*}
The element $\alpha_j$ associated to $d_j$ satisfies
\[\alpha_j \in\mathcal{K}^y(\alpha_j) : \langle A\alpha_j - d_j, \alpha_j -v \rangle \leq 0 \quad \forall v \in\mathcal{K}^y(\alpha_j).\]
%We choose $v=\Phi'( y)(\alpha_j)$ as a test function. This leads, via \eqref{ass:PhiDerivativeLipschitzConstant}, to 
%\begin{align*}
%C_a\norm{\alpha_j}{V}^2 &\leq \norm{d_j}{V^*}\norm{\alpha_j}{V}  + (C_b\norm{\alpha_n}{V}+\norm{d_j}{V^*})\norm{\Phi'( y)(\alpha_j)}{V}\\
%&\leq C\norm{\alpha_j}{V}  + (C_b\norm{\alpha_j}{V}+C)C_L\norm{\alpha_j}{V},%\tag{using the assumption}
%\end{align*}
%and we see that since $C_a - C_bC_L$ is strictly positive, $\{\alpha_j\}$ is bounded and %, we need
%%\[ \geq c_0 > 0,\]
%%which holds: since $C_\Phi < C_a\slash (C_a+C_b)$, the left-hand side of the above inequality is $C_a^2\slash (C_a+C_b)$. 
%we obtain, for a subsequence,
%\[\alpha_n \weaklyto \alpha \quad \text{in $V$}\]
%for some $\alpha \in V$ that we need to identify. We first prove that the above convergence is also strong.
Take $j,k \in \mathbb{N}$ and in the inequality for $\alpha_j$, take the test function $v = \alpha_k - \Phi'( y)(\alpha_k) + \Phi'( y)(\alpha_j)$ which is clearly feasible, whilst in the inequality for $\alpha_k$, set $v = \alpha_j - \Phi'( y)(\alpha_j) + \Phi'( y)(\alpha_k)$ to obtain
\begin{align*}
\langle A\alpha_j - d_j, \alpha_j -\alpha_k + \Phi'( y)(\alpha_k) - \Phi'( y)(\alpha_j)  \rangle &\leq 0,\\
\langle A\alpha_k - d_k, \alpha_k -\alpha_n + \Phi'( y)(\alpha_j) - \Phi'( y)(\alpha_k)  \rangle &\leq 0.
\end{align*}
Adding these, we find
\begin{align*}
\langle A(\alpha_j - \alpha_k) - (d_j-d_k), \alpha_j - \alpha_k + \Phi'( y)(\alpha_k) - \Phi'( y)(\alpha_n)  \rangle \leq 0,
\end{align*}
which implies, using \eqref{ass:PhiDerivativeLipschitzConstant},
\begin{align*}
C_a\norm{\alpha_j-\alpha_k}{V}^2 %&\leq  \langle d_j-d_k, \alpha_j - \alpha_k \rangle + \langle A(\alpha_k-\alpha_j) - (d_k-d_j), \Phi'( y)(\alpha_k) - \Phi'( y)(\alpha_j) \rangle\\
&\leq \norm{d_j-d_k}{V^*}\norm{\alpha_j - \alpha_k}{V}+ C_b\norm{\alpha_k-\alpha_j}{V}\norm{\Phi'( y)(\alpha_k) - \Phi'( y)(\alpha_j)}{V}\\
&\quad + \norm{d_k-d_j}{V^*}\norm{\Phi'( y)(\alpha_k) - \Phi'( y)(\alpha_j)}{V}\\
&\leq \norm{d_j-d_k}{V^*}\norm{\alpha_j - \alpha_k}{V} + C_bC_L\norm{\alpha_k - \alpha_j}{V}^2 + C_L\norm{d_k-d_j}{V^*}\norm{\alpha_k - \alpha_j}{V}.
\end{align*}
Manipulating, we find that $\{\alpha_j\}$ is a Cauchy sequence and thus there exists an $\alpha \in V$ with
\[\alpha_j \to \alpha \quad \text{in $V$}.\]
Now, in the inequality for $\alpha_j$, choose the test function $v_j := v-\Phi'( y)(\alpha) + \Phi'( y)(\alpha_j)$ where $v$ is such that $v \in \mathcal{K}^ y(\alpha)$. It follows that $v_j \to v$ in $V$. This allows us to pass to the limit and we get
\[\langle A\alpha -d, \alpha - v\rangle \leq 0 \quad \forall v \in \mathcal{K}^y(\alpha)\]
and it remains to be seen that $\alpha \in \mathcal{K}^y(\alpha)$, which is evident since the critical cone is closed.%, which also shows uniqueness of solutions to \eqref{eq:QVIforAlpha}.
%This is easy to do: the strong convergence of $\alpha_j$ in $V$ implies that $\alpha_j \to \alpha$ pointwise q.e. and we know that $\alpha_j - \Phi'( y)(\alpha_n) \leq 0$ on $\{ y = \Phi( y)\} \setminus A_n$ where $A_n$ is a set of capacity zero. Utilising the fact that a countable union of sets of capacity zero has capacity zero, we find $\alpha-\Phi'( y)(\alpha) \leq 0$ q.e. on $\{ y = \Phi( y)\}$. This shows that $\alpha$ solves the desired QVI.
\end{proof}
\subsection{Complementarity characterisation of the directional derivative}
We now look for an analogue of the complementarity characterisation of Proposition \ref{lem:complementarityForQVI} for the QVI \eqref{eq:QVIforAlpha} satisfied by the directional derivative. 
\begin{prop}\label{lem:complementarityCharacterisationDerivativeA}
The QVI \eqref{eq:QVIforAlpha} %satisfied by $\alpha$ 
is equivalent to the complementarity system
\begin{subequations}
\begin{align}
\alpha &- \Phi'( y)(\alpha) \in \mathcal{K}^y,\\
\xi_d &= d-A\alpha,\\
\xi_d &\in (\mathcal{K}^y)^\circ,\\
\langle \xi_d&, \Phi'(y)(\alpha) - \alpha \rangle = 0.\label{eq:ccalphaQVIOrtho}
\end{align}
\end{subequations}
\end{prop}
\begin{proof}
As noted above, $\alpha - \Phi'( y)(\alpha)$ belongs to the set $\mathcal{K}^y.$ Define $\xi_d:=d-A\alpha$ which by definition satisfies
\[\alpha - \Phi'(y)(\alpha) \in \mathcal{K}^y : \langle \xi_d, \alpha - v\rangle \geq 0 \quad \forall v \in V : v - \Phi'( y)(\alpha) \in \mathcal{K}^y.\]
Taking $v=\Phi'( y)(\alpha)$ here and then $v=2\alpha - \Phi'( y)(\alpha)$ (which is feasible since $v-\Phi'( y)(\alpha)$ is twice a function that belongs to $\mathcal{K}^y$) shows the orthogonality condition \eqref{eq:ccalphaQVIOrtho}.

Let $w \in \mathcal{K}^y$ and select $v=\alpha +w$ (this is feasible since $v-\Phi'( y)(\alpha) = \alpha - \Phi'(y)(\alpha) + w \in \mathcal{K}^y + \mathcal{K}^y$ and the tangent cone, being a convex cone, is closed under addition). With this choice, we obtain
\[\langle \xi_d, w \rangle \leq 0 \quad \forall w \in \mathcal{K}^y,\]
meaning precisely that $\xi_d \in (\mathcal{K}^y)^\circ$.  The reverse direction holds by the same trick as in the proof of Proposition \ref{lem:complementarityForQVI}.
\end{proof}
\subsection{Examples of QVIs with multiple solutions}\label{sec:examplesMulti}
In this section, we construct explicit examples of QVIs with non-unique solutions such that the assumptions of Theorem \ref{thm:dirDiff1} are satisfied, thus verifying that multiplicity of solutions is not lost under our assumptions. %(showing in particular that the assumptions do not imply uniqueness of solutions).

\paragraph{Example 1} Let $\Omega \subset \mathbb{R}^n$ be a bounded Lipschitz domain  and set $V:=H^{k}(\Omega)$ with $H=L^2(\Omega)$ forming a Gelfand triple.  %For $n=1, \hdots, N$, %we define $y_n \in V$ as the solution of
%\[Ay_n = n.\]
Below, all norms and inner products that appear are over $H$. 

Pick $\delta>0$ and select a sequence $\{y_n\}_{n=1}^N$ of smooth functions satisfying $\norm{y_n-y_m}{}^2 > 4\delta^2$ for each $m,n \in \{1, \hdots, N\}$ with $m\neq n$ and $N \geq 2$ fixed. 
%Take $f$ such that $f \geq N$. 
Take a smooth cutoff function $\nu \in C^\infty(\mathbb{R})$ with $0 \leq \nu \leq 1$ and
\[\nu(t) = \begin{cases}
1 &: \text{if $t \in (-\delta^2, \delta^2)$},\\
0 &: \text{if $|t| \geq 2\delta^2$}.
\end{cases}
\]
For a parameter $y \in V$, define the map $\Phi_y\colon V \to V$ by
\[\Phi_y(u) := \nu(\norm{u-y}{}^2)y\] and set 
\[\Phi(u) := \sum_{n=1}^N \Phi_{y_n}(u).\]
Note that $\Phi\colon V \to V$ and  $\Phi(y_n) =y_n$ (because $\Phi_{y_n}(y_m) = y_n\delta_{nm}$).  Let the elliptic operator $A\colon V \to V^*$ have the property that $Ay_n \in H$ for each $n$  and define the pointwise a.e. maximum $f:=\max(Ay_1, \cdots, Ay_N) \in H$. Then the QVI
\[\text{find } u \leq \Phi(u) : \langle Au - f, u-v \rangle \leq 0 \quad \forall v \in V : v \leq \Phi(u)\]
has multiple solutions and indeed each $y_n \in \mathbf{Q}(f)$ is a solution. To see this, simply observe that $Ay_n-f \leq 0$ and $y_n-v = \Phi(y_n)-v \geq 0$ for all $v \in V$ with $v \leq \Phi(y_n)$.

%$Ay_n-f = n-f \leq 0$ and $y_n-v \geq y_n - \Phi(y_n) = y_n - y_n = 0$ so the inequality is satisfied.

It follows from the expression $\Phi_y'(u)(h) = 2y\nu'(\norm{u-y}{}^2)(h,u-y)$ that
\[\Phi'(u)(h) = \sum_{n=1}^N 2y_n\nu'(\norm{u-y_n}{}^2)(h,u-y_n)\]
and hence $\Phi'(B_\delta(y_n))\equiv 0$ and thus \eqref{itm:smallnessOfDerivOfPhi} is trivially satisfied (hence also \eqref{ass:PhiLipBound} and \eqref{ass:PhiDerivativeLipschitzConstant} by Remark \ref{rem:remarks} (ii) and the digression at the start of \S \ref{sec:ptl}).
%Now, fix $m$. From
%\begin{align}
%\frac{\Phi_{y_n}(y_m+sh)-\Phi_{y_n}(y_m) - s\Phi_{y_n}'(y_m)(h)}{s} &= \frac{\nu(\norm{y_m+sh-y_n}{}^2)y_n - y_n\delta_{nj}}{s} \label{eq:quotientToRef}%=\frac{(\nu(s^2\norm{h}{}^2)-1)y_n}{s},
%\end{align}
%we see that if $h$ belongs to a bounded set of elements with norm not exceeding $K$, denoting $M:=\max_{n}\norm{y_m-y_n}{}$, we have
%%$s^2\norm{h}{}^2 \leq s^2K^2$ so 
%%if $s \leq \epsilon\slash (K^2+2MK)$, 
%using
%\begin{align*}
%\left|s^2\norm{h}{}^2 + 2s(y_m-y_n,h)\right| &\leq s^2\norm{h}{}^2 + 2s\norm{y_m-y_n}{}\norm{h}{}\\
%&\leq sK^2 + 2sMK\tag{for $s \leq 1$}
%\end{align*}
%that, if $s \leq \min(1,2\delta^2\slash (K^2+2MK))$ and assuming for now that $n \neq m$,
%\begin{align*}
%\norm{y_m+sh-y_n}{}^2 &= \norm{y_m-y_n}{}^2 + s^2\norm{h}{}^2 + 2s(y_m-y_n,h)\\
%&\geq 4\delta^2 - s(K^2+2MK)\\
%&\geq 2\delta^2.
%\end{align*}
%Hence, for such $s$ and $n \neq m$, the quotient \eqref{eq:quotientToRef} equals zero. If $n=m$, then 
%\begin{align*}
%\frac{\Phi_{y_n}(y_m+sh)-\Phi_{y_n}(y_m) - s\Phi_{y_n}'(y_m)(h)}{s} &= \frac{(\nu(s^2\norm{h}{}^2)-1)y_n}{s}
%\end{align*}
%which also vanishes if $s \leq \delta\slash K$.
%%and hence
%%\begin{align*}
%%\frac{\Phi(y_n+sh)-\Phi(y_n) - s\Phi'(y_n)(h)}{s} &= \frac{(\nu(s^2\norm{h}{}^2)-1)y_n}{s}
%%\end{align*}
%It follows that $\Phi$ is bounded differentiable at $y_m$. 
Hence, all the requirements of Theorem \ref{thm:dirDiff1} have been met and we obtain for every $d \in V^*$ the existence of of $y_m^s \in \mathbf{Q}(f+sd)$ and $\alpha_m \in V$ such that
\[\lim_{s \to 0^+}\frac{y_m^s-y_m}{s} = \alpha_m.\]
 Let us also note that in addition, $\Phi'(y_n)\colon V \to V$ is completely continuous.  
\paragraph{Example 2}
A second example, without the need for the source term $f$ to be defined in terms of $\{y_n\}$, can be given under the same initial setting as above. For $n=1, \hdots, N$, take $\psi_n \in V$ to be given distinct obstacles such that the associated solutions $y_n \in V$ of the VIs 
\[y_n \leq \psi_n : \langle Ay_n - f, y_n - v \rangle \leq 0 \quad \forall v \in V : v \leq \psi_n\]
are distinct too. We suppose that $\delta$ is chosen such that $\norm{y_n-y_m}{}^2 > 4\delta^2$, which is possible since the $y_n$ are distinct functions. With $\nu$ as above, define now $\Phi_n\colon V \to V$ by
\[\Phi_n(u) := \nu(\norm{u-y_n}{}^2)\psi_n\] and set 
\[\Phi(u) := \sum_{n=1}^N \Phi_{n}(u).\]
We have $\Phi(y_n) = \psi_n$ and each $y_n$ is again a solution of the QVI associated to $\Phi$ with source term $f$, i.e., $y_n \in \mathbf{Q}(f)$.  Furthermore, %It follows from the expression $\Phi_n'(u)(h) = 2\psi_n\nu'(\norm{u-y_n}{}^2)(h,u-y_n)$ that
\[\Phi'(u)(h) = \sum_{n=1}^N 2\psi_n\nu'(\norm{u-y_n}{}^2)(h,u-y_n)\]
and we can argue as before to derive the other properties and results.

\section{Existence of optimal controls}\label{sec:existenceOC}
We now address the optimal control problem \eqref{eq:ocProblem}. Regarding the function space context in this section, we take
\begin{enumerate}[label=(\roman*)]
\item $V \cts H$ to be a continuous embedding of reflexive Banach spaces, \vspace{-0.2cm}
\item $U$ to be a reflexive Banach space with $U \ctsCompact V^*$,\vspace{-0.2cm}
\item $U_{ad} \subseteq U$ to be a non-empty and weakly sequentially closed\footnote{That is, if $u_n \weaklyto u$ in $U$ with $u_n \in U_{ad}$, then $u \in U_{ad}$.} set.
\end{enumerate}
Given $\nu >0$ and a desired state $y_d \in H$, define $J\colon H \times U \to \mathbb{R}$ by
\[J(y,u) := \frac 12\norm{y-y_d}{H}^2 + \frac{\nu}{2}\norm{u}{U}^2,\] 
and consider the problem \eqref{eq:ocProblem} which we recall here:
\begin{equation*}%\label{eq:ocProblem}
\min_{\substack{u \in U_{ad}\\ y \in \mathbf Q(u)}}J(y,u).%\tag{{P\textsubscript{OC}}}
\end{equation*}
%We will suppose that $\mathbf{Q}$ is well defined (see \S \ref{sec:existence}).
\begin{theorem}\label{thm:existenceOC}
Let Assumption \ref{ass:seqReg} hold, suppose that $\mathbf{Q}(u)$ is non-empty\footnote{See \S \ref{sec:existence}.} for every $u \in U_{ad}$ and let the feasbility condition
%Let $U \ctsCompact V^*$ and $U_{ad} \subset U$ be a weakly sequentially closed set and 
\eqref{eq:assH1} and the complete continuity \eqref{ass:PhiCC} hold. Then there exists an optimal control $u^* \in U_{ad}$ and associated state $y^* \in \mathbf{Q}(u^*)$ to the problem \eqref{eq:ocProblem}.
\end{theorem}
\begin{proof}
%\textcolor{red}{HOW TO SHOW THAT $y^* \leq \Phi(y^*)$??? Need $V^+$ closed, so ass 2.12? Note that Mazur lemma implies $V^+$ is also weakly closed since it is convex, so $V^+$ is weakly closed iff norm closed...}
Let $u_n \in U_{ad}$ be an infimising sequence with $y_n \in \mathbf{Q}(u_n)$, i.e., \[J(y_n, u_n) \to \inf_{\substack{u \in U_{ad},\\ y \in \mathbf{Q}(u)}}J(y,u).\]
Then $\{u_n\}$ and $\{y_n\}$ are bounded in $U$ and $V$ respectively (the latter arises from \eqref{eq:assH1}) and therefore, there exists a subsequence such that
%\begin{align*}
\[u_{n_j} \weaklyto u^* \text{ in $U$} \qquad\text{and}\qquad y_{n_j} \weaklyto y^* \text{ in $V$}.\]
%\end{align*}
By assumption, $u^*$ also belongs to $U_{ad}$. Since the $y_n$ are solutions of QVIs, we have the following estimate
\[\norm{y_{n_j} - y_{n_k}}{V} \leq C\left(\norm{u_{n_j} - u_{n_k}}{V^*} + \norm{\Phi(y_{n_j})- \Phi(y_{n_k})}{V}\right).\]
In the limit, the first term on the right-hand side vanishes due to the compact embedding, and the second term vanishes too because $\Phi$ is completely continuous due to \eqref{ass:PhiCC}. Thus $\{y_{n_j}\}$ is Cauchy in $V$ and $y_{n_j} \to y^*$ in $V$. Taking an arbitrary $v \in V$ such that $v \leq \Phi(y^*)$, we set $v_{n_j} := v-\Phi(y^*) + \Phi(y_{n_j})$ and use this as a test function in the QVI for $y_{n_j}$ in which we can pass to the limit to find $y^* \in \mathbf Q(u^*)$. To see that this pair is optimal, we observe that (dispensing with the subsequence notation now), using the continuity of the embedding $V \cts H$,
\begin{align*}
J(y^*,u^*) &\leq \liminf_{n \to \infty} J(y_n, u_n) \leq \lim_{n \to \infty} J(y_n, u_n) = \min_{\substack{u \in U_{ad}\\ y \in \mathbf Q(u)}}J(y,u).\qedhere
\end{align*}
\end{proof}
Regarding regularity of the optimal control, see Theorem \ref{thm:ocPenalisation}. In general there is no uniqueness for the optimal control and state regardless of whether $\mathbf{Q}$ is single valued or not.  
%Note that the conditions of the theorem are satisfied with $U=L^2(\Omega)$ and $U_{ad} = U_+ = L^2_+(\Omega)$.       

%We will from now on denote by $(u^*, y^*)$ a minimiser of \eqref{eq:ocProblem}. 
\subsection{The penalised optimal control problem}\label{sec:penalisedOCProblem}
%We again assume that $V \cts H$ and $U \cts V^*$ for some reflexive Banach spaces $H$ and $U$. 
Let us return to the context of \S \ref{sec:penalisationForExistence} and consider for each $\rho>0$ the penalisation of \eqref{eq:ocProblem}:
\begin{equation}\label{eq:ocProblemPen}
\min_{u \in U_{ad}} J(y_\rho,u) \quad\text{such that}\quad  Ay_\rho + \frac 1\rho m_\rho(y_\rho-\Phi(y_\rho)) = u.
\end{equation}
We remind the reader that $m_\rho$ is taken to satisfy \eqref{ass:mrPenal}--\eqref{ass:mrForFeas}. Recalling the map $\mathbf P_\rho$ from \S \ref{sec:penalisationForExistence}, we can write the equation above as $y_\rho \in \mathbf{P}_\rho(u)$. The reason for considering this problem is because we will use this to derive stationarity conditions in the next section but first 
%\[\min_{\substack{u \in U_{ad}\\y_\rho \in \mathbf P_\rho(u)}} J(y_\rho,u).\]
%We write the optimal control of this as $(y_\rho^*, u_\rho^*)$.
%Denote the optimal solution of this problem as $(y_\rho^*, u_\rho^*)$ and the optimal solution of the original problem \eqref{eq:ocProblem} as $(y^*, u^*)$. 
let us check that this minimisation problem suitably approximates \eqref{eq:ocProblem}.
\begin{prop}\label{lem:convPenOCProblems}%Let $U_{ad} \subset U$ be non-empty and weakly sequentially closed and $U \ctsCompact V^*$.
Let Assumption \ref{ass:seqReg}, \eqref{ass:feasiblePoint}, \eqref{ass:mrCC} and \eqref{ass:PhiCC}  hold and suppose that $\mathbf{Q}$ is single valued. Then there exist optimal pairs $(y_\rho^*, u_\rho^*)$  of \eqref{eq:ocProblemPen} and an optimal pair $(y^*,u^*)$ of \eqref{eq:ocProblem} such that
 %$(y_\rho^*, u_\rho^*)$ converges to the optimal control $(y^*, u^*)$ of \eqref{eq:ocProblem}:  
\[(y_\rho^*, u_\rho^*) \to (y^*, u^*) \text{ in $V \times U$.}\]
%\quad\text{and}\quad \text{$y_\rho^* \weaklyto y^*$ in $V$}.\]
%where 
\end{prop}
\begin{proof}
First, observe that $\mathbf{P}_\rho(u)$ is non-empty for all $u \in U_{ad}$ by Proposition \ref{lem:existencePenalisedPDE} (after possibly renorming $V$, see Example \ref{eg:generalCase}). Now, let $(y_\rho^*,u_\rho^*)$ denote an optimal pair of \eqref{eq:ocProblemPen}, which exists by standard arguments (like in the proof of Theorem \ref{thm:existenceOC}) making use of \eqref{ass:mrCC}  (to show weak continuity of the solution map). By definition,
\begin{equation}\label{eq:Jeq}
J(y_\rho^*, u_\rho^*) \leq J(w_\rho, u)\quad \forall u \in U_{ad},\quad \forall w_\rho \in \mathbf P_\rho(u).
\end{equation}
Given any $\tilde u \in U_{ad}$, we pick a subsequence $\{\tilde y_{\rho_n}\}$ such that $\mathbf P_{\rho_n}(\tilde u) \ni \tilde y_{\rho_n} \to \tilde y$ where $\tilde y\in \mathbf{Q}(\tilde u)$; this is possible by Theorem \ref{thm:penalisedConvergence}. The inequality \eqref{eq:Jeq} implies that $J(y_{\rho_n}^*, u_{\rho_n}^*)$ is bounded above by $J(\tilde y_{\rho_n}, \tilde u)$ which in turn is bounded uniformly in $\rho_n$ because $\tilde y_{\rho_n}$ is bounded in $V$ by the estimate of Proposition \ref{lem:existencePenalisedPDE}:%. Furthermore, by Lemma \ref{lem:existencePenalisedPDE},
\[\norm{y_{\rho_n}^*}{V} \leq C\left(\norm{u_{\rho_n}^*}{V^*} + \norm{v_0}{V}\right).\]
Hence for another subsequence (which we shall relabel)
\begin{align*}
u_{\rho_n}^* &\weaklyto u^*\quad\text{in $U_{ad}$},\\
y_{\rho_n}^* &\weaklyto y^* \quad\text{in $V$},
\end{align*}
for some $(u^*, y^*)$ that we need to show is an optimal pair.
By following steps 3 and 4 in the proof of Theorem \ref{thm:penalisedConvergence}, $y_{\rho_n}^* \to y^* =\mathbf{Q}(u^*)$ in $V$ (since $u_{\rho_n}^* \to u^*$ in $V^*$). Hence $(y^*, u^*)$ is a feasible point of \eqref{eq:ocProblem}. Then observe that for $(\hat y, \hat u)$ being any optimal point of \eqref{eq:ocProblem},
\begin{align*}
J(\hat y, \hat u) &\leq J(y^*, u^*) \leq \liminf_{n \to \infty} J(y_{\rho_n}^*, u_{\rho_n}^*) \leq \limsup_{n \to \infty} J(y_{\rho_n}^*, u_{\rho_n}^*) \leq \limsup_{n \to \infty} J(w_{\rho_n}^*, \hat u) \quad \forall w_{\rho_n}^* \in \mathbf P_{\rho_n}(\hat u)
\end{align*}
with the last inequality by \eqref{eq:Jeq}. Now it becomes necessary for $\mathbf{Q}$ to be single-valued since then, $\hat y = \mathbf{Q}(\hat u)$ and it must be the case that we can select a sequence $\{w_{\rho_n}^*\}$ such that $w_{\rho_n}^* \in \mathbf P_{\rho_n}(\hat u)$ and $w_{\rho_n}^* \to \hat y$ in $V$ (by Theorem \ref{thm:penalisedConvergence}), and %since the convergence is strong in $H$ (for a subsequence), 
we find
\[J(\hat y, \hat u) \leq J(y^*, u^*) \leq \lim_{n \to \infty} J(y_{\rho_n}^*, u_{\rho_n}^*) \leq  J(\hat y, \hat u).\]
Because $J(\hat y, \hat u)$ is the minimal value and hence is either independent of $(\hat y, \hat u)$ or uniquely determined by $(\hat y, \hat u)$, the subsequence principle shows that $J(y_\rho^*, u_\rho^*) \to J(\hat y, \hat u)$ (for the entire sequence). Furthermore, the above inequality shows that $(y^*, u^*)$ is optimal and we get $u_\rho^* \to u^*$ in $H$ since we have weak convergence and convergence of the norm.
\end{proof}	
Regarding the assumption in this lemma that $\mathbf{Q}$ is single valued, this is the case if, for example, $\Phi$ is (globally) Lipschitz with Lipschitz constant strictly smaller than $C_a\slash(C_a+C_b)$, see the discussion around \cite[Equation (21)]{AHR}. An alternative condition for uniqueness for QVIs in a specific setting is given in \cite{MR0380554}. %Let us also point out that we could have asked for \eqref{15a} instead of 15b if we enforced hemicontinuity and a liminf condition.

Let us see how the results of this section change if we do not assume complete continuity of $\Phi \colon V \to V$.
\begin{remark}\label{rem:newTwo}
%\begin{itemize}
%\item 
(1) We can drop \eqref{ass:PhiCC} from Theorem \ref{thm:existenceOC} in favour of the conditions in Theorem \ref{thm:newOne} as long as in the Gelfand triple regime \eqref{ass:gf1} we assume $U \cts H$. Examining the proof of Theorem \ref{thm:existenceOC}, the feasibility of the limit of the infimising sequence follows exactly as in the proof of Theorem \ref{thm:newOne}. %To see this, note that we can pass to the limit in the QVI for $y_{n_j}$
%\[\langle Ay_{n_j} - u_{n_j}, y_{n_j} - v \rangle \leq 0 \quad \forall v \in V : v \leq \Phi(y_{n_j}),\]
%by taking as test function  $v_{n_j} =v-\Phi(y^*) + \Phi(y_{n_j})$ where $v \in V$ with $v \leq \Phi(y^*)$ is an arbitrary function; we can pass to the limit using $U \ctsCompact V^*$ to find $y^* \in \mathbf{Q}(u^*)$ (we get $y^* \leq \Phi(y^*)$ due to the weak closedness of $V_+$).  %In the first case, thanks to the compact embedding of $U \ctsCompact V^*$ (and so $u_{n_j} \to u^*$ in $V^*$) and the fact that $y_{n_j}-v_{n_j} \weaklyto y^*-v^*$ in $V$ allows us to pass to the limit in that term. 
The Cauchy estimate is not necessary. Weak lower semicontinuity of the norm allows us to retain the final line in the proof. %The second is similarly handled.
%Under a Gelfand triple setting where $U=H$, the above weak sequential continuity assumption can be replaced with the complete continuity assumption \eqref{ass:newPhiCC} (which is weaker than that).

%\item  We can drop \eqref{ass:PhiCC} from Proposition \ref{lem:convPenOCProblems} in favour of  $V \ctsCompact H$, \eqref{ass:newWLSC} and \eqref{ass:PhiWeakCts} but we would get only $y_\rho^* \weaklyto y^*$ in $V$ (and still $u_\rho^* \to u^*$ in $U$). To see this, we simply need to modify the proof to use Theorem \ref{thm:newOne} instead of Theorem \ref{thm:penalisedConvergence}. The compact embedding into $H$ is needed to bound the term $\limsup_{n \to \infty} J(w_{\rho_n}^*, \hat u)$ above by $J(\hat y, \hat u)$.

\medskip

\noindent (2) If we drop \eqref{ass:PhiCC} from Proposition \ref{lem:convPenOCProblems} in favour of $V \ctsCompact H$  and the conditions in Theorem \ref{thm:newOne} as long as in the Gelfand triple regime \eqref{ass:gf1} we assume $U \cts H$, we would get $y_\rho^* \weaklyto y^*$ in $V$ (i.e., a weak convergence). To see this, we simply need to modify the proof to use Theorem \ref{thm:newOne} instead of Theorem \ref{thm:penalisedConvergence}. The compact embedding into $H$ is needed to bound from above the term $\limsup_{n \to \infty} J(w_{\rho_n}^*, \hat u)$ by $J(\hat y, \hat u)$.

%In fact, under \eqref{ass:PhiWeakCts} and \eqref{ass:newWLSC}, we get the result the weak convergence $y_\rho^* \weaklyto y^*$. We would have $\tilde y_{\rho_n} \weaklyto \tilde y \in \mathbf{Q}(\tilde u)$ in $V$ by Remark \ref{rem:newOne}. By following the argument in Remark \ref{rem:newOne} that modifies the argument given in step 3 of the proof of Theorem \ref{thm:penalisedConvergence}, we would get $y_{\rho_n}^* \weaklyto y^* = \mathbf{Q}(u^*)$ in $V$.. We  have $w_{\rho_n}^* \weaklyto \hat y$ in $V$ and compactness of $V \ctsCompact H$ would allow us to obtain $w_{\rho_n}^* \to \hat y$ in $H$ and hence the $J(\hat y, \hat u) \leq J(y^*, u^*) \leq J(\hat y, \hat u)$ sandwich. Check the strong convergence of the controls.

%Since \eqref{ass:newWLSCForStrong} is enforced, all of these convergences are strong.
%\end{itemize}
\end{remark}

\section{Stationarity}\label{sec:stationarity}
In this section, we shall derive various forms of necessary conditions satisfied by optimal controls and states. Let us first formally define some concepts of stationarity which are motivated by analogous concepts from the VI case and also by the results that we shall obtain later. 

Let $(y,u)  \in V \times H$ be a solution of the optimal control problem \eqref{eq:ocProblem} where $V \cts H$ with $V$ a reflexive Banach space and $H$ a Hilbert space, $U_{ad} \subset H$ is non-empty and weakly sequentially closed (in the context of the previous section, we have assumed $U \equiv H$).

 Inspired by the results we obtain in \S \ref{sec:weakStationarity} in a general function space setting, we say that $(y, u)$ is a \emph{weak C-stationarity} point of \eqref{eq:ocProblem} if there exists $(p, \xi, \lambda) \in V \times V^* \times V^*$ such that
\begin{equation*}%\label{eq:weakStationaritySystem}
\begin{aligned}
y + (\Id-\Phi'(y))^*\lambda + A^*p &= y_d,\\
Ay - u + \xi &= 0,\\
\xi \geq 0 \text{ in $V^*$}, \quad y \leq \Phi(y), \quad \langle \xi, y-\Phi(y)\rangle &= 0,\\
u \in U_{ad} : (\nu u - p , u-v )_H &\leq 0\quad \forall v \in U_{ad},\\
\langle \lambda, p \rangle &\geq 0.
\end{aligned}
\end{equation*}
The function $p$ is said to be the \emph{adjoint state} and $\lambda$ is the \emph{Lagrange multiplier} associated to the  adjoint state equation (the first equation above).

Let us now restrict the discussion to when $H=L^2(\Omega)$ on a domain $\Omega \subset \mathbb{R}^n$. Certain sets associated to the lower-level QVI problem in \eqref{eq:ocProblem} are important in stating the following stationarity conditions. Denoting $\xi:= u-Ay$ (see Proposition \ref{lem:complementarityForQVI}), let us formally define then the following sets:
\begin{align*}
\mathcal{A} &:= \{y = \Phi(y)\} \text{ is the \emph{active} (or coincidence) set,}\\
\mathcal{I} &:= \{y < \Phi(y)\} \text{ is the \emph{inactive} set,}\\
\mathcal{A}_s &:= \{\xi > 0\} \text{ is the \emph{strongly active} set,}\\
\mathcal{B} &:= \{y=\Phi(y)\} \cap \{\xi = 0\} \text{ is the \emph{biactive} set}.
\end{align*}
These definitions are merely heuristic due to the (in general) low regularity of $\xi$, see for example \cite[\S 3 and Appendix A]{Wachsmuth} or \cite{MR3796767} for a rigorous approach to define these objects.  

We say that $(y, u) \in V \times H$ is a \emph{C-stationarity} point of \eqref{eq:ocProblem} if $(y,u)$ is a solution of \eqref{eq:ocProblem} and there exists $(p, \xi, \lambda) \in V \times V^* \times V^*$ such that
\begin{subequations}\label{eq:stationaritySystem}
\begin{align}
y + (\Id-\Phi'(y))^*\lambda + A^*p &= y_d,\label{eq:first}\\
Ay - u +  \xi &= 0,\\
\xi \geq 0 \text{ in $V^*$}, \quad y \leq \Phi(y), \quad \langle \xi, y-\Phi(y)\rangle &= 0,\label{eq:pp}\\
%p &= 0 \text{ a.e. in } \mathcal{A}_s,\\
u \in U_{ad} : (\nu u - p, u-v)_H &\leq 0 \quad  \forall v \in U_{ad},\\
\langle \xi, p^+ \rangle = \langle \xi, p^- \rangle &= 0\label{eq:pxi}\\
\langle \lambda, p \rangle &\geq 0, \quad \langle \lambda, y-\Phi(y)\rangle = 0,\label{eq:penult}\\
\langle \lambda, v \rangle &= 0 \quad \forall v \in V  : v = 0 \text{ a.e. on $\Omega \setminus \mathcal{I}$}\label{eq:last}.
\end{align}
\end{subequations}
Note that we use the condition \eqref{eq:pxi} in lieu of the more commonly seen condition $p = 0 \text{ a.e. in } \{\xi > 0 \}$ due to the low regularity of $\xi$. 
\begin{remark}
It is worth remarking that in certain works \cite{MR3056408}, rather than the inequality constraint in \eqref{eq:penult}, the stronger condition
\begin{equation}\label{eq:strongerpenult}
\langle \lambda, \psi p \rangle \geq 0 \quad \text{for all sufficiently smooth and non-negative $\psi$}
\end{equation}
is required in order to satisfy C-stationarity; this is a direct analogy of the corresponding (element-wise) condition in the finite dimensional setting in \cite{MR1854317}. We will also  consider the obtainment of \eqref{eq:strongerpenult} in Proposition \ref{prop:altCondition}.
\end{remark}

The condition \eqref{eq:last} is in practice difficult to check due to the fact that in general, $\lambda$ possesses only the low $V^*$ regularity. Therefore, one looks for a weaker concept. In the first instance, for an \emph{almost C-stationarity} point, \eqref{eq:last} is replaced by
\[\langle \lambda, v \rangle = 0 \quad \forall v \in V  : v = 0 \text{ a.e. on $\Omega \setminus \mathcal{I}, \;v|_{\mathcal{I}} \in H^1_0(\mathcal{I})$.}\]
More generally, an \emph{$\mathcal{E}$-almost C-stationarity} point, the concept of which was introduced by Hinterm\"uller and Kopacka in \cite{MR2822818, MR2515801}, satisfies \eqref{eq:first}--\eqref{eq:penult} but now \eqref{eq:last} is replaced with 
\[\forall \tau > 0, \exists E^\tau \subset \mathcal{I} \text{ with } |\mathcal{I} \setminus E^\tau| \leq \tau : \langle \lambda, v \rangle = 0 \quad \forall v \in V  : v = 0 \text{ a.e. on $\Omega \setminus E^\tau$}.
\]
This is a condition that arises from an application of Egorov's theorem as we shall see later.

Now, in the other direction, a point which satisfies \eqref{eq:first}--\eqref{eq:pp} and additionally 
\begin{align*}
p &\geq 0 \quad\text{q.e. on $\mathcal{B}$ and } p=0 \text{ q.e. on $\mathcal{A}_s$},\\
\langle \lambda, v \rangle &\geq 0 \quad \forall v \in V : v \geq 0 \text{ q.e. on $\mathcal{B}$ and } v = 0 \text{ q.e. on $\mathcal{A}_s$},
\end{align*}
is called a \emph{strong stationarity} point, which is typically the most stringent notion of stationarity possible and requires differentiability of the control-to-state map to be obtainable. 

In the proceeding sections, we will show that there exist weak C-stationarity, ($\mathcal{E}$-almost) C-stationarity and strong stationarity points under various assumptions. We will, however, first start in \S \ref{sec:BS} with the so-called \emph{Bouligand stationarity} which is a primal condition and is defined below. It also requires differentiability of $\mathbf{Q}$.

  %We find that $y \in \mathbf{Q}(u)$ is well defined for right-hand sides $u$ taken from $U_{ad}$ by the results of \S \ref{sec:existence}.  

\subsection{Bouligand stationarity}\label{sec:BS}
 In the case where $\mathbf{Q}$ is directionally differentiable from the results of \S \ref{sec:Diff}, we have the following Bouligand stationarity (or B-stationarity) characterisation of the optimal control, see \cite[\S 5]{MR0423155} and \cite[Lemma 3.1]{MR739836} for the VI case. To start, define the {radial cone} of $U_{ad}$ at $u^*$ and the tangent cone respectively by
\[\mathcal{R}_{U_{ad}}(u^*)= \{h \in H : \exists  s^* > 0 \text{ such that } u^*+sh \in U_{ad} \; \;\;\forall s \in [0,s^*]\} \quad\text{and}\quad \mathcal{T}_{U_{ad}}(u^*) := \overline{\mathcal{R}_{U_{ad}}(u^*)}.\]
\begin{prop}[Bouligand stationarity]\label{lem:characterisationOfOC}Let $U_{ad}$ be non-empty and $(y^*, u^*)$ be a local minimiser of \eqref{eq:ocProblem} and let %the local assumption \eqref{ass:PhiDerivativeLipschitzConstant}  hold at $y^*$. 
% the assumptions of Proposition \ref{lem:alphaUniquenessAndCty} hold. 
%If 
the assumptions\footnote{These assumptions should be evaluated locally at $y^*$, of course.} of Theorem \ref{thm:dirDiff1} hold.  Then %every minimiser $(u^*, y^*)$ of the optimal control problem \eqref{eq:ocProblem} satisfies
\begin{equation}\label{eq:Bstationarity}
(\alpha_h,y^*- y_d)_H + \nu (u^*, h)_H \geq 0 \quad \forall h \in \mathcal{T}_{U_{ad}}(u^*),
\end{equation}
%whereas if instead the assumptions of Theorem \ref{thm:dirDiff2} hold, the above inequality holds for all \[h \in \cl{\mathcal{R}_{U_{ad}}(u^*)\cap (H_+-G)}.\] 
%\[(\alpha_h,y^*- y_d) + \nu (u^*, h) \geq 0 \quad \forall h \in \cl{\mathcal{R}_{U_{ad}}(u^*)\cap (H_+-G)},\]
where $\alpha_h$ is %, which satisfies 
%\begin{equation*}
%\begin{aligned}
%\alpha_h \in \mathcal{K}_{\mathbf{K}(y^*)}(y^*,\alpha_h) &: \langle A\alpha_h - h, \alpha -v \rangle \leq 0 \quad \forall v \in \mathcal{K}_{\mathbf{K}(y^*)}(y^*,\alpha_h).
%\end{aligned}
%\end{equation*}
the directional derivative %associated to the perturbation $y^s \in \mathbf{Q}(u^* + sh)$ 
given uniquely through Theorem \ref{thm:dirDiff1} as the solution of \eqref{eq:QVIforAlpha} with source $h$.% or Theorem \ref{thm:dirDiff2} respectively.%, and 
%where
%\begin{align*}
%$\mathcal{K}_{\mathbf{K}(y^*)}(y^*,\alpha_h)= 
%\{ \varphi \in V : \varphi \leq \Phi'(y^*)(\alpha_h) \text{ q.e. on }\{y^*=\Phi_{y^*}\} \text{ and } \langle Ay^*-u^*, \varphi -\Phi'(y^*)(\alpha_h)\rangle = 0\}.$
%\end{align*}
\end{prop}
\begin{proof}Take $h$ in the radial cone of $U_{ad}$ at $u^*$ so that it is an admissible direction. Using this direction term, we define $y_s$ as given by Theorem \ref{thm:dirDiff1} after having initially selected $y^* \in \mathbf{Q}(u^*)$. This satisfies $y_s =  y^* + s\alpha_h + o(s)$ where $\alpha_h$ is the directional derivative (uniquely determined thanks to Proposition \ref{prop:strongConvAlphas}) and $o$ is a remainder term. It follows that $(u^*+sh,y_s)$ can be made arbitrarily close to $(u^*,y^*)$ if $s$ is sufficiently small (since $y_s-y^* = s\alpha_h + o(s)$ and the right-hand side tends to zero in $V$). Hence, by definition of local minimiser, we have $J(y_s, u^* + sh) \geq J(y^*, u^*)$ for $s$ sufficiently small. Writing this inequality out, we get
\begin{align*}
0 &\leq \norm{y_s-y_d}{H}^2 + \nu\norm{u^* + sh}{H}^2 - \norm{y^*-y_d}{H}^2 - \nu\norm{u^*}{H}^2\\
&=  \norm{y_s}{H}^2 - \norm{y^*}{H}^2 + 2(y^*-y_s, y_d)_H + \nu  s^2\norm{h}{H}^2 + 2\nu s(u^*, h)_H.
\end{align*}
%% global minimiser
%By definition of the minimiser, we have $J(y_s, u^* + sh) \geq J(y^*, u^*)$ for any admissible direction $h$, $s \geq 0$ and any $y_s \in \mathbf{Q}(u^* + sh)$, and it is clear that every $h$ in the radial cone of $U_{ad}$ at $u^*$ is an admissible direction. Writing this inequality out, we get
%\begin{align*}
%0 &\leq \norm{y_s-y_d}{H}^2 + \nu\norm{u^* + sh}{H}^2 - \norm{y^*-y_d}{H}^2 - \nu\norm{u^*}{H}^2\\
%&=  \norm{y_s}{H}^2 - \norm{y^*}{H}^2 + 2(y^*-y_s, y_d)_H + \nu  s^2\norm{h}{H}^2 + 2\nu s(u^*, h)_H.
%\end{align*}
%We select $y_s$ as given by Theorem \ref{thm:dirDiff1} after having initially selected $y^* \in \mathbf{Q}(u^*)$, which satisfies $y_s =  y^* + s\alpha_h + o(s)$ where $\alpha_h$ is the directional derivative (uniquely determined thanks to Proposition \ref{prop:strongConvAlphas}) and $o$ is a remainder term. 
This leads to
\begin{align*}
0 &\leq \norm{y^* + s\alpha_h + o(s)}{H}^2 - \norm{y^*}{H}^2 - 2(s\alpha_h + o(s), y_d)_H + \nu  s^2\norm{h}{H}^2 + 2\nu s(u^*, h)_H\\
&=\norm{s\alpha_h + o(s)}{H}^2 + 2(s\alpha_h + o(s),y^*- y_d)_H + \nu  s^2\norm{h}{H}^2 + 2\nu s(u^*, h)_H\\
&=s^2\norm{\alpha_h + s^{-1}o(s)}{H}^2 + 2(s\alpha_h + o(s),y^*- y_d)_H + \nu  s^2\norm{h}{H}^2 + 2\nu s(u^*, h)_H.
\end{align*}
Dividing by $s$ and sending to zero, the above yields
\begin{align*}
0 &\leq 2(\alpha_h,y^*- y_d)_H + 2\nu (u^*, h)_H \quad \forall h \in \mathcal{R}_{U_{ad}}(u^*),
\end{align*}
and by density and the continuity result of Proposition \ref{lem:alphaUniquenessAndCty}, also for $h \in \mathcal{T}_{U_{ad}}(u^*).$% := \overline{\mathcal{R}_{U_{ad}}(u^*)}$ in the \emph{tangent cone}.

%In the setting where Theorem \ref{thm:dirDiff2} is applied, the above displayed inequality only holds for all $h \in \mathcal{R}_{U_{ad}}(u^*) \cap (H_+ - G)$ since we need the direction to be bounded from below by a function $G$. Then taking the closure in $H$ again leads to the stated inequality.
\end{proof}

\subsection{Weak C-stationarity}\label{sec:weakStationarity}

%\subsection{Penalisation by PDEs}\label{sec:penalisationOfQVI}
In this section we will show a type of weak C-stationarity for the optimal pair by passing to the limit in the stationarity system satisfied by the optimal pair of the PDE regularisation of the QVI. Recall the notations and framework of \S\ref{sec:penalisationForExistence} and \S \ref{sec:penalisedOCProblem} where we studied the convergence of solutions of certain PDEs to a solution of the associated QVI and the associated optimal control problems.   In this section, we again take 
\[\text{$(y^*, u^*)$ to be an arbitrary local minimiser of \eqref{eq:ocProblem}}.\]%We take $U_{ad} \subseteq H$ to be a non-empty, closed and convex.
In addition to the basic setup of Assumption \ref{ass:seqReg}, we need the following fundamental conditions related to $\Phi$, in which we also recall two assumptions that were stated earlier for the convenience of the reader. 
\begin{ass}\label{ass:commonAssumptions}
Assume that
\begin{align}
&\text{there exists $v_0 \in V$ such that $v_0 \leq \Phi(v)$ for all $v \in V$,}\tag{\ref{ass:feasiblePoint}}
\\	
&\text{$\Phi\colon V \to V$ is completely continuous},\tag{\ref{ass:PhiCC}}\\
&\text{there exists $\epsilon > 0$  such that $\Phi\colon V \to V$ is continuously Fr\'echet differentiable on }  B_\epsilon(y^*),\label{eq:assPhiFrechet}\\
\nonumber &\text{$\mathbf{Q}$ is single valued}.
\end{align}
\end{ass}
We also introduce the following invertibility assumptions; these are stated separately from above since they will come in use later in another section. Note that these types of conditions are also needed in \cite{MR4083198}.
\begin{ass}\label{ass:forEpsAlmostCStat}%Assume the invertibility condition \eqref{ass:newPhiInvInvertible}. 
%For sequences  $z_n \to z$ and  $q_n \weaklyto q$ in $V$ with $z_n, z \in B_\epsilon(y^*)$, we assume that
Assume that
\begin{align}
&(\Id-\Phi'(z))\colon V \to V \text{ is invertible for $z \in B_\epsilon(y^*)$},\label{ass:newPhiInvInvertible}\\
&A(\Id-\Phi'(z))^{-1}\colon V \to V^* \text{ is uniformly bounded and uniformly coercive in  $z \in B_\epsilon(y^*)$}\label{ass:newAPhiInvCoerciveUniform}.
&%z_n \to z \text{ and } q_n \weaklyto q \text{ in $V$ implies } 
%&z_n \to z \text{ and } 
%\text{if } q_n \to q \text{ in $V$ then } 
%(\Id-\Phi'(z_n))^{-1}w_n \to (\Id-\Phi'(z))^{-1}w \text{ in $V$}\label{ass:doubleCont}\\
%&\lim_{n \to \infty} (\Id-\Phi'(z_n))^{-1}w_n = (\Id-\Phi'(z))^{-1}w \text{ in $V$}\label{ass:doubleCont}\\
%&%z_n \to z \text{ and } q_n \weaklyto q \text{ in $V$ implies } 
%(z, (\Id-\Phi'(z))^{-1}q)  \leq \liminf_{n \to \infty} (z_n, (\Id-\Phi'(z_n))^{-1}q_n) \label{ass:doubleLSC}\\
%& %z_n \to z \text{ and } q_n \weaklyto q \text{ in $V$ implies } 
%(y_d, (\Id-\Phi'(z))^{-1}q) \geq \limsup_{n \to \infty} (y_d, (\Id-\Phi'(z_n))^{-1}q_n) \label{ass:limsupCond}
\end{align}
\end{ass}
 %For the assumption \eqref{eq:assExistencePDE} below, see Remark \ref{rem:prec}.
%A few words are in order concerning Assumption \ref{ass:forEpsAlmostCStat}.

%\begin{enumerate}[label=(\roman*)]\item 
The main result of this section is the following theorem which shows that local minimisers are weak C-stationarity points. %In the theorem, we will again require an assumption that already appeared earlier but we shall restate it here for convenience.
\begin{theorem}[Weak C-stationarity]\label{thm:weakStationarity}
Suppose that
\begin{equation}
\text{$U_{ad}$ is non-empty, closed and convex and $V \ctsCompact H \cts V^*$ is a Gelfand triple.}\label{ass:forWeakS}
\end{equation}
In addition to Assumptions \ref{ass:seqReg},  \ref{ass:commonAssumptions} and \ref{ass:forEpsAlmostCStat}, suppose that $m_\rho$ satisfies along with \eqref{ass:mrPenal}--\eqref{ass:mrForFeas} the conditions
%one of the following 
%\begin{align}
%&m_\rho(\Id-\Phi)\colon V \to V^* \text{ is completely continuous,}\tag{\ref{ass:mrCC}}\\
%&m_\rho(\Id-\Phi)\colon V \to V^* \text{ is monotone, radially continuous and bounded}\tag{\ref{ass:mrMonotoneBdded}},
%\end{align}
%as well as 
\begin{align}
&m_\rho\colon H \to V^*  \text{ is continuous}\label{ass:mrCts}\\
&m_\rho\colon V \to V^*  \text{ is continuously Fr\'echet differentiable}\label{ass:mrDiff}.\end{align}
%Let \eqref{ass:PhiCC}, \eqref{ass:feasiblePoint} and the local assumptions \eqref{eq:assPhiFrechet}, \eqref{ass:mrDiff},  and Assumptions \ref{ass:forWeakS}, \ref{ass:forEpsAlmostCStat}, and suppose that $\mathbf{Q}$ is single-valued. 
Then there exist multipliers $(p^*, \xi^*, \lambda^*) \in V \times V^* \times V^*$ satisfying the weak C-stationarity system 
\begin{subequations}\label{eq:weakStationaritySystemAgain}
\begin{align}
y^* + (\Id-\Phi'(y^*))^*\lambda^* + A^*p^* &= y_d,\label{eq:wSS1}\\
Ay^* - u^* + \xi^* &= 0,\label{eq:wSS2}\\
\xi^* \geq 0 \text{ in $V^*$}, \quad y^* \leq \Phi(y^*), \quad \langle \xi^*, y^*-\Phi(y^*)\rangle &= 0,\label{eq:wSS4}\\
u^* \in U_{ad} : (\nu u^* - p^* , u^*-v )_H &\leq 0\quad \forall v \in U_{ad},\label{eq:wSS3}\\
\langle \lambda^*, p^* \rangle &\geq 0.\label{eq:wSS5}
\end{align}
\end{subequations}
%In addition, $(y^*,u^*,p^*,\xi^*,\lambda^*)$ can be characterised as a limit of subsequences as in \eqref{eq:listOfConvW} and every local minimiser of \eqref{eq:ocProblem} satisfies the above stationarity system. 
\end{theorem}
Here, we have \emph{assumed} the existence of $C^1$ maps $m_\rho$ --- this will have to be verified on a case-by-case basis (we leave the possibility of being able to define such maps satisfying the conditions \eqref{ass:mrPenal}--\eqref{ass:mrForFeas} and \eqref{ass:mrCts}--\eqref{ass:mrDiff} in the general setting to the interested reader, who may find \cite{MR3244144} useful for this purpose). However, let us note that in the most common case of interest where the function spaces involve functions over domains in $\mathbb{R}^n$ with the usual ordering, it is usually possible to construct sufficiently smooth $m_\rho$, see for example \S \ref{sec:penalisationOfQVI}.
\begin{remark}\label{rem:afterWeakStat}
\begin{enumerate}[label=(\roman*)]
\item  We assumed the complete continuity \eqref{ass:PhiCC} to utilise the strong convergence result of Theorem \ref{thm:penalisedConvergence}. It would be interesting to see how the calculations below can be adapted in the case where (we do not have complete continuity and) we only have weak convergence from Theorem \ref{thm:newOne}.

\item Due to the Gelfand triple setup and complete continuity of $\Phi$ here, we find from \eqref{ass:mrCts} that the complete continuity of $m_\rho$ condition \eqref{ass:mrCC} is satisfied.
\item The meaning of  \eqref{ass:newAPhiInvCoerciveUniform} is that for all $z \in B_\epsilon(y^*)$, the operator $A(\Id-\Phi'(z))^{-1}$ has a boundedness constant $C_b'$ and a coercivity constant $C_a'$ both of which are independent of $z$. 
%\item 
A consequence is that %of \eqref{ass:newAPhiInvCoerciveUniform} is that
\begin{equation}
(\Id-\Phi'(z))^{-1}\colon V \to V \text{ is bounded uniformly for all $z \in B_\epsilon(y^*)$}.\label{ass:newPhiInvBoundedUniform}
\end{equation}
%which follows from
%\begin{align*}
%C_a\norm{(\Id-\Phi'(z))^{-1}w}{V}^2 \leq \langle A(\Id-\Phi'(z))^{-1}w, (\Id-\Phi'(z))^{-1}w \rangle \leq C_b'\norm{w}{V}\norm{(\Id-\Phi'(z))^{-1}w}{V}.
%\end{align*}
Since $\Phi$ is $C^1$, we automatically  have that $(\Id-\Phi'(z))^{-1}$ is bounded; \eqref{ass:newPhiInvBoundedUniform} clarifies that the bound is uniform.
%which implies that
 %In it, we will use the following fact. Since we have shown that $y_\rho^* \to y^*$ in $V$, whenever $\rho$ is sufficiently small, we obtain that $y_\rho^* \in B_\epsilon(y^*)$ and hence for such $\rho$ the assumption \eqref{eq:newAss2} is applicable and thus 
%\[\Phi'(y_\rho^*)(p_\rho^*)p_\rho^* \leq C_P(p_\rho^*)^2 \quad\text{a.e. in $\Omega$.}\] 
%Passing to the limit we obtain the following result.  %We need the assumption on $U_{ad}$ in the next theorem to pass to the limit in the VI satisfied by $u_\rho^*$.
%\item If $\Phi'(z)$ is independent of $z$ (as is the case when $\Phi$ is linear, for example), then \eqref{ass:newAPhiInvCoerciveUniform} can be replaced with \eqref{ass:newAPhiInvCoercive}.
%\item If $\Phi'(z)$ is independent of $z$ (as is the case when $\Phi$ is linear, for example), then \eqref{ass:newAPhiInvCoerciveUniform}--\eqref{ass:limsupCond} can be replaced with  \eqref{ass:newAPhiInvCoercive}: the first follows from \eqref{ass:newAPhiInvCoercive}, \eqref{ass:doubleCont} is automatic from $\Phi$ being Fr\'echet differentiable and the last two follows because a bounded linear operator preserves weak continuity. %That is to say that only the weak lower semicontinuity of \eqref{ass:weakLowerSC} is needed.
%\end{enumerate}
\end{enumerate}
\end{remark}
Let us proceed with proving this result.
\subsubsection{Stationarity for the penalised optimal control problem}\label{sec:stationPenalisedOC}
Recall the penalised problem \eqref{eq:ocProblemPen} that approximates \eqref{eq:ocProblem}:
\begin{equation}
\min_{u \in U_{ad}} J(y_\rho,u) \quad\text{such that}\quad  Ay_\rho + \frac 1\rho m_\rho(y_\rho-\Phi(y_\rho)) = u.\tag{\ref{eq:ocProblemPen}}
\end{equation}
%Recalling the map $\mathbf P_\rho$ from \S \ref{sec:penalisationForExistence}, we can write the equation above as $y_\rho \in \mathbf{P}_\rho(u)$.
%Since by \eqref{eq:Uad} $u_\rho^*$ is assumed to be a bounded sequence, we have with $u_\rho^* \weaklyto u$ for some $u \in L^2(\Omega)$. The equation for $y_\rho^*$ then implies that $y_\rho^* \weaklyto y$ in $V$ for some $y$, and by Lemma \ref{lem:convPenOCProblems} we have $u=u^*$ and $y=y^* \in \mathbf{Q}(u^*)$. Theorem \ref{thm:penalisedConvergence} gives us (the strong convergence) $y_\rho^* \to y^*$ in $V$. 
%\subsection{Stationarity for the penalised optimal control problem}
Under Assumption \ref{ass:commonAssumptions} and \eqref{ass:mrCC}, %either \eqref{ass:mrCC} or \eqref{ass:mrMonotoneBdded}, 
Proposition \ref{lem:convPenOCProblems} is applicable.  For the moment and for purposes of a simpler exposition, let us assume that
\begin{equation}\label{eq:locMinCond}
\text{$(y^*, u^*)$ is the optimal point of \eqref{eq:ocProblem} given in Proposition \ref{lem:convPenOCProblems}}
\end{equation}(we will discard this later on). Via the proposition, we obtain the existence of minimisers $(y_\rho^*, u_\rho^*)$  of \eqref{eq:ocProblemPen} such that
\[(y_\rho^*, u_\rho^*) \to (y^*, u^*) \text{ in $V \times H$.}\]
%Since we have shown that $y_\rho^* \to y^*$ in $V$, 
Thus, for any $\epsilon > 0$, we can find a $\rho_0$ such that $\rho \leq \rho_0$ implies \[y_\rho^* \in B_\epsilon(y^*)\] (this is why it has been possible to formulate most assumptions on $\Phi$ only locally). To derive stationarity conditions for the penalised problem \eqref{eq:ocProblemPen}, we check the Zowe--Kurcyusz constraint qualification \cite{MR526427} (see also the Robinson condition \cite{MR0410522}). To do so, we make the necessary surjectivity assumption \eqref{ass:existenceForPDE} below regarding existence for the linearised equation --- we discuss instances where it holds in Remark \ref{rem:prec}.

\begin{lem}\label{lem:RKZPenal}Assume \eqref{eq:assPhiFrechet}, \eqref{ass:mrDiff}, \eqref{ass:forWeakS} and suppose that %$\Phi$ is such that $g\colon V \times H \to V^*$ is continuously Fr\'echet differentiable and
\begin{align}
%&\exists \epsilon > 0  : \Phi\colon V \to H \text{ is continuously Fr\'echet differentiable on }  B_\epsilon(y^*), \label{eq:assPhiFrechet}\\
&\forall \rho \leq \rho_0,\; \forall f \in V^*,\; \exists z \in V : Az + \frac 1\rho m'_{\rho}(y_\rho^*-\Phi(y_\rho^*))(I-\Phi'(y_\rho^*))(z) = f\label{ass:existenceForPDE}.
%&\Phi'(v)\colon V \to V^* \text{ is completely continuous for all $v \in V$}\label{ass:weakCompContOfDerivOfPhi}\\
\end{align}
Then, for such $\rho$ and any optimal point $(y_\rho^*, u_\rho^*)$ of \eqref{eq:ocProblemPen}, there exists $p_\rho^* \in V$ such that 
\begin{equation}\label{eq:penalisedOCSystem}
\begin{aligned}
%Ay_\rho^* + \frac 1\rho \max_{\epsilon(\rho)}(0,y_\rho^* - \Phi(y_\rho^*)) &= u_\rho^*,\\
A^*p_\rho^* +\frac 1\rho (\Id-\Phi'(y_\rho^*))^*m_\rho'(y_\rho^* - \Phi(y_\rho^*))^*p_\rho^* &=y_d -  y_\rho^*,\\
(\nu u_\rho^* - p_\rho^*, u_\rho^* - v)_H  &\leq 0 \qquad \forall v \in U_{ad}.
\end{aligned}
\end{equation}
%where the action of the second term on the left-hand side of the equation is to be understood as in \eqref{eq:tbU}.
%and that for all $v \in V$,
%\begin{enumerate}
%\item\label{ass:PhiDerivativeLinear} $\Phi$ is such that $g$ is continuously differentiable, 
%\item\label{ass:positivityOfMaxDerv} $\int \max'_{\epsilon(\rho)}(0,v-\Phi(v))(I-\Phi'(v))yy \geq 0$ a.e. 
%\item\label{ass:boundednessofPhiDeriv} $\norm{\Phi'(v)y}{V} \leq C\norm{y}{V}$ for all $y \in V$ holds for some constant $C$. 
%\end{enumerate}%, or that $\norm{\Phi'(y_\rho)y}{V} \leq C\norm{y}{V}$ holds
\end{lem}
\begin{proof}We introduce the following notation:
\begin{align*}
&X:=V\times H, %\quad F(x) := J(y,u), 
%\quad C:=V\times U_{ad},
\quad g(x) = g(y,u) := Ay + \frac 1\rho m_\rho(y-\Phi(y)) -u,\\
&x_\rho = (y_\rho^*, u_\rho^*), \quad C(x_\rho) := \{k(v-y_\rho^*, h-u_\rho^*) : v \in V, h \in U_{ad}, k \geq 0\}.%\qquad\text{and} \qquad K(g(x_\rho)) = \{0\}.
%&Y=V^*, \quad K=\{0\}.
\end{align*}
The map $g\colon X \to V^*$, being a composition of $C^1$ maps, is continuously Fr\'echet differentiable at $x_\rho$ and we must check that $g'(x_\rho)C(x_\rho)%-K(g(x_\rho))
=V^*$, but
%The condition is thus $g'(x_\rho)C(x_\rho) = V^*$, and 
since $\tilde C:=V\times \{0\} \subset C(x_\rho)$, it suffices to verify $g'(x_\rho)\tilde C = V^*$.
%Let us show that for every $f \in V^*$, there exists a $y \in V$ such that $g'(x_\rho)(y,0) = f$. 
Observing that %for $x_\rho = (y_\rho, u_\rho)$,
\begin{align*}
%g(x_\rho) &:= Ay_\rho + \frac 1\rho m_\rho(y_\rho-\Phi(y_\rho)) - u_\rho \qquad\text{and}\qquad 
g'(x_\rho)(y,0) = Ay + \frac 1\rho m'_\rho(y_\rho^*-\Phi(y_\rho^*))(y-\Phi'(y_\rho^*)(y)),
\end{align*}
it follows that we need existence for the PDE in \eqref{ass:existenceForPDE} 
%\begin{equation}\label{eq:pdeWeNeedEFor}
%Az + \frac 1\rho m'_{\rho}(y_\rho^*-\Phi(y_\rho^*))(\Id-\Phi'(y_\rho^*))(z) = f,
%\end{equation}
and this is guaranteed by assumption for $\rho$ sufficiently small.  %Since by \eqref{ass:weakCompContOfDerivOfPhi} we have complete continuity of $\Phi'(v)(\cdot)\colon V \to V$, we can apply \cite[\S II, Corollary 2.2]{Showalter}) to conclude existence for the PDE.
%which 
%\begin{equation*}
%Az + \frac 1\rho \max'_{\epsilon(\rho)}(0,y_\rho-\Phi(y_\rho))(I-\Phi'(y_\rho))z = f.
%\end{equation*}
%is guaranteed by assumption
%\paragraph{Case 2 (incomplete and wrong atm)}Generally: this has a solution if there exists a constant such that
%\[|\max'_{\epsilon(\rho)}(0,d-\Phi(d))(I-\Phi'(d))y| \leq C|y|\]
%a.e. Since $\max'_{\epsilon(\rho)}(0,\cdot) \leq 1$, we find the result by \cite{GilbargTrud}.
Calculating the adjoint $g'(x_\rho)^*\colon V \to X^*$ of $g'$ via
\begin{align*}
\langle g'(x_\rho)(y,u), v \rangle &= \langle Ay, v\rangle + \frac 1\rho \langle m'_\rho(y_\rho^*-\Phi(y_\rho^*))(y-\Phi'(y_\rho^*)(y)), v \rangle - (u,v)\\
%&= \langle y, A^*v\rangle + \frac 1\rho \int_\Omega m'_\rho(y_\rho^*-\Phi(y_\rho^*))v(\Id-\Phi'(y_\rho^*))y - (v,u)\\
&= \langle y, A^*v\rangle + \frac 1\rho \langle y, (\Id-\Phi'(y_\rho^*))^{*}m'_\rho(y_\rho^*-\Phi(y_\rho^*))^*v\rangle - (v, u),
%&= \langle (y,u), g'(x_\rho)^*v \rangle,
\end{align*}
we find
\[g'(x_\rho)^*(v) = \left(A^*v + \frac{1}{\rho}(\Id-\Phi'(y_\rho^*))^*m_\rho'(y_\rho^*-\Phi(y_\rho^*))^*v, - v\right).\]
%where we understand the action of the second term in the first argument via
%\begin{equation}\label{eq:tbU}
%\langle (\Id-\Phi'(y_\rho^*))^*m_\rho'(y_\rho^*-\Phi(y_\rho^*))^*v, \varphi \rangle = \langle v, m_\rho'(y_\rho^*-\Phi(y_\rho^*))(\Id-\Phi'(y_\rho^*))\varphi \rangle.
%\end{equation}
Applying, for example, \cite[Theorem 6.3]{MR2583281}, we get the existence of $p_\rho^* \in V$ such that $J'(x_\rho)-g'(x_\rho)^*p_\rho^* \in C(x_\rho)^\circ$, i.e., for all $k \geq 0$,
\begin{align*}
\langle y_\rho^* - y_d + \frac 1\rho (I-\Phi'(y_\rho^*)) ^*m_\rho'(y_\rho^* - \Phi(y_\rho^*))^*p_\rho^*+ A^*p_\rho^*, k(c_1-y_\rho^*) \rangle &\geq 0\qquad \forall c_1 \in V,\\
(\nu u_\rho^* - p_\rho^*, k(c_2 - u_\rho^*))_H &\geq 0 \qquad \forall c_2 \in U_{ad}.
\end{align*}
As $c_1 \in V$ can be chosen arbitrarily, we find the stated result.%then $g'(x_\rho)\tilde C = V^*$ and hence 
\end{proof}
{\begin{remark}\label{rem:prec} 
The conditions of Assumption \ref{ass:forEpsAlmostCStat} are clearly sufficient to guarantee the surjectivity condition \eqref{ass:existenceForPDE}; and in fact \eqref{ass:newAPhiInvCoerciveUniform} can be replaced with asking for
%\begin{align}
%\nonumber &(\Id-\Phi'(z))\colon V \to V \text{ is invertible for all $z \in B_\epsilon(y^*)$},\\%\label{ass:newPhiInvInvertible}\\
$A(\Id-\Phi'(z))^{-1}\colon V \to V^* \text{ to be coercive for all $z \in B_\epsilon(y^*)$}.$ %\label{ass:newAPhiInvCoercive}
%\end{align}
Indeed, first observe that the bounded inverse theorem guarantees that $A(\Id-\Phi'(z))^{-1}\colon V \to V^*$ is bounded for $z \in B_\epsilon(y^*)$. Now, the equation %\eqref{eq:pdeWeNeedEFor}
%\[Az + \frac 1\rho m'_{\rho}(y_\rho^*-\Phi(y_\rho^*))(I-\Phi'(y_\rho^*))(z) = f\]
%can be reformulated
\[A(I-\Phi'(y_\rho^*))^{-1}w + \frac 1\rho m'_{\rho}(y_\rho^*-\Phi(y_\rho^*))w = f\]
has a unique solution $w \in V$  by the Lax--Milgram theorem, leading to existence of $z:=(I-\Phi'(y_\rho^*))^{-1}w \in V$ satisfying the equation in \eqref{ass:existenceForPDE}.

\end{remark}
%\begin{remark}
%Assumption \eqref{eq:newAss2} is sufficient to guarantee that \eqref{eq:pdeWeNeedEFor} has a solution but perhaps not necessary.
%\end{remark}
%\begin{remark}\label{rem:prec}
%We have chosen to make the general assumption \eqref{eq:assExistencePDE} and leave it to be verified on a case-by-case basis since it appears difficult to formulate general conditions that imply the assumption. The factor $\rho^{-1}$ in front of the term $m'_\rho(y_\rho^*-\Phi(y_\rho^*))\Phi'(y_\rho^*)(z)$ means the resulting boundedness constant cannot be absorbed into the coercivity constant for $A$ for $\rho$ sufficiently small and therefore the obvious elliptic PDE methods for existence cannot be applied. Nevertheless, if 
%\[\norm{\Phi'(y_\rho^*)z}{V} \leq C\norm{z}{V} \qquad \forall z \in V\]
%for some constant $C$ and
%\begin{equation}\label{ass:onMonotonicityOfPhi}
%\int m'_{\rho}(y_\rho-\Phi(y_\rho))(I-\Phi'(v))zz \geq 0 \quad\forall z \in V
%\end{equation}
%then \eqref{eq:assExistencePDE} holds if we have complete continuity of $\Phi'(v)(\cdot)\colon V \to V$ (we can apply \cite[Corollary 2.2]{Showalter}).
%\end{remark}
\subsubsection{Passage to the limit $\rho \to 0$}
Now the objective is to pass to the limit in \eqref{eq:penalisedOCSystem} as $\rho \to 0$ for which we shall need %further conditions that assumed thus far. %Assumption \ref{ass:forEpsAlmostCStat} relies on the existence of an inverse to $(\Id-\Phi'(z))$ for all $z$ in a neighbourhood 
%We need 
some technical results.

\begin{lem}\label{lem:technicalLemma}Under Assumption \ref{ass:forEpsAlmostCStat}, if $z_n \to z$  and $q_n \weaklyto q$  in $V$ with $z_n, z \in B_\epsilon(y^*)$, then
\begin{align}
&(\Id-\Phi'(z_n))^{-1}q_n \weaklyto (\Id-\Phi'(z))^{-1}q \quad\text{in $V$,}\label{eq:otherOne}\\
&\langle A(\Id-\Phi'(z))^{-1}q, q \rangle \leq \liminf_{n \to \infty}\langle A(\Id-\Phi'(z_n))^{-1}q_n, q_n \rangle\label{ass:weakLowerSC}.
\end{align}
The convergence in \eqref{eq:otherOne} is strong if $q_n \to q$ in $V$. 
\end{lem}
In order to not disturb the flow of the paper, the proof of this lemma has been placed in Appendix \ref{app:technical}. As an immediate corollary to Lemma \ref{lem:technicalLemma}, for sequences %$z_n \to z$, 
$w_\rho \to w$ and $q_\rho \weaklyto q$  in $V$, %with $z_n, z \in B_\epsilon(y^*)$
we have
\begin{align}
&\lim_{n \to \infty} (\Id-\Phi'(y^*_\rho))^{-1}w_\rho = (\Id-\Phi'(y^*))^{-1}w \text{ in $V$}\label{ass:doubleCont},\\
&%z_\rho \to z \text{ and } q_\rho \weaklyto q \text{ in $V$ implies } 
(y^*, (\Id-\Phi'(y^*))^{-1}q)_H  \leq \liminf_{n \to \infty} (y^*_\rho, (\Id-\Phi'(y^*_\rho))^{-1}q_\rho)_H \label{ass:doubleLSC},\\
& %z_\rho \to z \text{ and } q_\rho \weaklyto q \text{ in $V$ implies } 
(y_d, (\Id-\Phi'(y^*))^{-1}q)_H \geq \limsup_{n \to \infty} (y_d, (\Id-\Phi'(y^*_\rho))^{-1}q_\rho)_H.\label{ass:limsupCond}
\end{align}
We are now ready to conclude.
\begin{proof}[Proof of Theorem \ref{thm:weakStationarity}]First, note that Proposition \ref{lem:complementarityForQVI} directly gives \eqref{eq:wSS4}. Assumption \ref{ass:forEpsAlmostCStat} implies the surjectivity condition \eqref{ass:existenceForPDE} (see Remark \ref{rem:prec}), therefore the stationarity conditions in \eqref{eq:penalisedOCSystem} for the penalised problem are available. 

Now, the weak form of the equation for $p_\rho^*$ is
\begin{align*}
\langle A^*p_\rho^*, \varphi\rangle  +   \frac 1\rho \langle m_\rho'(y_\rho^* - \Phi(y_\rho^*))^*p_\rho^*, (\Id-\Phi'(y_\rho^*))\varphi \rangle &=(y_d-y_\rho^*, \varphi)_H \qquad \forall \varphi \in V.
\end{align*}
By defining $v:=(\Id-\Phi'(y_\rho^*))\varphi$, thanks to the invertibility assumption \eqref{ass:newPhiInvInvertible}, this can be transformed to 
\begin{align*}
\langle A(\Id-\Phi'(y_\rho^*))^{-1}v, p_\rho^* \rangle  +   \frac 1\rho \langle m_\rho'(y_\rho^* - \Phi(y_\rho^*))^*p_\rho^*, v \rangle &=(y_d-y_\rho^*, (\Id-\Phi'(y_\rho^*))^{-1}v)_H\qquad \forall v \in V.
\end{align*}
Selecting $v=p_\rho^*$, using the coercivity \eqref{ass:newAPhiInvCoerciveUniform}, the monotonicity of $m_\rho$ (which implies that $\langle m_\rho'(v)(h),h \rangle \geq 0$ for all $v, h \in V$), Young's inequality with $\gamma>0$ and the uniform boundedness of $(\Id-\Phi'(y_\rho^*))^{-1}$ assured by \eqref{ass:newPhiInvBoundedUniform}, we obtain
\begin{align*}
C_a'\norm{p_\rho^*}{V}^2 \leq C_\gamma\norm{y_d-y_\rho^*}{H}^2 + \gamma\norm{p_\rho^*}{V}^2.
\end{align*}
Selecting $\gamma$ sufficiently small so that the right-most term is absorbed onto the left, we obtain a 
%\[q_\rho^* := (\Id-\Phi'(y_\rho^*))p_\rho^*,\]
% we can rewrite the equation for $p_\rho^*$ as
%\[A^*(I-\Phi'(y_\rho^*))^{-1}q_\rho^* + \frac 1\rho m'_{\rho}(y_\rho^*-\Phi(y_\rho^*))q_\rho^* + y_\rho^* = y_d,\]
%which by \eqref{ass:newAPhiInvCoercive} immediately leads to a 
bound on $\{p_\rho^*\}$ independent of $\rho$. This gives rise to %a bound on $\{p_\rho^*\}$  and thus for a subsequence (relabelled here) 
the convergence (for a subsequence that has been relabelled)
\[p_\rho^* \weaklyto p^* %\quad\text{and}\quad q_\rho^* \weaklyto q^* 
\quad \text{in $V$}.\]
Define 
\begin{align*}
\lambda_\rho^* &:= \frac 1\rho m_{\rho}'(y_\rho^*-\Phi(y_\rho^*))^*p_\rho^*,\\
\mu_\rho^* &:= \frac 1\rho (\Id-\Phi'(y_\rho^*))^*m_{\rho}'(y_\rho^*-\Phi(y_\rho^*))^*p_\rho^* = y_d-y_\rho^* - A^*p_\rho^*,\\
\xi_\rho^* &:= \frac 1\rho m_\rho(y_\rho^*-\Phi(y_\rho^*)) = u_\rho^* - Ay_\rho^*,
\end{align*}
the latter two of which, since their right-hand sides converge, satisfy the following convergences both in $V^*$:
\begin{alignat}{3}
\mu_\rho^* &\weaklyto \mu^* :=y_d-y^*-A^*p^*\qquad \text{and} \qquad % \quad&&\text{ in $V^*$}\\
\xi_\rho^* \to \xi^*:= u^*-Ay^*.\label{eq:convMuAndXi}% &&\text{ in $V^*$}
\end{alignat}
Again using monotonicity of $m_\rho$, 
\begin{align*}
\langle \mu_\rho^*, (\Id-\Phi'(y^*_\rho))^{-1}p_\rho^* \rangle %= -\langle A(\Id-\Phi'(y_\rho^*))^{-1}p_\rho^*, p_\rho^* \rangle  (y_\rho^* - y_d, (\Id-\Phi'(y_\rho^*))^{-1}p_\rho^*) 
&= \frac 1\rho \langle m_{\rho}'(y_\rho^*-\Phi(y_\rho^*))^*p_\rho^*, p_\rho^* \rangle \geq 0,
\end{align*}
and taking the limit superior of this, recalling the definition of $\mu^*$, we obtain
\begin{align*}
0 %&\leq \limsup_{\rho \to 0} \langle \mu_\rho^*, (\Id-\Phi'(y^*_\rho))^{-1}p_\rho^* \rangle\\
%  &= \limsup_{\rho \to 0} \langle y_d-y_\rho^*-A^*p_\rho^*, (\Id-\Phi'(y^*_\rho))^{-1}p_\rho^*\rangle\\
&= \limsup_{\rho \to 0} \langle y_d, (\Id-\Phi'(y^*_\rho))^{-1}p_\rho^*\rangle -\liminf_{\rho \to 0} \langle y_\rho^*, (\Id-\Phi'(y^*_\rho))^{-1}p_\rho^*\rangle 
 - \liminf_{\rho \to 0} \langle  A(\Id-\Phi'(y^*_\rho))^{-1}p_\rho^*, p_\rho^*\rangle \\
&\leq   \langle y_d-y^*, (\Id-\Phi'(y^*))^{-1}p^*\rangle -\langle  A(\Id-\Phi'(y^*))^{-1}p^*, p^* \rangle\tag{using the weak semicontinuity results  \eqref{ass:weakLowerSC}, \eqref{ass:doubleLSC} and \eqref{ass:limsupCond}}\\
&= \langle \mu^*, (\Id-\Phi'(y^*))^{-1}p^* \rangle.
%&\geq 0.
\end{align*}
Finally, writing the VI relating $u_\rho^*$ and $p_\rho^*$ in \eqref{eq:penalisedOCSystem} as
\[(\nu u_\rho^*, u_\rho^*-  v)_H - \langle u_\rho^* - v, p_\rho^* \rangle \leq 0 \qquad \forall v \in U_{ad},\]
using the strong convergence of $u_\rho^*$ in $H$ (and hence also in $V^*$) and the weak convergence of $p_\rho^*$ in $V$, we can pass to the limit.

Collecting the results (and recalling that the inverses and adjoints of bounded linear operators commute), we have shown the satisfaction of \eqref{eq:wSS2}--\eqref{eq:wSS3} and 
\begin{align*}
y^* + \mu^* + A^*p^* &= y_d,\\
%Ay^* - u^* + \xi^* &= 0,\\
%u \in U_{ad} : (\nu u^* - p^* , u^*-v ) &\leq 0\quad \forall v \in U_{ad},\\
%\xi^* \geq 0 \text{ in $V^*$}, \quad y^* \leq \Phi(y^*), \quad \langle \xi^*, y^*-\Phi(y^*)\rangle &= 0,\\
\langle (\Id-\Phi'(y^*)^*)^{-1}\mu^*, p^* \rangle &\geq 0,
\end{align*}
Setting $\lambda^* := (\Id-\Phi'(y^*)^*)^{-1}\mu^*$ we get the system \eqref{eq:weakStationaritySystemAgain}.

Thus far, we have only shown the existence of a stationarity point and not that every local minimiser is such a point since we assumed \eqref{eq:locMinCond}. Suppose now that $(y^*,u^*)$ is an arbitrary local minimiser (instead of \eqref{eq:locMinCond}) as claimed in the statement of the theorem. 
%In order to prove that every local minimiser of \eqref{eq:ocProblem} satisfies the stationarity system, we take an arbitrary minimiser $(\bar y,\bar u)$ and 
Denote by $\gamma$ the radius such that $u^*$ is the minimiser on $U_{ad}\cap B^H_\gamma(u^*)$ (the latter object is the closed ball in $H$ of radius $\gamma$ with centre $u^*$). Consider for $\bar J(y_\rho, u) := J(y_\rho,u) + \norm{u-u^*}{H}^2$ the problem
\begin{equation}\label{eq:ocProblemPenBarbu}
\min_{u \in U_{ad} \cap B_\gamma^H(u^*) } \bar J(y_\rho,u) \quad\text{such that}\quad  Ay_\rho + \frac 1\rho m_\rho(y_\rho-\Phi(y_\rho)) = u.
\end{equation}
Denote by $(\bar y_\rho, \bar u_\rho)$ a minimiser of this problem. It follows from $\bar J(\bar y_\rho,  \bar u_\rho) \leq \bar J(y_\rho(u^*), u^*)$ and $\mathbf{P}_\rho(u^*) \ni y_\rho(u^*) \to y^*$ that
\[\limsup_{\rho \to 0}\bar J(\bar y_\rho, \bar u_\rho) \leq J(y^*, u^*).\]
On the other hand, from uniform bounds, we obtain the existence of $\hat u$ such that $\bar u_\rho \weaklyto \hat u$ in $H$ and $\bar y_\rho \to \mathbf{Q}(\hat u) =: \hat y$ in $V$, giving (by the identity $\limsup(a_n) + \liminf(b_n) \leq \limsup(a_n + b_n)$ and using weak lower semicontinuity)
\[\limsup_{\rho \to 0}\bar J(\bar y_\rho, \bar u_\rho) \geq J(\hat y, \hat u) + \limsup_{\rho \to 0}\norm{\bar u_\rho-u^*}{H}^2 \geq J(y^*, u^*) + \limsup_{\rho \to 0}\norm{\bar u_\rho-u^*}{H}^2, \]
with the last inequality because $(y^*, u^*)$ is a local minimiser and $\hat u \in B_\gamma^H(u^*)$. Combining these two inequalities shows that $\hat u = u^*$ and $\bar u_\rho \to u^*$ in $H$. The latter fact implies that for $\rho$ sufficiently small, $\bar u_\rho \in B_\gamma^H(u^*)$ automatically and hence the feasible set in \eqref{eq:ocProblemPenBarbu} can be taken to be just $U_{ad}$. For such $\rho$ (assuming of course that the local conditions in Assumptions \ref{ass:commonAssumptions} and \ref{ass:forEpsAlmostCStat} hold around $y^*$),  the same arguments as above can be used to derive stationarity conditions for \eqref{eq:ocProblemPenBarbu} and in passing to the limit in those conditions, we will find that $(y^*, u^*)$ satisfies the same conditions as above.
\end{proof}
The proof reveals that the stationarity point satisfying \eqref{eq:locMinCond} can be characterised as a limit of the following subsequences (which we have relabelled):
\begin{equation*}%\label{eq:listOfConvW}
\begin{aligned}
y_\rho^* &\to y^* &&\text{in $V,$}\\
u_\rho^* &\to u^* &&\text{in $H,$}\\
p_\rho^* &\weaklyto p^* &&\text{in $V,$}\\
\rho^{-1}m_\rho(y_\rho^*-\Phi(y_\rho^*)) &\to \xi^*&&\text{in $V^*,$}\\
\rho^{-1}m'_\rho(y_\rho^*-\Phi(y_\rho^*))p_\rho^* &\weaklyto \lambda^*&&\text{in $V^*,$}
\end{aligned}
\end{equation*}
where $(y_\rho^*,u_\rho^*,p_\rho^*)$ are as in Lemma \ref{lem:RKZPenal}. 

\subsection{$\mathcal{E}$-almost C-stationarity}\label{sec:penalisationOfQVI}
We specialise to the case where $H$ is an $L^2$ space on a bounded domain with box constraints, which allows us to improve the weak C-stationarity system.
\begin{ass}\label{ass:forEACS}
Let $\Omega \subset \mathbb{R}^n$ be a bounded Lipschitz domain, set $H:=L^2(\Omega)$ and take $V\in \{H^{1}(\Omega), H^{1}_0(\Omega)\}$  and assume the Gelfand triple $(V,H,V^*)$ structure. Finally, we take $U_{ad}$ to be of the box constraint type
\begin{equation}\label{eq:Uad}
U_{ad} = \{ u \in H: u_a \leq u \leq u_b \text{ a.e. in $\Omega$}\}
\end{equation}
for given functions $u_a, u_b \in H$. 
\end{ass}
The assumption can be generalised, see Remark \ref{rem:genEACS}. 

As before, we denote by 
\[\text{$(y^*, u^*)$ an arbitrary local minimiser of \eqref{eq:ocProblem}}.\]
 %Under invertibility and uniformity conditions on $\Id-\Phi'(B_\epsilon(y^*))$, we can derive $\mathcal{E}$-almost C-stationarity conditions as the next theorem shows. %We will also see in Theorem \ref{thm:epsAlmostCStationaryStrongerAss} that such conditions can be obtained under a lower bound assumption in place of the aforementioned condition.
\begin{theorem}[$\mathcal{E}$-almost C-stationarity]\label{thm:ocPenalisation}
%Suppose that there exists a constant $C_0\geq C_b$ such that
%\begin{equation}\label{ass:coercivityForP}
%\langle A^*z,z \rangle  + \frac 1\rho \int_\Omega m_{\rho}'(y_\rho^* - \Phi(y_\rho^*))(I-\Phi'(y_\rho^*))z^2 \geq C_0\norm{z}{V}^2 \quad \forall z \in V.
%\end{equation}
%Suppose that 
%\begin{equation}\label{ass:newAss}
%\exists \epsilon > 0 : \Phi'(z)(v)v \leq C_Pv^2 \text{ a.e. in $\Omega$} \quad \forall z \in B_\epsilon(y^*), \forall v \in V, \text{ where $C_P<1$},
%\end{equation}
Let Assumptions \ref{ass:commonAssumptions}, \ref{ass:forEpsAlmostCStat} and  \ref{ass:forEACS} hold. 
%If, in addition to the assumptions of Theorem \ref{thm:weakStationarity} %(which are Assumption \ref{ass:forEACS}, \eqref{ass:PhiCC}, \eqref{ass:feasiblePoint} and the local assumptions \eqref{eq:assPhiFrechet}, \eqref{ass:existenceForPDE}, and suppose that $\mathbf{Q}$ is single-valued), 
% Assumption \ref{ass:forEACS} also holds, %Let also  hold. 
%Let the assumptions of Lemmas \ref{lem:convPenOCProblems} and \ref{lem:RKZPenal} hold (in particular, $\mathbf{Q}$ must be single-valued). 
Then %there exists %an point $(y^*,u^*)$ for \eqref{eq:ocProblem}, i.e., 
there exist multipliers $(p^*, \xi^*, \lambda^*) \in V \times V^* \times V^*$ satisfying the $\mathcal{E}$-almost C-stationarity system
\begin{subequations}\label{eq:epsAlmostStationaritySystem}
\begin{align}
y^* + (\Id-\Phi'(y^*))^*\lambda^* + A^*p^* &= y_d,\label{eq:eafirst}\\
Ay^* - u^* +  \xi &= 0,\\
\xi^* \geq 0 \text{ in $V^*$}, \quad y^* \leq \Phi(y^*), \quad \langle \xi^*, y^*-\Phi(y^*)\rangle &= 0,\\
u^* \in U_{ad} : (\nu u^* - p^*, u^*-v) &\leq 0 \quad  \forall v \in U_{ad},\\
\langle \xi^*, (p^*)^+ \rangle = \langle \xi^*, (p^*)^- \rangle &= 0\label{eq:eaXiP}\\
\langle \lambda^*, p^* \rangle &\geq 0, \quad \langle \lambda^*, y^*-\Phi(y^*)\rangle = 0,\label{eq:eaCs}\\
\forall \tau > 0, \exists E^\tau \subset \mathcal{I} \text{ with } |\mathcal{I} \setminus E^\tau| \leq \tau : \langle \lambda^*, v \rangle &= 0 \quad \forall v \in V  : v = 0 \text{ a.e. on $\Omega \setminus E^\tau$}.\label{eq:eaEa}
\end{align}
\end{subequations}
In addition, 
%\begin{enumerate}[label=(\roman*)]\vspace{-.2cm}
%\item every local minimiser of \eqref{eq:ocProblem} is an $\mathcal{E}$-almost C-stationarity point,\vspace{-.2cm}
%\item 
if $u_a, u_b \in V$ then the optimal control has the regularity $u^* \in V$.%,\vspace{-.2cm}
%\item $(y^*,u^*,p^*,\xi^*,\lambda^*)$ can be characterised as in \eqref{eq:listOfConvW}.
%\end{enumerate}
\end{theorem}
To prove the theorem, we choose a particular $m_\rho$ (that appeared in the work of Hinterm\"uller and Kopacka \cite{MR2822818} for VIs), namely the superposition operator defined through the real-valued function %$m_\rho\colon \mathbb{R} \to \mathbb{R}$ by
\begin{equation}\label{eq:mrhoHK}
m_\rho(r) \equiv \max_{\epsilon(\rho)}(0,\cdot) := \begin{cases}
0 &: r \leq 0\\
\frac{r^2}{2\epsilon} &: 0 < r < \epsilon\\
r-\frac{\epsilon}{2} &: r \geq \epsilon;
\end{cases}
\end{equation}
here, $\epsilon = \epsilon(\rho) > 0$ is chosen such that $\{\epsilon(\rho)\}$ is bounded. The parameter $\epsilon$ is a smoothing parameter utilised for ensuring differentiability at $0$. 
By \cite[Lemmas 2.83,  2.87, 2.88, 2.90]{CarlNonsmooth} and the fact that $m_\rho \in C^1(\mathbb{R})$ with $m_\rho' \in [0,1]$, we obtain relevant lattice properties  for the spaces involved and differentiability properties for $m_\rho$. That $m_\rho$ satisfies \eqref{ass:mrPenal}, \eqref{ass:mrMonotonicity} and \eqref{ass:mrCts} is clear. Let us check condition \eqref{ass:mrForFeas}. Since $\{\epsilon(\rho)\}$ is bounded, we have (for a subsequence that we relabelled) $\epsilon(\rho) \to \bar \epsilon$ for some $\bar \epsilon\geq 0$ and we get
 \begin{align*}
\norm{\max_{\bar\epsilon}(0, z) - \max_{\epsilon(\rho)}(0, z_\rho)}{V^*} &\leq C\left(\norm{\max_{\bar\epsilon}(0, z) - \max_{\bar\epsilon}(0, z_\rho)}{H}  + \norm{\max_{\bar\epsilon}(0, z_\rho) - \max_{\epsilon(\rho)}(0, z_\rho )}{H}\right)\\
&\leq C\left(\norm{z - z_\rho}{H} + \frac 32|\bar \epsilon - \epsilon(\rho)|\right)\\
&\to 0
\end{align*}
with the final inequality due to Lipschitz properties given in \cite[Lemma 2.1 (v), (vi)]{MR2822818} and the convergence due to the compact embedding $V \ctsCompact H$ . Hence we find $z \leq 0.$ Finally, by the regularity of $m_\rho$ (which has a bounded derivative) we have that $m_\rho\colon H^1(\Omega) \to H$ is $C^1$ (see, e.g. \cite[Proposition 4]{MR2943603}), thus we have \eqref{ass:mrDiff}. This shows that $m_\rho$ is a valid choice.

%We point out that $V \in \{W^{1,p}(\Omega), W^{1,p}_0(\Omega)\}$  for $p \in [2,\infty)$ is valid (this follows by \cite[Lemmas 2.83,  2.87, 2.88, 2.90]{CarlNonsmooth} and the fact that $m_\rho \in C^1(\mathbb{R})$ with $m_\rho' \in [0,1]$).  

\begin{remark}\label{rem:genEACS}Assumption \ref{ass:forEACS} can be generalised as follows. 
Let $\Omega \subset \mathbb{R}^n$ be a bounded Lipschitz domain, set $H:=L^2(\Omega)$ and take $V$ to be a separable Hilbert space with $V \ctsCompact H$ and $(V,H,V^*)$ a Gelfand triple. 
%Let $m_\rho(\cdot)\equiv \max_{\epsilon(\rho)}(0,\cdot)$ be the following regularisation of the positive part function $(\cdot)^+$:
We assume that $V$ is such that $(\cdot)^+\colon V \to V$ is continuous and that the superposition operator $m_\rho$ takes $V$ into $H$ with
$m_\rho\colon V \to H$ being $C^1$.

The requirement for the Nemytskii operator to be Fr\'echet differentiable is in general a delicate issue.
\end{remark}
%This choice of $m_\rho$ is the so-called global penalisation used in \cite{MR2822818}.  
\begin{proof}[Proof of Theorem \ref{thm:ocPenalisation}]
Elements of the proof are similar to that of \cite[Theorem 3.4]{MR2822818} but the more complicated problem structure in this paper requires additional work. 

\medskip

\noindent\textit{1. Weak C-stationarity.} 
Observing that Assumption \ref{ass:forEACS} implies \eqref{ass:forWeakS}, \eqref{ass:mrCts} and \eqref{ass:mrDiff} (as discussed above),  we have the weak C-stationarity result of Theorem \ref{thm:weakStationarity} immediately at hand. 

\medskip

\noindent\textit{2. Regularity of optimal control.} Owing to the characterisation of the VI relating $u_\rho^*$ and $p_\rho^*$ given in \cite[\S II.3]{Kinderlehrer}, thanks to the strong convergence in $H$ of $p_\rho^*$ and continuity of $(\cdot)^+\colon H \to H$, we find that
\[u_\rho^* = \frac 1\nu p_\rho^* + \left(u_a-\frac{p_\rho^*}{\nu}\right)^+-\left(\frac{p_\rho^*}{\nu}-u_b\right)^+ \to \frac 1\nu p^*+ \left(u_a-\frac{p^*}{\nu}\right)^+-\left(\frac{p^*}{\nu}-u_b\right)^+ = u^*.\]
%in $H$ thanks to . Thus, 
It follows that $u^* \in V$ if $u_a$ and $u_b$ belong to $V$.
%Define 
%\begin{align*}
%\mu_\rho^* &:= \frac 1\rho m_{\rho}'(y_\rho^*-\Phi(y_\rho^*))(\Id-\Phi'(y_\rho^*))p_\rho^* = y_d-y_\rho^* - A^*p_\rho^*,\\
%\xi_\rho^* &:= \frac 1\rho m_\rho(y_\rho^*-\Phi(y_\rho^*)) = u_\rho^* - Ay_\rho^*,
%\end{align*}
%which, since their right-hand sides converge, satisfy the following convergences both in $V^*$:
%\begin{alignat*}{3}
%\mu_\rho^* &\weaklyto \mu^* :=y_d-y^*-A^*p^*\qquad \text{and} \qquad % \quad&&\text{ in $V^*$}\\
%\xi_\rho^* \to \xi^*:= u^*-Ay^*.% &&\text{ in $V^*$}
%\end{alignat*}

\medskip

\noindent\textit{3. Orthogonality condition.} For the condition on $y^*-\Phi(y^*)$ in \eqref{eq:eaCs}, observe that since $m_\rho'$ vanishes on $(-\infty,0]$,
\[\langle \mu_\rho^*, (I-\Phi'(y_\rho^*))^{-1}(y_\rho^* - \Phi(y_\rho^*))^- \rangle = \frac 1\rho\int_\Omega m_{\rho}'(y_\rho^*-\Phi(y_\rho^*))^*p_\rho^* (y_\rho^*-\Phi(y_\rho^*))^-= 0,\]
which, due to the continuity of $(\cdot)^-\colon V \to V$ and the joint sequential continuity result of \eqref{ass:doubleCont} implies that
\[%\langle (\Id-\Phi'(y^*)^*)^{-1}\mu^*, (y^*- \Phi(y^*))^- \rangle = 
\langle \mu^*, (\Id-\Phi'(y^*))^{-1}(y^*- \Phi(y^*))^- \rangle =0,\]
and since $y^* \leq \Phi(y^*)$, the negative part above can be dropped. %we have shown that $\langle (\Id-\Phi'(y^*)^*)^{-1}\mu^*, y^*-\Phi(y^*) \rangle = 0.$

\medskip

\noindent\textit{4. $\mathcal{E}$-almost statement.} Since $y_\rho^* \to y^*$ in $V$, $y_\rho^*-\Phi(y_\rho^*) \to y^*-\Phi(y^*)$ pointwise a.e. in $\Omega$ for a subsequence that we do not relabel. Take $x \in \Omega$ such that $y^*(x)-\Phi(y^*)(x) < 0$, then there exists a $\hat \rho=\hat \rho(x)$ such that if $\rho \leq \hat \rho$, then
\[y_\rho(x)-\Phi(y_\rho)(x) \leq \frac12 (y^*(x)-\Phi(y^*)(x)) < 0\]
and hence $\rho^{-1}m_\rho'(y_\rho(x)-\Phi(y_\rho)(x)) = 0$ for $\rho \leq \hat \rho$. That is, $\rho^{-1}m_\rho'(y_\rho(x)-\Phi(y_\rho)(x)) \to 0$ pointwise a.e. on $\{y^* < \Phi(y^*)\}$ and by Egorov's theorem, for every $\tau > 0$, there exists $B^\tau \subset \{y^* < \Phi(y^*)\}$ with $|B^\tau| < \tau$ such that this convergence also holds uniformly on $\{y^* < \Phi(y^*)\}\setminus B^\tau$. 

Take $v \in V$ with $v=0$ a.e. on $\{y^*=\Phi(y^*)\}\cup B^\tau$. By the uniform convergence, for any $\gamma > 0$, there exists $\bar\rho$ such that if $\rho \leq \bar\rho$,
\begin{align*}
\left|\langle \mu_\rho^*, (\Id-\Phi'(y_\rho^*))^{-1}v \rangle\right| = \left|\int_{\{y^* < \Phi(y^*)\} \cap (B^\tau)^c} \frac 1\rho m_\rho'(y_\rho-\Phi(y_\rho))p_\rho^*v\right| \leq \gamma\norm{p_\rho^* v}{L^1(\Omega)}.
\end{align*}
The norm on the right-hand side is bounded uniformly and the left-hand side converges to $|\langle \mu^*, (\Id-\Phi'(y^*))^{-1}v \rangle|$ (thanks to $\mu^*_\rho \weaklyto \mu^*$ in $V^*$ from \eqref{eq:convMuAndXi} and the strong convergence of $(\Id-\Phi'(y_\rho^*))^{-1}v$ in $V$ given by  \eqref{ass:doubleCont}), thus giving 
\begin{align*}
\left|\langle \mu^*, (\Id-\Phi'(y^*))^{-1}v \rangle\right| \leq  C\gamma
\end{align*}
for a constant $C>0$. Since this holds for every $\gamma$, we obtain \eqref{eq:eaEa} (simply set $E^\tau := \mathcal{I}\setminus B^\tau$).

\medskip

\noindent\textit{5. Relation between $\xi^*$ and $p^*$.} 
%This gives $\mu_\rho^*(x) = 0$ for all such $\rho$. That is, $\mu_\rho^* \to 0$ pointwise a.e. on $\{y^* < \Phi(y^*)\}$, and then applying Egorov's theorem gives the last statement of the system.
%We emphasise that we have not been able to obtain the condition
%\[p^* = 0 \text{ a.e. in $\{\xi^* > 0\}$}\]
%on the adjoint state (this corresponds to equation (3.2e) of \cite{MR2822818} in the VI case), hence the adjective `almost' given to the subtitle of this section. Let us give some remarks as to where the attempt to obtain the condition breaks down. 
%\begin{remark}
In order to show the remaining statement \eqref{eq:eaXiP}, %relating $\xi^*$ to $p^*$, 
let us introduce the sets
\[M_1(\rho) := \{ 0 \leq y_\rho^* - \Phi(y_\rho^*) < \epsilon\} \qquad\text{and}\qquad M_2(\rho) := \{y_\rho^* - \Phi(y_\rho^*) \geq \epsilon\}.\]
Since $\langle \xi_\rho^*, y_\rho^*-\Phi(y_\rho^*) \rangle \to \langle \xi^*, y-\Phi(y) \rangle = 0$,  we find
\begin{align}
\nonumber (\xi_\rho^*, y_\rho^* - \Phi(y_\rho^*))&=\frac 1\rho \int_\Omega m_\rho(y_\rho^*-\Phi(y_\rho^*))(y_\rho^* - \Phi(y_\rho^*))\\
&= \frac 1\rho \int_{M_1(\rho)}\frac{(y_\rho^* - \Phi(y_\rho^*))^3}{2\epsilon} + \frac 1\rho\int_{M_2(\rho)}\left(y_\rho^* - \Phi(y_\rho^*)-\frac \epsilon 2\right)(y_\rho^* - \Phi(y_\rho^*))\label{eq:second}\\
\nonumber &\to 0,
\end{align}
and as both integrands in \eqref{eq:second} are non-negative, each integral must individually converge to zero too. Hence
\begin{equation}\label{eq:pr1}
\norm{\frac{\chi_{M_1(\rho)}(y_\rho^* - \Phi(y_\rho^*))^{\frac 32}}{\sqrt{\rho \epsilon}}}{} \to 0\qquad\text{and}\qquad \norm{\frac{\chi_{M_2(\rho)}(y_\rho^* - \Phi(y_\rho^*)-\frac\epsilon 2)}{\sqrt{\rho}}}{} \to 0,
\end{equation}
where for the second convergence we used the fact that $y_\rho^* - \Phi(y_\rho^*) \geq y_\rho^* - \Phi(y_\rho^*)-\epsilon\slash 2 \geq 0$. % in the second integral in \eqref{eq:second}. %and thus the second integral of \eqref{eq:second} exceeds
%\[\int_{M_2(\rho)}\left(y_\rho^* - \Phi(y_\rho^*)-\frac \epsilon 2\right)(y_\rho^* - \Phi(y_\rho^*)) \geq \int_{M_2(\rho)}\left(y_\rho^* - \Phi(y_\rho^*)-\frac \epsilon 2\right)^2.\]
%\[\norm{\frac{\chi_{M_2(\rho)}(y_\rho^* - \Phi(y_\rho^*)-\frac\epsilon 2)}{\sqrt{\rho}}}{L^2(\Omega)} \to 0.\]
%\subsection{The case $U_{ad} = U$}
We calculate
\begin{align}
\nonumber \langle \xi_\rho^*, p_\rho^* \rangle &= \frac 1\rho \int_{M_1(\rho)}\frac{(y_\rho^*-\Phi(y_\rho^*))^2}{2\epsilon}p_\rho^*  + \frac 1\rho \int_{M_2(\rho)}\left(y_\rho^*-\Phi(y_\rho^*)-\frac \epsilon 2\right)p_\rho^*\\
\nonumber &=  \frac 12\int_\Omega \chi_{M_1(\rho)}\frac{(y_\rho^*-\Phi(y_\rho^*))^{3\slash 2}}{\sqrt{\rho\epsilon}}\frac{(y_\rho^*-\Phi(y_\rho^*))^{1\slash 2}}{\sqrt{\rho\epsilon}}\chi_{M_1(\rho)}p_\rho^*  + \int_\Omega \frac{\chi_{M_2(\rho)}\left(y_\rho^*-\Phi(y_\rho^*)-\frac \epsilon 2\right)}{\sqrt{\rho}}\frac{\chi_{M_2(\rho)}p_\rho^*}{\sqrt{\rho}}\\
&\leq \frac 12\norm{\chi_{M_1(\rho)}\frac{(y_\rho^*-\Phi(y_\rho^*))^{3\slash 2}}{\sqrt{\rho\epsilon}}}{}\norm{\frac{(y_\rho^*-\Phi(y_\rho^*))^{1\slash 2}}{\sqrt{\rho\epsilon}}\chi_{M_1(\rho)}p_\rho^*}{}  + \norm{\frac{\chi_{M_2(\rho)}\left(y_\rho^*-\Phi(y_\rho^*)-\frac \epsilon 2\right)}{\sqrt{\rho}}}{}\norm{\frac{\chi_{M_2(\rho)}p_\rho^*}{\sqrt{\rho}}}{}.\label{eq:productXiP}
\end{align}
Now, using \eqref{eq:pr1}, the first factor in each term above converges to zero and hence the above right-hand side will converge to zero if we are able to show that the second factor in each term remains bounded.
Since $\mu_\rho^*$ and $(\Id-\Phi'(y_\rho^*))^{-1}p_\rho^*$ are bounded (the latter due to \eqref{ass:newPhiInvBoundedUniform}), so is their duality product, and therefore
\begin{align*}
\nonumber C &\geq |\langle \mu_\rho^*, (\Id-\Phi'(y_\rho^*))^{-1}p_\rho^* \rangle|\\
\nonumber  &=\frac 1\rho \left|\int_\Omega  m_\rho'(y_\rho^*-\Phi(y_\rho^*))(p_\rho^*)^2\right|\\
\nonumber &= \frac 1\rho \left|\int_{M_1(\rho)}\frac{y_\rho^*-\Phi(y_\rho^*)}{\epsilon}(p_\rho^*)^2+ \int_{M_2(\rho)}(p_\rho^*)^2\right|\\
&=\frac{1}{\rho}\int_{\Omega}\chi_{M_1(\rho)}\frac{y_\rho^*-\Phi(y_\rho^*)}{\epsilon}(p_\rho^*)^2 + \frac{1}{\rho}\int_{\Omega}\chi_{M_2(\rho)}(p_\rho^*)^2.%\label{eq:ForCalc}
\end{align*}
%Let us further assume that $\Phi'$ is such that both integrands on the right-hand side above are non-negative, so that
%\begin{align*}
%\norm{\chi_{M_1(\rho)}\sqrt{\frac{y_\rho^*-\Phi(y_\rho^*)}{\rho\epsilon}(I-\Phi'(y_\rho^*))(p_\rho^*)^2}}{L^2(\Omega)} \leq C  \quad \text{and} \quad \norm{\chi_{M_2(\rho)}\sqrt{\frac 1\rho(I-\Phi'(y_\rho^*))(p_\rho^*)^2}}{L^2(\Omega)}\leq C.
%\end{align*}
%In the VI case where $\Phi' \equiv 0$, these two bounds are sufficient to control \eqref{eq:productXiP} above. In our QVI case, we seem to require additional assumptions. %\textcolor{red}{Give the working out here to control one of the terms.}
%Therefore,
%\begin{align*}
%\frac{1-C_P}{\rho}\left(\int_{\Omega}\chi_{M_1(\rho)}\frac{y_\rho^*-\Phi(y_\rho^*)}{\epsilon}(p_\rho^*)^2 + \int_{\Omega}\chi_{M_2x(\rho)}(p_\rho^*)^2\right) &= \frac{1-C_P}{\rho}\int_{M_1(\rho)}\frac{y_\rho^*-\Phi(y_\rho^*)}{\epsilon}(p_\rho^*)^2 + \frac{1-C_P}{\rho}\int_{M_2}(p_\rho^*)^2\\
%&\leq 
%\frac 1\rho \int_{M_1(\rho)}\frac{y_\rho^*-\Phi(y_\rho^*)}{\epsilon}(I-\Phi'(y_\rho^*))(p_\rho^*)^2 \\
%&\quad+ \frac 1\rho\int_{M_2}(I-\Phi'(y_\rho^*))(p_\rho^*)^2\\
%&\leq C.
%\end{align*}
Both of the terms on the right-hand side are individually bounded uniformly in $\rho$ as the integrands are non-negative. This fact then implies from \eqref{eq:productXiP} that 
\[\langle \xi^*, p^* \rangle = 0.\]
Replacing $p_\rho^*$ by $(p_\rho^*)^+$ in \eqref{eq:productXiP} and in the above calculation, we also obtain in the same way (utilising the fact that $v_n \weaklyto v$ in $V$ implies that $v_n^+ \weaklyto v^+$ in $V$)
\[\langle \xi^*, (p^*)^+ \rangle = 0.\]

\medskip

\noindent\textit{Conclusion.} Finally, setting $\lambda^* := (\Id-\Phi'(y^*)^*)^{-1}\mu^*$, we have shown the desired system
%\begin{align*}
%%y^* + \mu^* + A^*p^* &= y_d,\\
%%Ay^* - u^* + \xi^* &= 0,\\
%%u \in U_{ad} : (\nu u^* - p^* , u^*-v ) &\leq 0\quad \forall v \in U_{ad},\\
%%\xi^* \geq 0 \text{ in $V^*$}, \quad y^* \leq \Phi(y^*), \quad \langle \xi^*, y^*-\Phi(y^*)\rangle &= 0,\\
%%\langle (\Id-\Phi'(y^*)^*)^{-1}\mu^*, p^* \rangle \geq 0, \quad 
%\langle \lambda^*, y^*-\Phi(y^*) \rangle &= 0,\\
%\langle \xi^*, (p^*)^+ \rangle = \langle \xi^*, (p^*)^- \rangle &= 0\\
%\forall \tau > 0, \exists E^\tau \subset \{y^* <  \Phi(y^*) \} \text{ with } |\{y^* <  \Phi(y^*) \}\setminus E^\tau| \leq \tau : \langle \lambda^*, v\rangle =0 \quad &\forall v \in V  : v = 0 \text{ a.e. on $\Omega \setminus E^\tau$},
%\end{align*}
%and we get the system 
\eqref{eq:epsAlmostStationaritySystem}.
\end{proof}
We conclude this section by showing that the alternative (stronger) condition \eqref{eq:strongerpenult} occasionally used in literature for defining a C-stationarity point can be achieved under additional assumptions.
\begin{prop}[Satisfaction of alternative criterion in C-stationarity]\label{prop:altCondition}For $q_\rho \weaklyto q$ in $V$, under the conditions of %assume either the assumptions of 
Theorem \ref{thm:ocPenalisation} and
\begin{equation}\label{ass:wlscForPhiInvForSCS}
\liminf_{n \to \infty} \langle A^*q_\rho, (\Id-\Phi'(y_\rho^*))^{-1}(\psi q_\rho) \rangle  \geq \langle A^*q, (\Id-\Phi'(y^*))^{-1}(\psi q) \rangle\quad \forall \psi \in W^{1,\infty}(\Omega) \text{ with } \psi \geq 0,
\end{equation}
%or the assumptions of Theorem \ref{thm:epsAlmostCStationaryStrongerAss} and 
%\begin{equation}\label{ass:wlscForPhiInvForSCSOld}
%\liminf_{n \to \infty} \langle A^*q_\rho, \psi q_\rho \rangle  \geq \langle A^*q, \psi q \rangle\quad \forall \psi \in W^{1,\infty}(\Omega) \text{ with } \psi \geq 0.
%\end{equation}
the inequality condition in \eqref{eq:eaCs} can be strengthened to
\[\langle \lambda^*, \psi p^*\rangle \geq 0 \quad \forall \psi \in W^{1,\infty}(\Omega) \text{ with } \psi \geq 0.\]
\end{prop}
\begin{proof}
%Consider the first case. 
Testing the equation for $p_\rho^*$ with $(\Id-\Phi'(y_\rho^*))^{-1}(\psi p_\rho^*) $, noticing that $\psi p_\rho^* \weaklyto \psi p^*$ in $V$ and making use again of \eqref{ass:doubleLSC} and \eqref{ass:limsupCond} in a similar way to the proof of Theorem \ref{thm:weakStationarity},
\begin{align*}
%\limsup_{\rho \to 0}\langle \lambda_\rho^*, \psi p_\rho^* \rangle  &= 
\limsup_{\rho \to 0}\langle \mu_\rho^*, (\Id-\Phi'(y_\rho^*))^{-1}(\psi p_\rho^*) \rangle  %&= \limsup_{\rho \to 0}\langle y_d - A^*p_\rho^* - y_\rho^*, (\Id-\Phi'(y_\rho^*))^{-1}(\psi p_\rho^*)\rangle\\
&= \limsup_{\rho \to 0}\langle y_d, (\Id-\Phi'(y_\rho^*))^{-1}(\psi p_\rho^*)\rangle  -\liminf_{\rho \to 0}\langle  y_\rho^*, (\Id-\Phi'(y_\rho^*))^{-1}(\psi p_\rho^*)\rangle\\
&\quad - \liminf_{\rho \to 0}\langle A^*p_\rho^*, (\Id-\Phi'(y_\rho^*))^{-1}(\psi p_\rho^*) \rangle\\
&\leq \langle y_d - y^*, (\Id-\Phi'(y^*))^{-1}(\psi p^*) \rangle  - \langle A^*p, (\Id-\Phi'(y^*))^{-1}(\psi p^*) \rangle\tag{using \eqref{ass:wlscForPhiInvForSCS} for the last term}\\
&= \langle \mu^*, (\Id-\Phi'(y^*))^{-1}(\psi p^*)\rangle\\
&= \langle \lambda^*, \psi p^*\rangle.
\end{align*}
On the other hand, we have
\[ \limsup_{\rho \to 0}\langle \mu_\rho^*, (\Id-\Phi'(y_\rho^*))^{-1}(\psi p_\rho^*) \rangle = \limsup_{\rho \to 0}  \langle \lambda_\rho^*, \psi p_\rho^* \rangle  = \limsup_{\rho \to 0} \int_\Omega m_\rho'(y_\rho^*-\Phi(y_\rho^*))(p_\rho^*)^2\psi\geq 0\]
which implies the result. %The second case follows the same argument with the test function being chosen as $\psi\rho_\rho^*$.%i.e.,
%\begin{align}
%\langle \lambda^*, \psi p^*\rangle \geq 0 \qquad \forall \psi \in W^{1,\infty}(\Omega), \psi \geq 0.
%\end{align}
\end{proof}
\begin{remark}
Let us consider when assumption  \eqref{ass:wlscForPhiInvForSCS} of the previous proposition holds.  Suppose that $A$ is of the form 
\begin{equation}\label{eq:ellipticOperator}
\langle Au, v \rangle =  \sum_{i,j=1}^n \int_\Omega a_{ij}\frac{\partial u}{\partial x_i}\frac{\partial v}{\partial x_j} + \sum_{i=1}^n \int_\Omega b_{i}\frac{\partial u}{\partial x_i}v + \int_\Omega c_0 uv\qquad \forall u, v \in V,
\end{equation} with $a_{ij}=a_{ji} \in C^{0,1}(\bar\Omega)$, $b_i \in W^{1,\infty}(\Omega)$, $c_0 \in L^\infty(\Omega)$ and 
\begin{equation}\label{eq:ellipticity}
\sum_{i,j=1}^n a_{ij}\xi_i\xi_j \geq C|\xi|^2 \quad\text{a.e.}
\end{equation}
for some $C>0$ and $c_0 \geq \lambda > 0$ a.e. with $\lambda$ a constant such that $A$ is coercive.  %and testing the equation for $p_\rho^*$ with $\psi p_\rho^*$, we obtain
%\begin{align*}
%\langle A^*q_\rho, \psi q_\rho \rangle  &=\langle q_\rho, A(\psi q_\rho) \rangle \\
%&= \sum_{i,j=1}^n \int_\Omega a_{ij}\left(\psi\frac{\partial q_\rho}{\partial x_i}\frac{\partial q_\rho}{\partial x_j} + q_\rho\frac{\partial \psi}{\partial x_i}\frac{\partial q_\rho}{\partial x_j}\right)  + \sum_{i=1}^n \int_\Omega b_{i}\left(\psi\frac{\partial q_\rho}{\partial x_i}q_\rho + \frac{\partial \psi}{\partial x_i}|q_\rho|^2\right) + \int_\Omega c_0 \psi|q_\rho|^2.
%\end{align*}
%Using the convergences $q_\rho \weaklyto q$ in $V$ and the regularity of $\psi$, it is easy to pass to the limit in all but the first term which can be handled by weak lower semicontinuity to give \eqref{ass:wlscForPhiInvForSCSOld}.
%%\begin{align*}
%%\liminf_{\rho \to 0} \langle A^*p_\rho^*, \psi p_\rho^* \rangle  \geq \langle A^*p^*, \psi p^* \rangle.
%%\end{align*}

Taking $\psi$ as in the above proposition, let $z_\rho = (\Id-\Phi'(y_\rho^*))^{-1}(\psi q_\rho)$. By \eqref{eq:otherOne}, $z_\rho \weaklyto z := (\Id-\Phi'(y^*))^{-1}(\psi q^*)$ in $V$.  We have, as done in \cite[Lemma 3.6]{MR3056408} and \cite[Lemma 4.5]{MR3484394},
\begin{align*}
\langle A^*q_\rho, (\Id-\Phi'(y_\rho^*))^{-1}(\psi q_\rho) \rangle &= \langle A^*q_\rho, z_\rho \rangle\\
  &=\langle q_\rho, Az_\rho \rangle \\
&= \sum_{i,j=1}^n \int_\Omega a_{ij} \frac{\partial z_\rho}{\partial x_i}\frac{\partial q_\rho}{\partial x_j}   + \sum_{i=1}^n \int_\Omega b_{i} \frac{\partial z_\rho}{\partial x_i}q_\rho  + \int_\Omega c_0 z_\rho q_\rho.
\end{align*}
Using the convergences $q_\rho \weaklyto q$ and $z_\rho \weaklyto z$ in $V$, the compactness of $V \ctsCompact H$ and the regularity of $\psi$, it is easy to pass to the limit in all but the first term. For that term, we need a weak lower semicontinuity of the form
\[\liminf_{\rho \to 0}\sum_{i,j=1}^n \int_\Omega a_{ij} \frac{\partial ((\Id-\Phi'(y_\rho^*))^{-1}(\psi q_\rho))}{\partial x_i}\frac{\partial q_\rho}{\partial x_j} \geq \sum_{i,j=1}^n \int_\Omega a_{ij} \frac{\partial ((\Id-\Phi'(y))^{-1}(\psi q))}{\partial x_i}\frac{\partial q}{\partial x_j}.\]
A condition ensuring this is the complete continuity of $\Id-\Phi'(y)\colon V \to V$ (examining the proof of Lemma \ref{lem:technicalLemma} shows that this condition would turn the convergence in \eqref{eq:otherOne} into a strong convergence so that $z_\rho \to z$ in $V$ and hence we can directly pass to the limit in that term).

\end{remark}
\subsection{From $\mathcal{E}$-almost to C-stationarity}\label{sec:epsAlmostToC}
In order to upgrade to C-stationarity, we need an additional condition given in the next proposition. The assumption preserves generality but is strong, however, we will explore an example below of a reasonable situation where it holds.
\begin{prop}[C-stationarity]\label{prop:CStationarity}Let the assumptions of Theorem \ref{thm:ocPenalisation} % or Theorem \ref{thm:epsAlmostCStationaryStrongerAss} 
 hold and  assume that 
\[y_\rho^* -\Phi(y_\rho^*) \to y^*-\Phi(y^*) \text{ in $L^\infty(\Omega)$}.\]
%take the dimension $n \leq 4$ and suppose that the (bounded Lipschitz) domain $\Omega$ and operator $A$ are such that\footnote{This is an elliptic regularity condition. When $\Omega$ is a $C^{1,1}$ domain and $A$ is of the form \eqref{eq:ellipticOperator} with $a_{ij} \in C^0(\bar\Omega) \cap W^{1,\infty}(\Omega)$, $b_i, c_0 \in L^\infty(\Omega)$, $c_0 \geq 0$ with the strict ellipticity \eqref{eq:ellipticity}, Theorem 9.15 of \cite{Trudinger1983} can be applied and it implies the required condition.}
%\[y \in H^1_0(\Omega), Ay \in L^2(\Omega) \implies y \in H^2(\Omega).\]
%also that $\Phi\colon V \cap C^{0,\alpha}(\bar\Omega) \to L^\infty(\Omega) \text{ is continuous.}$  
Then \eqref{eq:eaEa} can be strengthened to
\[\langle \lambda^*, v\rangle =0 \quad \forall v \in V  : v = 0 \text{ a.e. on $\{y^*=\Phi(y^*)\}$}.\]
\end{prop}
\begin{proof}
%Recalling the equation
%$Ay_\rho + ({1}\slash{\rho})m_\rho(y_\rho-\Phi(y_\rho)) = u_\rho,$
%we see that $Ay_\rho \in H$ and hence $y_\rho \in H^2(\Omega)$. 
%By a Sobolev embedding \cite[Theorem 6.3]{Adams}, $H^2(\Omega) \ctsCompact C^{0,\alpha}(\bar\Omega)$ for some $\alpha \in (0,1)$, giving $y_\rho^* \to y_\rho$ in that H\"older space (and thus pointwise uniformly). 
%By assumption, $\Phi(y_\rho^*) \to \Phi(y^*)$ uniformly a.e. 
By assumption, the convergence of $y_\rho^*-\Phi(y_\rho^*)$ to $y^*-\Phi(y^*)$ is uniform and hence $\rho^{-1}m_\rho'(y_\rho(x)-\Phi(y_\rho)(x)) \to 0$ uniformly a.e.  globally on $\{y^* < \Phi(y^*)\}$. This means that the argument in the proof of Theorem \ref{thm:ocPenalisation} can be repeated without recourse to Egorov's theorem.
%By \cite[Theorem 6.3]{Adams}, $H^2(\Omega) \ctsCompact C_B^0(\Omega)$, the space of continuous bounded functions with the sup norm. It follows that
%\[y_\rho^* \to y^* \quad\text{in $L^\infty(\Omega)$.}\]
\end{proof}
Sobolev embeddings are the most obvious paths to achieve the assumption of the above proposition. We demonstrate this now with an example. Take the dimension $n \leq 4$ and suppose that the (bounded Lipschitz) domain $\Omega$ and operator $A$ are such that
\[y \in H^1_0(\Omega) \cap H^2(\Omega) \implies Ay \in L^2(\Omega)\]
and\footnote{These are elliptic regularity conditions. When $\Omega$ is a $C^{1,1}$ domain and $A$ is of the form \eqref{eq:ellipticOperator} with $a_{ij} \in C^0(\bar\Omega) \cap W^{1,\infty}(\Omega)$, $b_i, c_0 \in L^\infty(\Omega)$, $c_0 \geq 0$ with the strict ellipticity \eqref{eq:ellipticity}, Theorem 9.15 of \cite{Trudinger1983} can be applied and it implies the first condition. The second follows from \cite[Lemma 9.17]{Trudinger1983}.}
\begin{align*}
&y \in H^1_0(\Omega), Ay \in L^2(\Omega) \implies 
\begin{cases} y \in H^2(\Omega),\\
\norm{y}{H^2(\Omega)} \leq C(\norm{y}{L^2(\Omega)} + \norm{Ay}{L^2(\Omega)}).
\end{cases}
\end{align*}
We take $V=H^1_0(\Omega)$ and use the fact that $m_\rho \colon V \to V$ (recall that $m_\rho$ has been chosen in \eqref{eq:mrhoHK}; see \cite[\S 2.2.3]{CarlNonsmooth} for when this type of property could hold for other maps).  Suppose that $\Phi \colon V^* \to V$ is given by the solution mapping of an elliptic equation, i.e., $\Phi(y)$ is defined as the solution $\phi$ of
\[B(\phi) = y\]
where $B$ is a second-order elliptic operator with sufficient properties guaranteeing well posedness in $H^1_0(\Omega)$ (with a continuous dependence estimate), and when $y \in L^2(\Omega)$, in the space $H^2(\Omega) \cap H^1_0(\Omega)$ including a regularity estimate of the form
\[\norm{\phi}{H^2(\Omega)} \leq C\norm{y}{L^2(\Omega)}.\]
%Further, suppose that $\Phi\colon V \to H^2(\Omega)$ and that $\Phi \colon V \to H$ is a bounded operator (
Due to this, we immediately have that $\Phi(y_\rho^*) \in H^2(\Omega)$ with a uniform bound:
\begin{equation}\label{eq:egeq1}
\norm{\Phi(y_\rho^*)}{H^2(\Omega)} \leq C_1\norm{y_\rho^*}{L^2(\Omega)} \leq C_2.
\end{equation}
Defining $z=y_\rho^*-\Phi(y_\rho^*) \in H^1_0(\Omega)$, we write the equation for $y_\rho^*$ as
\[Az + \frac 1\rho m_\rho(z) = u_\rho^*-A\Phi(y_\rho^*).\]
It follows from rearranging this equation that $Az \in H$, thus $z \in H^2(\Omega)$ and the equation holds in a pointwise a.e. sense. Suppose for simplicity that $A=-\Delta$ is the Dirichlet Laplacian. Test with $-\Delta z$ and use
\[\int_\Omega m_\rho(z)(-\Delta z) = \int_\Omega m_\rho'(z)|\grad z|^2 \geq 0\] 
to obtain
\[\norm{-\Delta z}{H}^2 \leq \norm{u_\rho^*-A\Phi(y_\rho^*)}{H}\norm{-\Delta z}{H}.\]
Dividing through by $\norm{-\Delta z}{H}$, the resulting right-hand side is bounded due to \eqref{eq:egeq1},  and using
%By the definition of $\Phi$, we have $A\Phi(y_\rho^*) = -\Delta B^{-1}(y_\rho^*)$ and this is bounded in $L^2$:
%\[\norm{A\Phi(y_\rho^*)}{L^2(\Omega)} \leq \norm{B^{-1}(y_\rho^*)}{H^2(\Omega)} \leq C\norm{y_\rho^*}{L^2(\Omega)}.\]
%Using this and the boundedness of $\{y_\rho^*\}$ in $H$ and 
the regularity condition above, we obtain uniform boundedness in $H^2(\Omega)$ of $z=y_\rho^*-\Phi(y_\rho^*)$. By the Sobolev embedding \cite[Theorem 6.3]{Adams} $H^2(\Omega) \ctsCompact C^{0,\alpha}(\bar\Omega)$ for some $\alpha \in (0,1)$, we get $y_\rho^*-\Phi(y_\rho^*) \to y^*-\Phi(y^*)$ in that H\"older space (and thus in $L^\infty(\Omega)$). 

%Suppose that $\Phi\colon V \cap C^{0,\alpha}(\bar\Omega) \to L^\infty(\Omega) \text{ is continuous}$. Recalling the equation
%$Ay_\rho + ({1}\slash{\rho})m_\rho(y_\rho-\Phi(y_\rho)) = u_\rho,$
%we see that $Ay_\rho \in H$ and hence $y_\rho \in H^2(\Omega)$. For simplicity, let us take $A=-\Delta$ the Laplacian.  The equation holds pointwise a.e. and we can multiply it by $-\Delta y_\rho$ and integrate to obtain
%\[\norm{-\Delta y}{H}^2 + \frac 1\rho \int_\Omega m_\rho(y-\Phi(y))\Delta y = \int_\Omega u(-\Delta y)\]
%and now we integrate by parts to obtain
%\begin{align*}
%\int_\Omega m_\rho(y-\Phi(y))\Delta y &= -\int_\Omega m_\rho'(y-\Phi(y))(|\grad y|^2 - \grad \Phi(y)\grad y) \geq 0.
%\end{align*}
%Then if we assume that $\Phi$ satisfies $|\grad y|^2 - \grad \Phi(y)\grad y \leq 0$ then we can neglect this term and use Young's inequality on the right-hand side of the above inequality to obtain a uniform bound on $y_\rho$ in $H^2(\Omega)$ and by the Sobolev embedding \cite[Theorem 6.3]{Adams} $H^2(\Omega) \ctsCompact C^{0,\alpha}(\bar\Omega)$ for some $\alpha \in (0,1)$, we get $y_\rho^* \to y_\rho$ in that H\"older space (and thus pointwise uniformly). 
%
%If for example $\Phi$ is given by $\Phi(y) = k(-\Delta)^{-1}y$, the inverse Laplacian with a coefficient $k > 0$,  we obtain
\subsection{Strong stationarity}\label{sec:ss}
%\subsection{Optimal control problem}
%Consider the problem \eqref{eq:ocProblem} with $V=H^1_0(\Omega)$ where $\Omega \subset \mathbb{R}^n$ is a bounded domain and the admissible set $U_{ad}$ is chosen as \eqref{eq:Uad}, i.e., of box constraint type.
%\begin{equation}\label{eq:ocProblemSS}
%\min_{\substack{u_a \leq u  \leq u_b\\ y \in \mathbf{Q}(u)}}\frac 12\norm{y-y_d}{H}^2 + \frac{\nu}{2}\norm{u}{U}^2.
%\end{equation}
%Setting $G:=-u_a$, we see that $G > 0$ q.e. and $G \in H^1_0(\Omega)$ and so the previous results apply and 
%Let us denote the minimiser by $(y^*, u^*)$ which should exist by the result in \S \ref{sec:existenceOC}.  

We now give strong stationarity conditions for \eqref{eq:ocProblem} in the setting of $V=H^1_0(\Omega)$,  $H=L^2(\Omega)$ and $U_{ad}$ of the box constraint form \eqref{eq:Uad}.

Let us first of all provide some background and context. Strong stationarity for the VI obstacle problem in the absence of constraints on the control was the focus of the classical works by Mignot \cite[Theorem 5.2]{MR0423155} and Mignot and Puel \cite{MR739836}. The approach in the latter work is as follows. By using the results on the differentiability of the solution map associated to VIs of Mignot \cite{MR0423155}, the Bouligand stationarity condition (for example, see Proposition \ref{lem:characterisationOfOC}) reads
\[(\alpha_h,y^*- y_d)_H + \nu (u^*, h)_H \geq 0 \quad \forall h \in H\]
where $\alpha_h$ denotes the directional derivative of the solution map with respect to the direction $h$. The key idea of Mignot and Puel in \cite{MR739836} is to use the fact that the optimal control $u^*$ in fact belongs to $V$ (in the unconstrained case, this follows from B-stationarity; otherwise this is a regularity result in certain situations or one may need to simply assume this) and to extend, by continuity, the above inequality to
\begin{equation}\label{eq:1}
(\alpha_h,y^*- y_d)_H + \nu \langle u^*, h \rangle \geq 0 \quad \forall h \in V^*
\end{equation}
so that the set of feasible directions has been enlarged to $V^*$. Then, by writing the duality product in \eqref{eq:1} as $\langle A A^{-1}h , \nu u^*\rangle$ and using properties of the projection operator with respect to the bilinear form generated by $A$ onto the critical cone, it is shown \cite[Theorem 3.3]{MR739836}  that this inequality is equivalent to a strong stationarity system. 

The presence of control constraints complicates the derivation of strong stationarity conditions. In the VI setting, by using the above-mentioned technique of Mignot and Puel of enlarging the set of feasible directions onto the dual space in combination with a fine analysis of the various resulting objects and sets, strong stationarity conditions for VI optimal control problems subject to box constraints were obtained by Wachsmuth in \cite{Wachsmuth}. The author also showed that certain restrictions are required on the control bounds in order to obtain a positive answer for strong stationarity, and counterexamples were given showing that violating those conditions can lead to a lack of strong stationarity. These necessary conditions (which are stated in \eqref{eqass1}--\eqref{eqass3} below) in the context of admissible sets as in \eqref{eq:Uad} are implied \cite[Lemma 5.3]{Wachsmuth} by the condition
\begin{equation*}
\text{$u_a, u_b \in H^1(\Omega)$ with $u_a < 0 \leq u_b \text{ q.e. on $\Omega$},$}
\end{equation*}
% $u_a, u_b \in H^1_0(\Omega)$ with $u_a < 0 \leq u_b$ q.e. in $\Omega$, 
(recall Example \ref{eg:dirichletCase} for the meaning of q.e.) which in turn implies that the control space must allow for negative functions, meaning that one ultimately needs existence and directional differentiability results for QVIs with source terms and directions that may be strictly negative\footnote{Our theory of differentiability for QVIs in the earlier paper \cite{AHR}  (which was for non-negative sources and directions) could not be immediately used to obtain strong stationarity by arguing in this fashion since the setting of \cite{AHR} would have forced $U_{ad}$ to be selected such that $U_{ad} \subset H_+$. This is why the development of the results of \S \ref{sec:existence} and \S \ref{sec:Diff} are crucial.}. 
%\begin{remark}
%For source terms $f \in U_{ad}$ with $U_{ad}$ satisfying \eqref{ass:uaub}, $\mathbf{Q}(f)$ is well defined through Theorems \ref{thm:existenceUnsigned} or \ref{thm:penalisedConvergence} or Theorem \ref{lem:existenceTransformedQVI} by taking a suitable lower bound function $v_0$. %G\geq -u_a$. 
%The derivatives for directions belonging to $U_{ad}$ also exist by Theorem \ref{thm:dirDiff1}.
%\end{remark}

Let $(y^*,u^*)$ be a local optimal pair of \eqref{eq:ocProblem}.  As in \cite{MR739836}, we make the fundamental assumption that  
%\begin{equation}\label{ass:optControlInV}
$u^* \in V$
%\end{equation}
and we refer to Theorem \ref{thm:ocPenalisation} from the previous section for the satisfaction of this assumption. Let us take $U_{ad}$ as stated in \eqref{eq:Uad} where we include the possibility of taking $u_a=-\infty$ and $u_b = \infty$, in which case the problem becomes one with no constraints and we can argue as in \cite{MR739836}. Outside of this case, we proceed as in \cite{Wachsmuth}. Let the assumptions of  Theorem \ref{thm:dirDiff1} hold and denote by  $j\colon H \to V^*$ the inclusion map through the Riesz isomorphism. Then, as done in \cite{Wachsmuth}, the Bouligand stationarity condition \eqref{eq:Bstationarity} can be extended to 
\begin{equation*}
(\alpha_h,y^*- y_d) + \nu \langle h, u^*\rangle \geq 0 \quad \forall h \in \cl{j\mathcal{T}_{U_{ad}}(u^*)}^{V^*}.
\end{equation*}
%Observe that we need the continuity in $V^*$ of $h \mapsto \alpha_h$ assured by Proposition \ref{lem:alphaUniquenessAndCty} to do this. 
This is starting point of the steps leading to the strong stationarity conditions in \cite{Wachsmuth} for the VI case. 

Defining the (quasi-closed) coincidence sets
\begin{align*}
%\mathcal{A}(y^*):= \{ y^* = \Phi(y^*) \}, \qquad 
U_a := \{ x \in \Omega : u^*(x) = u_a(x)\}\qquad \text{and}\quad U_b := \{ x \in \Omega : u^*(x) = u_b(x)\}
\end{align*}
and arguing identically to the proof of \cite[Lemma 4.3]{Wachsmuth}, we obtain the following sign conditions on $u^*$:
\begin{align*}
u^* &= 0 \text{ q.e. on $\mathcal{A}_s(y^*) \cap (\Omega \setminus (U_a \cup U_b))$},\\
u^* &\leq 0 \text{ q.e. on $\mathcal{A}_s(y^*) \cap U_b$},\\
u^* &\geq 0 \text{ q.e. on $(\mathcal{A}_s(y^*) \cap U_a)\cup (\mathcal{B}(y^*) \cap (\Omega \setminus U_b))$}
\end{align*}
where $\mathcal{B}(y^*) = \mathcal{A}(y^*) \setminus \mathcal{A}_s(y^*)$ is the \emph{biactive set}. 

Let $\mathrm{cap}(A)$ denote the capacity of a Borel subset $A$ of $\Omega$ with respect to $H^1_0(\Omega)$ (see \cite[Definition 6.47]{MR1756264}). We have the following strong stationarity characterisation, the proof of which involves modifications of \cite{Wachsmuth} and is sketched in Appendix \ref{app:1}. %After providing some background and context, we reduce this section to the essence of the statement of the result and relegate the proof to the appendix since it mainly involves modifications of \cite{Wachsmuth}. 

\begin{theorem}[Strong stationarity]\label{thm:strongStationarity}
Let $(y^*,u^*)$ be a local minimiser of \eqref{eq:ocProblem} with $u^* \in V$.

Assume Assumption \ref{ass:forDirDiff}, \eqref{ass:PhiCC},  the local assumptions\footnote{These, of course, should be evaluated at $y^*$.}  
%\[\text{there exists $\epsilon > 0$ such that $\Phi\colon V \to V$ is Hadamard directionally differentiable on $B_\epsilon(y)$,}\] 
\eqref{ass:PhiLipBound}, \eqref{ass:PhiDerivativeLipschitzConstant} and suppose that 
\begin{align}
\nonumber &\text{$\Phi\colon V \to V$ is Fr\`echet differentiable at $y^*$},\\
%\end{align}
%\eqref{ass:PhiBoundedlyDiff},   %\eqref{itm:compContOfDerivOfPhi},    
%\begin{subequations}\label{eqass}
%\begin{align}
%&(\Id-\Phi'(y^*))\colon V \to V \text{ is invertible},\\
&\mathrm{cap}(U_a \cap \mathcal{B}(y^*)) = 0\label{eqass1},\\
&u_b \geq 0 \text{ q.e. on $\mathcal{B}(y^*)$}\label{eqass2},\\
&u^* = 0 \text{ q.e. on $\mathcal{A}_s(y^*)$}\label{eqass3}.
\end{align}%hold. 
Then $(y^*,u^*)$ is a strong stationarity point, i.e., there exist multipliers $(p^*, \xi^*, \lambda^*) \in V\times V^*\times V^*$ such that
\begin{alignat*}{3}
 y^*  +  (\Id-\Phi'(y^*)^*)\lambda^* + A^*p^*  &= y_d,\\
Ay^* - u^* + \xi^* &=0,\\
\xi^* \geq 0 \text{ in $V^*$}, \quad y^* \leq \Phi(y^*), \quad \langle \xi^*, y^*-\Phi(y^*) \rangle&= 0,\\
u^* \in U_{ad} : (\nu u^* - p^* , u^*-v) &\leq 0\quad \forall v \in U_{ad},\\
p^* &\geq 0 \quad\text{q.e. on $\mathcal{B}(y^*)$ and } p^*=0 \text{ q.e. on $\mathcal{A}_s(y^*)$},\\
\langle \lambda^*, v \rangle &\geq 0 \quad \forall v \in V : v \geq 0 \text{ q.e. on $\mathcal{B}(y^*)$ and } v = 0 \text{ q.e. on $\mathcal{A}_s(y^*)$}.
%\sigma &\in \mathcal{N}_{U_{ad}}(u^*)
%\sigma  := p^*-\nu u^*
\end{alignat*}
\end{theorem}
Note also that, whilst this work was under preparation, a related result has recently been obtained in \cite{MR4083198} however only in the absence of control constraints (i.e., $U_{ad}$ is taken to be the whole space).
\appendix
\section{Technical proofs}\label{app:technical}\begin{proof}[Proof\footnotemark of Lemma \ref{lem:mct}]\footnotetext{We thank Jochen Gl\"uck for the idea of the proof.}
Take an arbitrary subsequence $\{v_{n_j}\}$; this remains uniformly bounded hence we can extract a weakly convergent subsequence such that $v_{n_{j_k}} \weaklyto v$ in $V$ to some $v$.  

Select an arbitrary $f \in V^*_+$ and set $l_n := \langle f, v_n \rangle$ which is a monotonic sequence (since $f$ is non-negative) and also bounded. Hence the monotone convergence theorem applies and we obtain the existence of $l$ such that $l_n \to l$. Since also $l_{n_{j_k}} \to l$, we conclude that $l=\langle f, v \rangle$.

Take another subsequence of $\{v_n\}$, say $\{v_{n_{m}}\}$, then by the above argument, we have $v_{n_{m_j}} \weaklyto \hat v$ for some $\hat v$ and $l=\langle f, \hat v \rangle$. That is, 
\[\langle f, v \rangle = \langle f, \hat v \rangle\quad \forall f \in V^*_+,\]
and from this, we can conclude via the weak-* density of $V^*_+ - V^*_+$ in $V^*$ (e.g., see \cite[Lemma 2.7]{MR2506943})   that $\hat v = v$. The subsequence principle then yields the result.
\end{proof}
\begin{proof}[Proof of Lemma \ref{lem:technicalLemma}]
Define $T_n = (\Id-\Phi'(z_n))$ and $T=(\Id-\Phi'(z))$. Then
\begin{align*}
T_n^{-1}q_n - T^{-1}q = (T_n^{-1}-T^{-1})q_n + T^{-1}(q_n-q)
\end{align*}
and we get $T^{-1}(q_n-q) \weaklyto 0$ in $V$ by continuity and linearity of $T^{-1}$.  For the first term on the right-hand side above, we use the identity $T_n^{-1}-T^{-1} = T_n^{-1}(T-T_n)T^{-1}$ relating the inverses of operators to see that
\begin{align*}
\norm{(T_n^{-1}-T^{-1})q_n}{V} &= \norm{T_n^{-1}(T-T_n)T^{-1}q_n}{V}\\
&\leq C_1\norm{(T-T_n)T^{-1}q_n}{V}\tag{by \eqref{ass:newPhiInvBoundedUniform}}\\
&\leq C_1\norm{T-T_n}{\mathcal{L}(V,V)}\norm{T^{-1}q_n}{V}\\
&\leq C_2\norm{\Phi'(z_n)-\Phi'(z)}{\mathcal{L}(V,V)}\tag{because $T^{-1}$ and $q_n$ are bounded}\\
&\to 0
\end{align*}
with the convergence because we assumed that $\Phi$ is continuously Fr\'echet differentiable and hence the derivative is continuous. Therefore, $T_n^{-1}q_n \weaklyto T^{-1}q$ in $V$. The strong convergence follows because if $q_n \to q$ then $T^{-1}(q_n-q) \to 0$ in $V$.
% for any $F \in V^*$,
%\begin{align*}
%%\langle F, T^{-1}(q_n-q) \rangle = \langle (T^{-1})^*F, q_n-q \rangle \to 0
%\end{align*}
For the final claim, we have
\begin{align*}
\langle AT_n^{-1}q_n, q_n \rangle - \langle AT^{-1}q, q \rangle &= 
\langle A(T_n^{-1}q_n-T^{-1}q_n), q_n \rangle + \langle AT^{-1}q_n, q_n \rangle - \langle AT^{-1}q, q \rangle
\end{align*}
and the first term on the right-hand side tends to zero by the calculation above. Since by \eqref{ass:newPhiInvBoundedUniform}, $AT^{-1}$ is bounded and coercive (as well as being linear), we obtain
\[\liminf_{n \to \infty} \langle AT^{-1}q_n, q_n \rangle - \langle AT^{-1}q, q \rangle \geq 0.\qedhere\]
\end{proof}
\section{Sketch proof of Theorem \ref{thm:strongStationarity}}\label{app:1}
Recall the notation $\alpha_h$ which stands for the directional derivative in the direction $h$ given through Theorem \ref{thm:dirDiff1}.
%\subsection{Characterisation of the optimal control}
\begin{lem}\label{lem:min1}%Let the assumptions of  Theorem \ref{thm:dirDiff1} hold and 
Denote by 
$j\colon H \to V^*$ the inclusion map. Then $0 \in V^*$ is a minimiser of the problem 
\begin{equation}\label{eq:min2}
\min_{h \in \cl{j\mathcal{T}_{U_{ad}}(u^*)}^{V^*}}(\alpha_h,y^*- y_d)_H + \nu \langle h, u^*\rangle.
\end{equation}
\end{lem}
\begin{proof}
Choosing the direction $h=0$ in the inequality of Proposition \ref{lem:characterisationOfOC} implies
$0 \leq (\alpha_0, y^*- y_d) + \nu (u^*, 0) = 0$ 
with the equality because $\alpha_0 = 0$. Hence $h=0$ is a minimiser of 
\begin{equation*}
\min_{h \in \mathcal{T}_{U_{ad}}(u^*)}(\alpha_h,y^*- y_d)_H + \nu (u^*, h).
\end{equation*}
%Since $(y^*, u^*)$ is the minimiser, the zero perturbation is the optimal one and so $h=0$ is a minimiser of the above problem.
As in Lemma 4.1 of \cite{Wachsmuth}, the feasible set can be enlarged (the continuity in $V^*$ of $h \mapsto \alpha_h$ assured by Proposition \ref{lem:alphaUniquenessAndCty} is needed here) to obtain the desired result.
\end{proof}
The aim now is to rewrite \eqref{eq:min2} over the space
\[W := \{ v \in V : v=0 \text{ q.e. in $\mathcal{A}_s(y^*)$}\}.\]
%First, let us consider the directional derivative $\alpha_h$ appearing in the objective functional of \eqref{eq:min2}. 
Using the characterisation of  the critical cone from \cite[Lemma 3.1]{Wachsmuth}, we see that $\mathcal{K}^{y^*} \subset W$. Denote by  $i\colon W \to V$ the inclusion map %; its adjoint $i^*\colon V^* \to W^*$ has the action of restricting the domain of elements of $V^*$ from $V$ to the subset $W$. 
and define the closed convex set
\[\mathcal{C}_{W}^{y^*}:= \{ v \in W : v \leq 0 \text{ q.e. in $\mathcal{B}(y^*)$}\},\] which satisfies $\mathcal{K}^{y^*} = i\mathcal{C}_{W}^{y^*}$.   Now, note that, using \eqref{ass:PhiDerivativeLipschitzConstant},
$(\Id-\Phi'(y^*))\colon V \to V$ is invertible. 
%Define $A_W := i^*A(\Id-\Phi'(y^*))^{-1}i$ and consider now, for any $d \in V^*$,
%\[u \in C_W : \langle A_W u - i^*d, u- w\rangle_{W^*,W} \leq 0 \quad \forall w\in C_W\] 
Define 
\[A_W\colon W \to W^*, \qquad A_W := i^*A(\Id-\Phi'(y^*))^{-1}i\] and observe that for any $\tilde d \in W^*$ the inequality
\[\delta \in \mathcal{C}_{W}^{y^*} : \langle A_W \delta - \tilde d, \delta- w\rangle_{W^*,W} \leq 0 \quad \forall w\in \mathcal{C}_{W}^{y^*}\] 
has a unique solution by the Lions--Stampacchia theorem since $A_W$ is bounded and coercive due to the Lipschitz condition \eqref{ass:PhiDerivativeLipschitzConstant} (see \cite[Lemma 3.3]{MR4083198}). Now suppose that for $d \in V^*$, $\delta$ solves
 \[\delta \in \mathcal{C}_{W}^{y^*} : \langle A_W \delta - i^*d, \delta- w\rangle_{W^*,W} \leq 0 \quad \forall w\in \mathcal{C}_{W}^{y^*}.\] 
Consider also 
\[z \in \mathcal{K}^{y^*} : \langle A(\Id-\Phi'(y^*))^{-1}z - d, z-v\rangle_{V^*,V} \leq 0 \quad \forall v \in K^{y^*}.\]
Then it is easy to see that that $z=i\delta$. 
%Indeed, we see that
%\begin{align*}
%0 &\geq \langle i^*A(\Id-\Phi'(y^*))^{-1}i \delta - i^*d, \delta- w\rangle\\
%&=\langle A(\Id-\Phi'(y^*))^{-1}i \delta - d, i\delta-iw\rangle  \\
%&=\langle A(\Id-\Phi'(y^*))^{-1}z - d, z- iw\rangle  \\
%&=\langle A(\Id-\Phi'(y^*))^{-1}z - d, z-\varphi\rangle \quad\text{where $\varphi = iw \in \mathcal{K}^{y^*}$}.
%\end{align*}
%By defining $i\beta_h := \alpha_h - \Phi'(y^*)(\alpha_h),$
% the QVI \eqref{eq:QVIforAlpha} satisfied by $\alpha_h$ can be written as %$\beta_h \in \mathcal{C}_{W}^{y^*}$ such that $\alpha_h - \Phi'(y^*)(\alpha_h) = i\beta_h$ and 
%\begin{equation}\label{eq:prelim3}
%\beta_h \in \mathcal{C}_{W}^{y^*} : \langle A(i\beta_h + \Phi'(y^*)(\alpha_h)) - h, i(\beta_h - \varphi) \rangle \leq 0 \quad \forall \varphi \in \mathcal{C}_{W}^{y^*}.
%\end{equation}

\begin{lem}
Define the operator %$\theta\colon W \to V$, $\zeta \colon W \to W^*$ and $A_{W}\colon W \to W^*$ by % and $T_{y^*}\colon W \to W^*$ by
%\begin{alignat*}{3}
$\theta \colon W \to V$ by 
\[\theta := (\Id-\Phi'(y^*))^{-1}i.\] 
%\\
%\zeta&\colon W \to W^*,  &\zeta&:= i^*A\Phi'(y^*)(\theta),\\
%A_{W}&\colon W \to W^*,   &A_{W} &:= i^*A i.
%T_{y^*} &:= i^*A\Phi'(y^*)((1-\Phi'(y^*))^{-1}i).
%\end{alignat*}
Then $(0,0)$ is a solution of
\begin{equation}\label{eq:compEqn}
\begin{aligned}
\min_{(\beta_h, h) \in W\times W^* }&(\theta(\beta_h),y^*- y_d)_H + \nu \langle h, u^*\rangle_{W^*, W} \text{ s.t. }\\
&\begin{cases}
\beta_h \in \mathcal{C}_{W}^{y^*}\\
h = A_{W}\beta_h\\
h \in \cl{i^*j\mathcal{T}_{U_{ad}}(u^*)}^{W^*}.
\end{cases}
\end{aligned}
\end{equation}
\end{lem}
\begin{proof}
By defining $\gamma_h := \alpha_h - \Phi'(y^*)(\alpha_h) =(\Id-\Phi'(y^*))\alpha_h,$
 the QVI \eqref{eq:QVIforAlpha} satisfied by $\alpha_h$
%\[\langle A\alpha_h -h, \alpha_h - v \rangle \leq 0\] 
%\[\langle A\alpha_h -h, \alpha_h - \Phi'(y^*)(\alpha_h) - \varphi \rangle \leq 0 \quad \forall \varphi \in \mathcal{K}^{y^*}\] 
  can be written as 
%\begin{equation}\label{eq:prelim3n}
%\langle A(\Id-\Phi'(y^*))^{-1}\gamma_h  - h, \gamma_h  + \Phi'(y^*)(\alpha_h)- v\rangle \leq 0 \quad  \forall v \in \mathcal{K}(\alpha)
%\end{equation}
%which gives after transforming the test function as
\begin{equation*}%\label{eq:prelim3nn}
\gamma_h \in \mathcal{K}^{y^*} : \langle A(\Id-\Phi'(y^*))^{-1}\gamma_h  - h, \gamma_h  - \varphi \rangle \leq 0 \quad \forall \varphi \in \mathcal{K}^{y^*}.
\end{equation*}
Now if $\beta_h$ satisfies
\[\beta_h \in \mathcal{C}_{W}^{y^*} : \langle A_W\beta_h - i^*h, \beta_h - \varphi \rangle \leq 0 \quad \forall \varphi \in \mathcal{C}_{W}^{y^*},\]
we have (as discussed above) $\gamma_h = i\beta_h$, hence
\[i\beta_h = (\Id-\Phi'(y^*))\alpha_h \quad \iff \quad \alpha_h =\theta(\beta_h)%= (\Id-\Phi'(y^*))^{-1}(i\beta_h)
\]
%Since $\alpha_h = \theta(\beta_h)$,  using \eqref{eq:prelim3} and the definition of $\zeta$, %as%above is then %and defining the elliptic operator
%\[A_{W} := i^*\circ A\circ i\colon W \to W^*.\]
%\begin{equation}\label{eq:betaHVI}
%\beta_h \in \mathcal{C}_{W}^{y^*}: \langle A_{W}\beta_h +\zeta(\beta_h) - i^*h, \beta_h - \varphi \rangle_{W^*_{y^*}	, W} \leq 0 \quad \forall \varphi \in \mathcal{C}_{W}^{y^*}
%\end{equation}
%and we can rewrite the minimisation problem \eqref{eq:min1} as
%%\[\min_{h \in i^*j\mathcal{T}_{U_{ad}}(u^*)}(\theta\beta_h,y^*- y_d) + \nu \langle h, u^*\rangle_{W^*, W}.\]
%Arguing as before,   
Therefore,  \eqref{eq:min2} can be restated and we get (using the continuity of $\Phi'(y^*)$) that $0$ is a solution of
\begin{equation*}\label{eq:betaHVI}
\begin{aligned}
\min_{h \in \cl{i^*j\mathcal{T}_{U_{ad}}(u^*)}^{W^*}}&(\theta(\beta_h),y^*- y_d)_H + \nu \langle h, u^*\rangle_{W^*, W} \text{ s.t. }\\
&\beta_h \in \mathcal{C}_{W}^{y^*}: \langle A_{W}\beta_h   - h, \beta_h - \varphi \rangle_{W^*_{y^*}	, W} \leq 0 \quad \forall \varphi \in \mathcal{C}_{W}^{y^*};
\end{aligned}
\end{equation*}
this is well defined because $u^* \in W$ due to \eqref{eqass3}. Hence, similarly to Proposition \ref{lem:complementarityCharacterisationDerivativeA},  $(0,0,0)$ is a solution of
\begin{align*}
\min_{(\beta_h, h, \xi_h) \in W\times W^* \times W^*}&(\theta(\beta_h),y^*- y_d)_H + \nu \langle h, u^*\rangle_{W^*, W} \text{ s.t. }\\
&\begin{cases}
\beta_h \in \mathcal{C}_{W}^{y^*}\\
\xi_h = h-A_{W}\beta_h\\
\xi_h \in (\mathcal{C}_{W}^{y^*})^\circ\\
\langle \xi_h, \beta_h \rangle = 0\\
h \in \cl{i^*j\mathcal{T}_{U_{ad}}(u^*)}^{W^*}.
\end{cases}
\end{align*}
Setting $\xi_h = 0$ leads to the result.%, then $(\beta_h, h) = (0,0)$ is a solution of \eqref{eq:compEqn}.
\end{proof}
We need to derive stationarity conditions for this problem and then transform the resulting system back to the original spaces and operators. Let us remark that under the assumptions of the theorem, we have that
%\begin{align*}
%&\text{
%$\zeta$ is continuously Fr\'echet differentiable and %\\%\label{ass:zetaC1},\\
$\theta$ is linear and bounded.%\label{ass:invIsLinear}
%\end{align*}
%\subsection{Checking constraint qualification}
%\subsection{Conclusion}
\begin{lem}\label{lem:ZKforSS}%Let %and %let $\Phi'(y^*)\colon V \to V$ be linear and 
Defining 
\begin{align*}
D:=\cl{i^*j\mathcal{T}_{U_{ad}}(u^*)}^{W^*}, \qquad \mathcal Y := W^*\times W \times W^*, \qquad C:=(\{0\}, \mathcal{C}_{W}^{y^*}, D),
\end{align*}
there exists $(\tilde p, \tilde \lambda, \sigma) \in \mathcal Y^* \cap C^\circ$ such that
\begin{alignat*}{3}
A_{W}^*\tilde p + \theta^*(j(y^*- y_d)) + \tilde \lambda &= 0,\\
\nu u^* - \tilde p + \sigma  &= 0,\\
\tilde \lambda &\in (\mathcal{C}_{W}^{y^*})^\circ,\\
\sigma &\in D^\circ. %&\equiv (i^*j\mathcal{T}_{U_{ad}}(u^*))^\circ.
%\sigma &\in  (i^*j\mathcal{T}_{U_{ad}}(u^*))^\circ.
\end{alignat*}
%, that is, 
\end{lem}
\begin{proof}
In addition to the notation introduced above, let us also define the space $\mathcal X := W \times W^*$.
%\begin{align*}
%D:=\cl{i^*j\mathcal{T}_{U_{ad}}(u^*)}^{W^*}, \qquad \mathcal X := W \times W^*, \qquad \mathcal Y := W^*\times W \times W^*, \qquad C:=(\{0\}, \mathcal{C}_{W}^{y^*}, D).
%\end{align*}
Define the map $g \colon \mathcal X \to \mathcal Y$ by $g(\beta, h) := (A_{W}\beta  - h, \beta, h)$ and observe that \eqref{eq:compEqn}  can be compactly written as
\begin{align}\label{eq:finalMinProblem}
\min_{g(\beta_h, h) \in C} (\theta(\beta_h),y^*- y_d) + \nu \langle h, u^*\rangle_{W^*, W}.
\end{align}
We now proceed with checking the Zowe--Kurcyusz constraint qualification $g'((0,0))\mathcal X - \mathcal{R}_C(g(0,0)) = \mathcal Y$ to deduce the existence of Lagrange multipliers. First observe that $D$ is a convex cone which in turn implies that $C$ is a convex cone and then by \cite[Example 2.62]{MR1756264}, $\mathcal{R}_C((0,0,0)) = C$ and $\mathcal{T}_C((0,0,0))^\circ = C^\circ$. Now, we see that $g(0,0) = (0, 0, 0)$ and $\mathcal{R}_C(g(0,0)) = C$. We also have
\[g'(0,0)(\gamma, d) =  (A_{W}\gamma  - d, \gamma, d) \qquad \forall (\gamma,d) \in W\times W^*.\]
%The derivative of $\Phi$ is positively homogeneous with respect to the direction, which implies that $\zeta'(0)(\gamma) = \zeta(\gamma).$ Since $\zeta$ is Fr\`echet differentiable at $0$, positively homogeneous and vanishes at the origin, it follows that $\zeta$ is linear. 
Therefore, we are required to show that for every $(w_1^*, w_2, w_3^*) \in \mathcal{Y}$, there exist $(\gamma, d,v,h) \in \mathcal X \times \mathcal{C}_{W}^{y^*} \times D$ such that
\begin{equation}\label{eq:ssSystemPrelim}
\begin{aligned}
A_{W}\gamma  - d &= w_1^*,\\
\gamma - v &= w_2,\\
d - h&= w_3^*.
\end{aligned}
\end{equation}
The first equation written in terms of $v$ and $h$ %it reads $A_{W}v + A_{W}w_2 + T(v + w_2) - g - w_3^* = w_1^*$
%i.e.,
reads $A_{W}v  - (w_1^* + w_3^*-A_{W}w_2) = h.$ In order to force solutions to belong to the desired sets, we consider the VI 
\begin{align}
\text{find } v \in \mathcal{C}_{W}^{y^*} : \langle A_{W}v  - (w_1^* + w_3^*-A_{W}w_2 ), v - \varphi \rangle \leq 0 \quad \forall \varphi \in \mathcal{C}_{W}^{y^*}\label{eq:weirdVI}
\end{align}
associated to the above PDE. %That is, we are required to show that
%\begin{equation}\label{eq:ssPrelim2}
%\forall h \in W^*, \exists  v \in \mathcal{C}_{W}^{y^*} : \langle A_{W}v + \zeta v -h, v - \varphi \rangle \leq 0 \quad \forall \varphi \in \mathcal{C}_{W}^{y^*}.
%\end{equation}
%We may do this as follows. 
%Apply Theorem \ref{thm:dirDiff1} with $y^* \in \mathbf{Q}(u^*)$ chosen as the base point and a direction $\tilde h \in V^*$: we get the existence of  a directional derivative $\alpha^*$; %satisfying
%%\[\alpha^* \in \mathcal{K}^{y^*}(\alpha^*) : \langle A\alpha^* - \tilde h, \alpha^* - \varphi \rangle \leq 0 \quad \forall \varphi \in \mathcal{K}^{y^*}(\alpha^*).\]
%define $\beta^*:=\alpha^*-\Phi'(y^*)(\alpha^*)  = \alpha^*-\Phi'(y^*)(1-\Phi'(y^*))^{-1}(\beta^*)$ which satisfies
%\[\beta^* \in \mathcal{K}^{y^*} : \langle A\beta^* + A\Phi'(y^*)(1-\Phi'(y^*))^{-1}\beta^* - \tilde h, \beta^* - \varphi \rangle \leq 0 \quad \forall \varphi \in \mathcal{K}^{y^*}.\]
%Writing $v:=i^{-1}\beta^*$, $\psi = i^{-1}\varphi$, the above can be rewritten as
%\[iv \in \mathcal{K}^{y^*} : \langle Aiv + A\Phi'(y^*)(1-\Phi'(y^*))^{-1}iv - \tilde h, iv - i\psi \rangle \leq 0 \quad \forall i\psi \in \mathcal{K}^{y^*}.\]
%%which is exactly \eqref{eq:ssPrelim2} with $h=i^*\tilde h$.
%%\[v \in \mathcal{C}_{W}^{y^*} : \langle A_Wv + \zeta v - i^*\tilde h, v - \psi \rangle \leq 0 \quad \forall \psi \in \mathcal{C}_{W}^{y^*}\]
As explained above, \eqref{eq:weirdVI} has a solution and furthermore, the following complementarity system (which can be derived by the same arguments as before) is satisfied by any solution:
\begin{align*}
\begin{cases}
v \in \mathcal{C}_{W}^{y^*}\\
\eta := (w_1^* + w_3^*-A_{W}w_2) - A_{W}v \\
\eta \in (\mathcal{C}_{W}^{y^*})^\circ \\
\eta \perp v.
\end{cases}
\end{align*}
Using this, we see that $h :=-\eta \in -(\mathcal{C}_{W}^{y^*})^\circ$.  The manipulations in the paragraph after Lemma 5.1 of \cite{Wachsmuth}  show that $(i^*j\mathcal{T}_{U_{ad}}(y^*))^\circ \subset -\mathcal{C}_{W}^{y^*}$ which implies that $-(\mathcal{C}_{W}^{y^*})^\circ \subset (i^*j\mathcal T_{U_{ad}}(y^*))^{\circ\circ} = D$, that is, $g \in D$. Then we simply define $\gamma$ and $d$ by \eqref{eq:ssSystemPrelim}.  %set $\gamma := v+w_2 \in W$ and $d:= h + w_3^* \in D + W^* \subset W^*$. %Thus the constraint qualification is met.	
Thus the constraint qualification is met for \eqref{eq:finalMinProblem}.

Writing the objective functional in \eqref{eq:finalMinProblem} as $\hat J$, we obtain the existence of a Lagrange multipler $(\tilde p, \tilde \lambda, \sigma) \in \mathcal Y^* \cap C^\circ$ such that 
\[\hat J'(0,0)(x) + \langle g'(0,0)^*(\tilde p,\tilde \lambda, \sigma), x \rangle = 0 \quad \forall x \in \mathcal X.\]
With $x=(\gamma, d)$, we see that since $\theta(0) = 0$, the first term above is
\begin{align*}
\hat J'(0,0)(x) %&= \lim_t \frac{\hat J(t\gamma, td)-\hat J(0,0)}{t}\\
%&= \lim_{t\to 0} \frac{(\theta(t\gamma), y^* - y_d)_H + \nu \langle td, u^* \rangle_{W^*, W}}{t}\\
%&= (\theta(\gamma), y^* - y_d)_H + \nu \langle d, u^* \rangle_{W^*, W}\\
%&= \langle j(y^* - y_d), \theta(\gamma) \rangle + \nu \langle d, u^* \rangle_{W^*, W}\\
&= \langle \theta^*(j(y^* - y_d)), \gamma\rangle_{W^*,W} + \nu \langle d, u^* \rangle_{W^*, W},
%&= (\gamma,\theta^A(y^* - y_d))_H + \nu \langle d, u^* \rangle\\
%&= \langle (\theta^*j(y^*- y_d), \nu u^*), x \rangle
\end{align*}
where $\theta^*\colon V^* \to W^*$ is the adjoint of  $\theta\colon W \to V$ (this exists due to the linearity assumption). We also have, by definition of the adjoint operator,
\begin{align*}
\langle g'(0,0)^*(\tilde p,\tilde \lambda, \sigma), x \rangle %&:= \langle (\tilde p,\tilde \lambda, \sigma), g'(0,0)x \rangle\\
&= \langle (\tilde p,\tilde \lambda, \sigma), (A_{W}\gamma   -d, \gamma, d) \rangle\\
%&= \langle \tilde p, A_{W}\gamma + \zeta\gamma -d\rangle_{W^*, W} + \langle \tilde \lambda, \gamma \rangle_{W^*, W} + \langle \sigma, d \rangle_{W^*, W} \\
&= \langle A_{W}^*\tilde p, \gamma \rangle_{W^*, W} +  \langle \tilde \lambda, \gamma \rangle_{W^*, W} + \langle \sigma-\tilde p, d \rangle_{W^*, W}. %\\
%&= \langle (A_{W} + T)^*\tilde p + \tilde \lambda, \gamma \rangle + \langle \nu - \tilde p, d \rangle\\
%&= \langle ((A_{W} + T)^*\tilde p + \tilde \lambda, \nu-\tilde p), (\gamma, d) \rangle \\
%&= \langle ((A_{W} + T)^*\tilde p + \tilde \lambda, \nu-\tilde p), x\rangle.
\end{align*}
%Hence %recalling again that $\mathcal{T}_C(0,0,0) = C$, we find that $(\tilde p, \tilde \lambda, \sigma)  \in C^\circ$ satisfies
This implies the result.
\end{proof}
%\begin{remark}We have decided to simply assume \eqref{ass:VIhasSolution} instead of deriving conditions on the operators that imply existence in order to leave the theory as general as possible. %It may also be possible to weaken the assumed linearity of $\Phi'(y^*)\colon V \to V$.
%\end{remark}
%\subsection{Conclusion}

We now transform all quantities back to the space $V$. 
\begin{proof}[Conclusion of sketch proof of Theorem \ref{thm:strongStationarity}]
Observe that under the assumptions, Proposition \ref{lem:characterisationOfOC}, Lemma \ref{lem:ZKforSS} and Theorem \ref{thm:dirDiff1} are applicable. To start with, let us define 
%\[\lambda^* := (i^*)^{-1}\tilde \lambda = -A^*ip^* - (i^*)^{-1}\zeta^*_{y^*}p^* -(i^*)^{-1}\theta^*j(y-y_d),\]
\[p^* := i\tilde p\]
and
\[\lambda^* := (\Id-\Phi'(y^*)^*)^{-1}(-A^*i\tilde p - j(y-y_d)),\]
and for convenience, denote $L := \Phi'(y^*).$  %ince $L$ is a a bounded linear operator, the assumptions of the  previous lemma are fulfilled.
%We follow the proof of \cite[Theorem 5.2]{Wachsmuth}.
\begin{itemize}[leftmargin=*]
\item  
%Using the commutation of inverses and adjoints of bounded linear operators, we find $(i^*)^{-1}\theta^* = (\Id-L^*)^{-1}$. Likewise, $(i^*)^{-1}\zeta^* = (\Id-L^*)^{-1}L^*A^*i$  which implies that the equality for $\lambda^*$ can be written as
%%\[A^*i\tilde p + (\Id-L^*)^{-1}L^*A^*i\tilde p + (\Id-L^*)^{-1}j(y-y_d) + \lambda^* = 0\]
%%which is
%\[(\Id + (\Id-L^*)^{-1}L^*)A^*i\tilde p + (\Id-L^*)^{-1}j(y-y_d) + \lambda^* = 0\]
%and using $\Id + (\Id-L^*)^{-1}L^*  = (\Id-L^*)^{-1}(\Id-L^*) + (\Id-L^*)^{-1}L^*=(\Id-L^*)^{-1}$, 
By definition of $\lambda^*$ and $p^*$, we get the first line in the system after etching away the inclusion map $j$.

\item We see from the definition of $\lambda^*$ and elementary manipulations to relate it to $\tilde \lambda \in (C_W)^\circ$ and the usage of the fact that $iC_W = \mathcal{K}^{y^*}$ that $\lambda^* \in (\mathcal{K}^{y^*})^\circ$. This implies the final condition of the system thanks to \cite[Lemma  3.1]{Wachsmuth}.

\item Since $\tilde p \in W$, it vanishes q.e. on the strongly active set. As $\tilde p = \nu u^* + \sigma$ and since $\sigma \in D^\circ$, Lemma 5.1 of \cite{Wachsmuth} tells us that $\sigma \geq 0$ q.e. on $\Omega \setminus U_a$. Thus 
\[\sigma|_{\mathcal{B}(y^*)} = \sigma|_{U_a \cap \mathcal{B}(y^*)} + \sigma|_{(\Omega \setminus U_a) \cap \mathcal{B}(y^*)} \geq \sigma|_{U_a \cap \mathcal{B}(y^*)}  = 0\] with the final equality because of \eqref{eqass1}.  Note also that 
\begin{align*}
u^*|_{\mathcal B(y^*)} = u^*|_{\mathcal B(y^*) \cap U_b} + u^*|_{\mathcal B(y^*) \cap (\Omega \setminus U_b)} \geq u^*|_{\mathcal B(y^*) \cap (\Omega \setminus U_b)} \geq 0 \text{ q.e.,}
\end{align*}
with the first inequality by \eqref{eqass2} and the final inequality by the third sign condition on $u^*$ stated in \S \ref{sec:ss}. This implies  the stated condition on $p^*$, which 
%\begin{align*}
%%A^*p + (i^*)^{-1}\zeta^*p + (i^*)^{-1}\theta^*j(y-y_d) + \lambda^* = 0\\
%\tilde p \geq 0 \text{ q.e. on $\mathcal{B}(y^*)$ and } \tilde p=0 \text{ q.e. on $\mathcal{A}_s(y^*)$}
%\end{align*}
is equivalent to  $-p^* \in \mathcal{K}^{y^*}$ due to the characterisation of the critical cone in \cite[Lemma 3.1]{Wachsmuth}. %Take a function $v \in W$ with $v \leq 0$ q.e. on $\mathcal{B}(y^*)$.  Then we see that, using $v = iv$, $\tilde p=i\tilde p$, and $A^*_{W} = i^*A^*i$
%\begin{align*}
%\langle \lambda, v \rangle_{V^*,V} &=  \langle -A^*\tilde p - (i^*)^{-1}\zeta^*p^* - (i^*)^{-1}\theta^*j(y-y_d), v \rangle_{V^*,V}\\
%&=  \langle -i^*A^*p^* -\zeta^*p^* - \theta^*j(y-y_d), v \rangle_{W^*,W}\\
%&=  \langle -A^*_{W}p^* -\zeta^*p^* - \theta^*j(y-y_d), v \rangle_{W^*,W}\\
%&= \langle \tilde \lambda, v \rangle_{W^*, W}\\
%&\leq 0
%\end{align*}
%since $\tilde \lambda$ is in the p^*olar cone of $\mathcal{K}_W(\bar y)$. 
\item We obtain $\sigma \in \mathcal{N}_{U_{ad}}(u^*)$ exactly as in the proof of Theorem 5.2 in \cite{Wachsmuth}\footnote{In \cite{Wachsmuth}, the notation $\mu$ is used instead of $\sigma$.} (where $\mathcal{N}_{U_{ad}}$ denotes the normal cone to $U_{ad}$ with respect to $H$), which is the polar cone of the tangent cone, see \cite[\S 2.2.4]{MR1756264}) and this is precisely the desired inequality constraint relating the control and the adjoint.
\qedhere
\end{itemize}
\end{proof}
\bibliographystyle{abbrv}
\bibliography{QVILatestBibEvenLater}
\end{document}